\documentclass{article}
\usepackage[truedimen,left=25truemm,right=25truemm,top=30truemm,bottom=30truemm]{geometry}
\usepackage{amsmath,amssymb}
\usepackage{graphicx}
\usepackage{tikz}
\usepackage{booktabs}
\usepackage{pinlabel}
\usepackage{ytableau}
\usepackage{braket}

\usetikzlibrary{arrows.meta, angles, quotes, calc}
\usetikzlibrary{positioning}
\usepackage{amsthm}
\usepackage{comment}
\usepackage{tikz-3dplot}
\usepackage{ulem}
\usepackage[absolute,overlay]{textpos}
\usepackage[legacycolonsymbols]{mathtools}

\theoremstyle{definition}
\newtheorem{theorem}{Theorem}
\newtheorem*{theorem*}{Theorem}

\newtheorem{definition}[theorem]{Definition}
\newtheorem*{definition*}{Definition}

\newtheorem{prop}[theorem]{Proposition}
\newtheorem{lemma}[theorem]{Lemma}
\newtheorem{coro}[theorem]{Corollary}
\newtheorem{example}[theorem]{Example}
\newtheorem{rem}[theorem]{Remark}
\newtheorem{notation}[theorem]{Notation}

\numberwithin{theorem}{section}
\numberwithin{equation}{section}

\DeclareFontFamily{U}{mathx}{}
\DeclareFontShape{U}{mathx}{m}{n}{<-> mathx10}{}
\DeclareSymbolFont{mathx}{U}{mathx}{m}{n}
\DeclareMathAccent{\widehat}{0}{mathx}{"70}
\DeclareMathAccent{\widecheck}{0}{mathx}{"71}

\DeclareMathSymbol{\antishriek}{\mathord}{operators}{'74}

\newcommand{\graphg}[0]{
\begin{tikzpicture}[x=1pt,y=1pt,yscale=-0.2,xscale=0.4,baseline=10pt, line width = 0.7pt]

\draw  [fill= {rgb, 255:red, 0; green, 0; blue, 0 }  ,fill opacity=1 ] (0, -170) circle (5);
\draw  [fill= {rgb, 255:red, 0; green, 0; blue, 0 }  ,fill opacity=1 ] (0, 75) circle (5);

\draw  (-50, -50) circle (5);
\draw  (0, -100) circle (5);
\draw  (50, -50) circle (5);
\draw  (0, 0) circle (5);

\draw  [dash pattern={on 4pt off 3pt}]  (0,75) -- (0,5);
\draw [dash pattern={on 4pt off 3pt}]  (45,-50) -- (-45,-50);
\draw  [dash pattern={on 4pt off 3pt}]  (0,-170) -- (0,-105);

\draw   [dash pattern={on 4pt off 3pt}]  (-45,-50) -- (-5,-95);
\draw [dash pattern={on 4pt off 3pt}]  (5,-95) -- (45,-50);
\draw  [dash pattern={on 4pt off 3pt}]  (45,-45) -- (5,0);
\draw  [dash pattern={on 4pt off 3pt}]  (-5,0) -- (-45,-45);

\end{tikzpicture}}

\newcommand{\chorddiagramb}[0]{
\begin{tikzpicture}[x=1pt,y=1pt,yscale=0.2,xscale=0.3,baseline=20pt, line width = 1pt]
\draw [color={rgb, 255:red, 0; green, 0; blue, 255 }, line width=1pt]  [-Stealth] (0,-100)--(0, 400);
\draw [color={rgb, 255:red, 0; green, 0; blue, 255 }, line width=1pt]  [-Stealth] (200,-100)--(200, 400);
\draw [color={rgb, 255:red, 0; green, 0; blue,0 }, line width=1pt]   [Stealth-, dash pattern = on 3pt off 3pt]  (0,300)--(200,0);
\node at (160,300) {$(1)$};
\draw [color={rgb, 255:red, 0; green, 0; blue,0 }, line width=1pt]   [-Stealth, dash pattern = on 3pt off 3pt]  (0,0)--(200,300) ;
\node at (40,300) {$(3)$};
\draw [color={rgb, 255:red, 0; green, 0; blue,0 }, line width=1pt]   [-Stealth, dash pattern = on 3pt off 3pt]  (0,150)--(200,150) ;
\node at (40,180) {$(2)$};

\draw  [fill={rgb, 255:red, 0; green, 0; blue, 0 }  ,fill opacity=1 ] (0,0) circle (5) ;
\node at (-20,0) {$O$};
\draw  [fill={rgb, 255:red, 0; green, 0; blue, 0 }  ,fill opacity=1 ] (0,150) circle (5) ;
\draw [fill={rgb, 255:red, 0; green, 0; blue, 0 }  ,fill opacity=1 ] (0,300) circle (5);

\draw  [fill={rgb, 255:red, 0; green, 0; blue, 0 }  ,fill opacity=1 ] (200,0) circle (5) ;
\node at (220,0) {$O$};
\draw  [fill={rgb, 255:red, 0; green, 0; blue, 0 }  ,fill opacity=1 ] (200,150) circle (5) ;
\draw [fill={rgb, 255:red, 0; green, 0; blue, 0 }  ,fill opacity=1 ] (200,300) circle (5);

\end{tikzpicture}}

\newcommand{\segmenta}[0]{
\begin{tikzpicture}[x=1pt,y=1pt,yscale=0.15,xscale=0.15, baseline=-3pt] 

\draw  [-Stealth, line width =1pt, color ={rgb, 255:red, 0; green, 0; blue, 225 } ] (0,0)--(500,0);

\draw  [dash pattern = on 2pt off 3 pt, line width =1pt] (100,0)--(100,150);
\draw  [dash pattern = on 2pt off 3 pt, line width =1pt] (300,0)--(300,150);

\draw  [fill={rgb, 255:red, 0; green, 0; blue, 0 }  ,fill opacity=1 ]  (100,0) circle (20);
\draw  [fill={rgb, 255:red, 0; green, 0; blue, 0 }  ,fill opacity=1 ] (300,0) circle (20);
\end{tikzpicture}}

\newcommand{\segmentb}[0]{
\begin{tikzpicture}[x=1pt,y=1pt,yscale=0.15,xscale=0.15, baseline=-3pt] 

\draw  [-Stealth, line width =1pt, color ={rgb, 255:red, 0; green, 0; blue, 225 }] (0,0)--(700,0);

\draw  [dash pattern = on 2pt off 3 pt, line width =1pt] (100,0)--(100,150);
\draw  [dash pattern = on 2pt off 3 pt, line width =1pt] (300,0)--(300,150);
\draw  [dash pattern = on 2pt off 3 pt, line width =1pt] (500,0)--(500,150);

\draw  [fill={rgb, 255:red, 0; green, 0; blue, 0 }  ,fill opacity=1 ]  (100,0) circle (20);
\draw  [fill={rgb, 255:red, 0; green, 0; blue, 0 }  ,fill opacity=1 ] (300,0) circle (20);
\draw  [fill={rgb, 255:red, 0; green, 0; blue, 0 }  ,fill opacity=1 ] (500,0) circle (20);

\end{tikzpicture}}

 \newcommand{\ctext}[1]{\raise0ex\hbox{\textcircled{\scriptsize{#1}}}}

\begin{document} 
\title{Some non-trivial cycles of the space of long embeddings detected by configuration space integral invariants using $g$-loop $(g =2, 3)$ graphs}
\author{Leo Yoshioka\thanks{Graduate School of Mathematical Sciences, The University of Tokyo\newline\qquad e-mail:yoshioka@ms.u-tokyo.ac.jp}}
\maketitle

\begin{abstract}
In this paper, we give some non-trivial geometric cycles of the space of long embeddings $\mathbb{R}^j \rightarrow \mathbb{R}^n$ $(n-j \geq 2)$ modulo immersions.
We construct a class of cycles from specific chord diagrams associated with $2$-loop or $3$-loop hairy graphs. 
To detect these cycles, we use cocycles obtained by the $2$-loop or $3$-loop part of modified configuration space integrals using a modified Bott-Cattaneo-Rossi graph complex.
We show the non-triviality of the cycles by pairing argument, which is reduced to pairing of graphs with the chord diagrams. 
As a corollary of the $2$-loop part, we provide an alternative proof of the non-finite generation of the $(j-1)$-th rational homotopy group of the space of long embeddings of codimension two, which Budney--Gabai and Watanabe first established. We also show the non-finite generation of the $2(j-1)$-th homotopy group by using the $3$-loop part. 
\end{abstract}

\part*{Introduction}
\addcontentsline{toc}{part}{Introduction}

A \textit{long embedding} is an embedding of $\mathbb{R}^j$ into $\mathbb{R}^n$  that coincides with the standard linear embedding outside a ball in $\mathbb{R}^j$. The difference $n-j$ of the dimensions is called the \textit{codimension}. A \textit{long $j$-knot} is a long embedding $\mathbb{R}^{j} \rightarrow \mathbb{R}^{j+2}$ of codimension two. We write  $\mathcal{K}_{n,j} = \text{Emb} (\mathbb{R}^j, \mathbb{R}^n) $ for the space of long embeddings equipped with the usual $C^{\infty}$ topology.  
Since embeddings are immersions, there is a map $\mathcal{K}_{n,j} \rightarrow \text{Imm} (\mathbb{R}^j, \mathbb{R}^n)$ to the space of \textit{long immersions}.
We write $\overline{\mathcal{K}}_{n,j} = \overline{\text{Emb}} (\mathbb{R}^j, \mathbb{R}^n)$ for the homotopy fiber of this map at the standard immersion. This space is called \textit{the space of long embeddings modulo immersions}. 

In 2017, Fresse, Turchin and Willwacher \cite{FTW 1, FTW 2}, following Arone and Turchin \cite{AT 1, AT 2}, showed the surprising result that the homology of a graph complex, called the \textit{hairy graph complex}, has the full dimensional information of the rational homotopy group of $\overline{\mathcal{K}}_{n,j}$, when the codimension is greater than or equal to three. 
Their result is based on a deep homotopy theory:  the \textit{Taylor approximation} $\overline{\mathcal{K}}_{n,j} \rightarrow T_{\infty} \overline{\mathcal{K}}_{n,j}$ by the Goodwillie-Klein-Weiss \textit{embedding calculus} \cite{GKW, GW, Wei}, which is weakly equivalent when $n-j \geq 3$. The hairy graph complex is introduced by applying the rational homotopy theory to the \textit{derived mapping space models} of $T_{\infty} \overline{\mathcal{K}}_{n,j}$ \cite{AT 1, BW}. 

On the other hand, from the 1990s to 2010s, Bott \cite{Bot}, Cattaneo, Rossi \cite{CR}, Sakai \cite{Sak} and Watanabe \cite{SW, Wat 2} developed a geometric approach which works well even when $n-j = 2$. They introduced another graph complex $GC^{BCR}_{n,j}$, which we call the \textit{BCR graph complex}. 
They gave a linear map from the BCR graph complex to the de Rham complex of $\mathcal{K}_{n,j}$ (and of $\overline{\mathcal{K}}_{n,j}$)  through Bott-Taubes type \textit{configuration space integrals} \cite{BT};
\[
I: GC^{BCR}_{n,j}  \longrightarrow \Omega_{dR} \mathcal{K}_{n,j} \quad \Gamma \longmapsto \int_{\text{Conf}_{\Gamma}} \omega(\Gamma). 
\]
Unfortunately, it is unknown whether this map gives a cochain map. Their solution was using the decomposition of the graph complex by the first Betti number $g$ of graphs: They restricted graphs to $g \leq 1$, and showed the non-triviality of infinitely many geometric cycles of $\mathcal{K}_{n,j}$ and $\overline{\mathcal{K}}_{n,j}$ for $n-j\geq 2$ by using configuration space integrals of $0$-loop  or $1$-loop graph cocycles.

In the author's paper \cite{Yos 1}, we addressed the $2$-loop part ($g=2$) for the simplest case. We explicitly gave the simplest graph cocycle for odd $n$ and $j$ which consists of $2$-loop BCR graphs including the $2$-loop hairy graph 
\[
\Theta(1,0,1) = \graphg.
\]
We showed that the configuration space integrals of this graph cocycle give a cocycle of $\overline{\mathcal{K}}_{n,j}$ after adding correction terms. We gave a geometric cycle $\overline{\psi} : S^{j-1} \times (S^{n-j-2} )^{\times 3} \rightarrow \overline{\mathcal{K}}_{n,j}$ and showed the non-triviality by pairing argument. However, we have obtained no other cocycles of $\overline{\mathcal{K}}_{n,j}$ from this approach: We have found no other graph cocycles of $GC^{BCR}_{n,j}$. Even if they exist, there is no guarantee they give cocycles of  $\overline{\mathcal{K}}_{n,j}$. 

To handle this situation, in \cite{Yos 3}, we constructed general cocycles of $\overline{\mathcal{K}}_{n,j}$ by modifying configuration space integrals. Instead of using the original BCR graph complex $GC^{BCR}_{n,j}$, we introduced a new graph complex $DGC_{n,j}$, the \textit{decorated graph complex}.  We defined configuration space integrals
\[
\overline{I}: DGC_{n,j} \rightarrow \Omega_{dR} (\overline{\mathcal{K}}_{n,j}), 
\] 
which is a cochain map when $n-j \geq 2$ and $j \geq 3$ \footnote{It holds even when $j = 2$, if $\overline{I}$ is restricted to the subspace of $DGC_{n,j}$ generated by decorated graphs with the first Betti number $g \leq 3$. }. We also showed that $DGC_{n,j}$ is quasi-isomorphic to the hairy graph complex $HGC_{n,j}$. 

\section*{}
In this paper, we show the non-triviality of some geometric cycles of $\overline{\mathcal{K}}_{n,j}$ $(n-j \geq 2)$  by using the $2$-loop and $3$-loop part of the modified configuration space integrals introduced in \cite{Yos 3}. The construction of the geometric cycles is based on \cite{Yos 1}: we give geometric cycles associated with chord diagrams, which also correspond to \textit{ribbon presentations} \cite{HKS, HS} or \textit{planetary systems} \cite{Yos 1}. The pairing is reduced to pairing of graphs with chord diagrams similarly to \cite{Yos 1}. The new point is that we give the chord diagrams associated with $2$-loop or $3$-loop hairy graphs. Then, somewhat surprisingly, the result of the pairing coincides with the coefficient of the original hairy graph in the given graph cocycle.

In Section \ref{Construction of general $2$-loop and $3$-loop cycles}, we give geometric cycles of $\overline{\mathcal{K}}_{n,j}$:
\[
d(\Theta(p,q,r)): S^{j-1} \times (S^{n-j-2})^{\times k} \rightarrow  \overline{\mathcal{K}}_{n,j} \quad (k = p+q+r+1). 
\]
We systematically construct these cycles from chord diagrams on oriented lines, which recover ribbon presentations and planetary systems \cite{Yos 1}. 
These chord diagrams are associated with the $2$-loop hairy graphs $\Theta(p,q,r)$, where $\Theta(p,q,r)$ is the graph obtained by attaching $p$, $q$, $r$ hairs to the graph $\Theta$. See Figure \ref{Thetapqr}. Note that the parameter space $(S^{n-j-2})^{\times k}$ already appeared in Sakai and Watanabe's construction \cite{SW} of \textit{wheel-like cycles} of the $1$-loop case, while the parameter space $S^{j-1}$ is firstly introduced in the (simplest) $2$-loop case in \cite{Yos 1}. 
We also give $3$-loop cycles 
\[
d(\ctext{Y}(p_i)_{i = 1, \dots, 6}) : (S^{n-j-2})^k \times (S^{j-1})^2  \rightarrow \overline{\mathcal{K}}_{n,j} \quad (k = 2 + \sum_{i=1}^6 p_i).
\]

In particular, when $n-j = 2$, the parameter space $(S^{n-j-2})^{\times k}$ equals to the product of $S^{0} = \{+1, -1\}$. We replace the cycles $d(\Theta(p,q,r))$ (and $d(\ctext{Y}(p_i)_{i = 1, \dots, 6})$) with the modified cycles $d^{\prime}(\Theta(p,q,r))$ (and $d^{\prime}(\ctext{Y}(p_i)_{i = 1, \dots, 6})$)  so that they are homologous to the $(+1, +1, \dots, +1)$ component and that the $(+1, +1, \dots, +1)$ component lies in the unknot component. See Notation \ref{def unknot component} for the definition of the unknot component $(\mathcal{K}_{n,j})_{\iota}$ (resp. $(\overline{\mathcal{K}}_{n,j})_{\overline{\iota}}$) of $\mathcal{K}_{n,j}$ (resp. $\overline{\mathcal{K}}_{n,j}$).

In Section \ref{Pairing for general cases}, we detect the non-triviality of the cycles in Section \ref{Construction of general $2$-loop and $3$-loop cycles} by the modified configuration space integrals. The following is our main result. 
\begin{theorem}[Theorem \ref{nontrivialityof2loopcycles}, \ref{nontrivialityof3loopcycles}] 
\label{main theorem on general cases}
Let $\mathcal{H}_{n,j}(g)$  ($g=2, 3$)  be the subspace of $H_{\ast}(\overline{\mathcal{K}}_{n,j})$ which is generated by (replaced) $g$-loop cycles  in Section \ref{Construction of general $2$-loop and $3$-loop cycles}. Then there is a map 
\[
Pair: H^{top}(HGC_{n,j}(g)) \otimes \mathcal{H}_{n,j} (g) \rightarrow \mathbb{R},
\]
which is non-degenerate with respect to $H^{top}(HGC_{n,j}(g))$. Here, $H^{top}(HGC_{n,j})$ is the top cohomology of $HGC_{n,j}$
 \footnote{The top cohomology is the subspace of the cohomology of degrees $k(n-j-2) + (j-1)(g-1)$, where $g$ and $k$ are the first Betti number and the order of graphs, respectively. If $n-j=2$ and $g=2$, all the degrees concentrated on degree $j-1$. }. It is generated by cocycles consisting of uni-trivalent hairy graphs.
\end{theorem} 

On the other hand, the $2$-loop part of $ H ^{top}(HGC_{n,j})$ is well-studied in \cite{Nak, CCTW}, and is shown to be infinite-dimensional. The  $3$-loop  part is also infinite-dimensional at least when $n-j$ is even \cite{MO, Yos 3}. As a corollary of the $2$-loop and $3$-loop part, we have the following. Let $\text{Emb}_{\partial}(D^j, D^n)_{\iota}$ be the unknot component of the space of embeddings $D^j \rightarrow D^n$ which are standard near $\partial D^j$ \footnote{The space $\text{Emb}_{\partial}(D^j, D^n)$ is weakly equivalent to $\mathcal{K}_{n,j}$.}. 

\begin{coro}
\label{corollarymaintheorem}
[Cororally \ref{non-finite generation}, \ref{3loop non-finite generation}]
Let $j \geq 2$. Then, 
\begin{itemize}
\item $\pi_{j-1} (\text{Emb}_{\partial}(D^j, D^{j+2})_{\iota}) \otimes \mathbb{Q}$ is infinite-dimensional.
\item $\pi_{2(j-1)} (\text{Emb}_{\partial}(D^j, D^{j+2})_{\iota}) \otimes \mathbb{Q}$  is infinite-dimensional.
\end{itemize}
\end{coro}

This first result is first established by Budney--Gabai and Watanabe \cite{BG, Wat 5} by studying the diffeomorphism group $\text{Diff}_{\partial}(D^{j+1}\times S^1)$. 
To the author's knowledge, the second result is new, though Budney--Gabai and Watanabe's approaches are very likely to be extended to this range. 
It would be interesting that in Budney--Gabai and Watanabe's approach, the parameter space of $S^1$ of $D^{j+1}\times S^1$ plays an essential role, while our approach is also applicable to higher codimensions $n-j \geq 3$. One would be able to relate our results with their results by describing our cycles in terms of Goussarov and Habiro's claspers \cite{GGP, Hab} and their higher-dimensional analogs by Watanabe \cite{Wat 1, Wat 3, Wat 4, Wat 5}.

Let us explain how Theorem \ref{main theorem on general cases} and Corollary \ref{corollarymaintheorem} are obtained. The map $Pair$ is induced by the pairing between our cycles and the modified configuration space integrals $\overline{I}$. (Recall that $DGC_{n,j}$ is quasi-isomorphic to $HGC_{n,j}$.)
Similarly to \cite{Yos 1}, the cocycle-cycle pairing is reduced to pairing between graphs and chord diagrams. We call this argument \textit{counting formula} (Theorem \ref{original counting formula}). Using this formula, we can show that the value of the cycle $d(\Theta(p,q,r))$ only depends on the coefficient $w(\Theta(p,q,r))$ of the graph $\Theta(p,q,r)$ in the given graph cocycle. Then, we can show the map $Pair$ is non-degenerate. Recall that $\mathcal{H}_{n,j} (g=2)$ gives, thanks to a replacement, a subspace of $\pi_{\ast}((\overline{\mathcal{K}}_{j+2, j})_{\overline{\iota}})$. Since $H^{top}((HGC_{j+2,j}(g=2))$ is infinite-dimensional and is concentrated in degree $j-1$, we have Cororally \ref{corollarymaintheorem}. The case $g=3$ works similarly. 

Our approach is very likely to show the non-finite generation of the $(g-1)(j-1)$th homotopy group of $\text{Emb}_{\partial}(D^j, D^{j+2})_{\iota}$ by using $g$-loop ($g\geq 2$) hairy graphs. (Recall that the $1$-loop part detects infinitely many ribbon $j$-knots $D^j \rightarrow D^{j+2}$ \cite{Wat 2, SW}.) One thing to address is to replace the parameter space $(S^{j-1})^{\times k}$ to the sphere $S^{(j-1)k}$, which would be realized when our cycles are described by claspers (see Corrigendum of \cite{Wat 3}). Another thing to do is to give a lower bound of $H^{top}(HGC_{n,j}(g))$.

\section*{}
This paper is organized as follows. 
In Section \ref{Preliminaries for construction of cycles}, we recall the notion of ribbon presentations, planetary systems and chord diagrams.  In Section \ref{Construction of general $2$-loop and $3$-loop cycles}, we give general $2$-loop  cycles $d(\Theta(p,q,r))$ from the ribbon presentations $P(\Theta(p,q,r))$ associated with the $2$-loop hairy graphs $\Theta(p,q,r)$. Then, we show some properties of $P(\Theta(p,q,r))$ so that the $2$-loop cycles are taken from the unknot component. At the end of Section \ref{Construction of general $2$-loop and $3$-loop cycles}, we give the class of $2$-loop cycles $\mathcal{H}_{n,j}(g=2)$. We also construct $3$-loop cycles. 
 
In Section \ref{Pairing for general cases}, we first recall the counting formula for configuration space integrals in \cite{Yos 1} and see that the formula is applicable to the modified configuration space integrals. Then, we show the non-triviality of the $2$-loop  and $3$-loop (co)cycles introduced in Section \ref{Construction of general $2$-loop and $3$-loop cycles}.

\section*{Acknowledgement}
This paper is based on a part of the author's PhD thesis. 
The author is deeply grateful to Victor Turchin, Tadayuki Watanabe and Keiichi Sakai for helpful and motivating discussions. 
This research was supported by Forefront Physics and Mathematics Program to Drive Transformation (FoPM), a World-leading Innovative Graduate Study (WINGS) Program, The University of Tokyo. This work was also supported by JSPS Grant-in-Aid for JSPS Fellows Grant Number 24KJ0565.

\tableofcontents

\section{Preliminaries for construction of cycles}
\label{Preliminaries for construction of cycles}

For convenience, we briefly recall the notion of ribbon presentations, planetary systems, chord diagrams on oriented lines, which are introduced in \cite{Yos 1}. 
We see that each one of the three recovers the other two. Readers who are familiar with ribbon presentations and their perturbation can skip reading this section. 

\subsection{The space of long embeddings}

\begin{definition}
A \textit{long embedding} is an embedding $\mathbb{R}^j \rightarrow \mathbb{R}^n$ that coincides with the standard linear embedding $\iota: \mathbb{R}^j \subset \mathbb{R}^n$ outside a disk in $\mathbb{R}^j$. We write $\mathcal{K}_{n,j} = \text{Emb}(\mathbb{R}^j, \mathbb{R}^n)$ for the space of long embeddings equipped with the induced topology from the $C^{\infty}$ topology. This space $\mathcal{K}_{n,j}$ is weakly equivalent to the space of embedding of disks $\text{Emb}(D^j, D^n)$. 
We define the space of \textit{long immersions} $\text{Imm}(\mathbb{R}^j, \mathbb{R}^n)$ similarly. 
\end{definition}

\begin{definition}
\label{embeddings modulo immersions}
We write $\overline{\mathcal{K}}_{n,j} = \overline{\text{Emb}}(\mathbb{R}^j, \mathbb{R}^n)$ for the homotopy fiber of the inclusion
\[
\mathcal{K}_{n,j} = \text{Emb}(\mathbb{R}^j, \mathbb{R}^n) \hookrightarrow \text{Imm}(\mathbb{R}^j, \mathbb{R}^n)
\]
at the standard linear embedding $\iota: \mathbb{R}^j \subset \mathbb{R}^n$. This space is called \textit{the space of long embeddings modulo immersions}. An element of $\overline{\mathcal{K}}_{n,j}$ is a one-parameter family $\{\overline{K}_t\}_{t \in [0,1]}$ of long immersions which satisfies $\overline{K}_0\in \mathcal{K}_{n,j}$ and $\overline{K}_1 =~\iota$. 
\end{definition}

\begin{notation}\label{def unknot component}
We write $\iota \in \mathcal{K}_{n,j}$ for the standard linear embedding $\mathbb{R}^j \subset \mathbb{R}^n$. We write $\overline{\iota} \in \overline{\mathcal{K}}_{n,j}$ for the trivial one-parameter family of the standard linear embedding $\mathbb{R}^j \subset \mathbb{R}^n$, that is, $\overline{\iota}_t = \iota \in \text{Imm}(\mathbb{R}^j, \mathbb{R}^n)$ for all $t \in [0,1]$. We call the component $(\mathcal{K}_{n,j})_{\iota}$ (resp. $(\overline{\mathcal{K}}_{n,j})_{\overline{\iota}}$) containing $\iota$ (resp. $\overline{\iota}$) the \textit{unknot component} of  $\mathcal{K}_{n,j}$ (resp. $\overline{\mathcal{K}}_{n,j}$).
\end{notation}

\subsection{Ribbon presentations and long embeddings associated with them}
\label{Ribbon presentations with more than one node}

\begin{definition}[Ribbon presenations \cite{HKS, HS}]
\label{general def of ribbon presentation}
A  \textit{ribbon presentation} $P = \mathcal{D} \cup \mathcal{B}$ of order $k$ is a based oriented, immersed 2-disk in $\mathbb{R}^{3}$, where
\begin{itemize}
\item $\mathcal{D} = D_0 \cup D_1 \cup \dots \cup D_k$ are disjoint $(k+1)$ disks,
\item $\mathcal{B}=  B_1 \cup B_2 \cup \dots \cup B_k$ are disjoint $k$ bands ($B_i \approx I \times I$). 
\end{itemize}
Each band connects two different disks, and along the way it can intersect transversally with the interior of disks. We put the base point $\ast$ on the boundary of $D_0$ (but not on the boundaries of $B_i$).
\end{definition}

\begin{notation}\cite[Definition 3.1]{HS}
\label{nodeleaf}
An intersection of a band and a disk is called a \textit{crossing}. A disk without intersecting bands is called a \textit{node}.  A disk is called a \textit{leaf} if it has at least one intersecting band and it is incident to exactly one band. We often draw a grey disk for a node except for the based disk $D_0$. 
\end{notation}

Near the crossing of a band and a disk, they look, using a local coordinate, like
\begin{align*}
B&= \{(x_1, x_2, x_3) \in \mathbb{R}^3\,|\,|x_1| \leq \frac{1}{2}, |x_2| < 3, x_3=0\},\\
D&= \{(x_1, x_2, x_3) \in \mathbb{R}^3\,|\,|x_1|^2 + |x_3|^2 \leq 1, x_2=0\}.
\end{align*}

\begin{notation}[Orientation of a crossing]
Since a ribbon presentation is an oriented disk, each disk $D_j$ is oriented. Orient the core of each band $B_i$ so that it goes from a leaf to the based disk $D_0$. 
Then the crossing of the band with a disk $D_j$ is called positive (resp. negative) if the core of the band gives the oriented normal vector (resp. minus of the oriented normal vector) of $D_j$ in $\mathbb{R}^3$.
\end{notation}

In \cite{HKS, HS}, Habiro and Shima introduced moves of ribbon presentations that do not change the isotopy class of the ribbon $2$-knot represented by the presentation. We can observe that the same is true for the ribbon $n$-knot represented by the presentation. The following are examples of the moves.

\begin{notation}\cite[Definition 3.2]{HS}
We define $S1$, $S2$, $S4$, $S7$ moves as in Figure \ref{exofmoves}.
\begin{figure}[htpb]
  \centering
    \includegraphics [width =12cm] {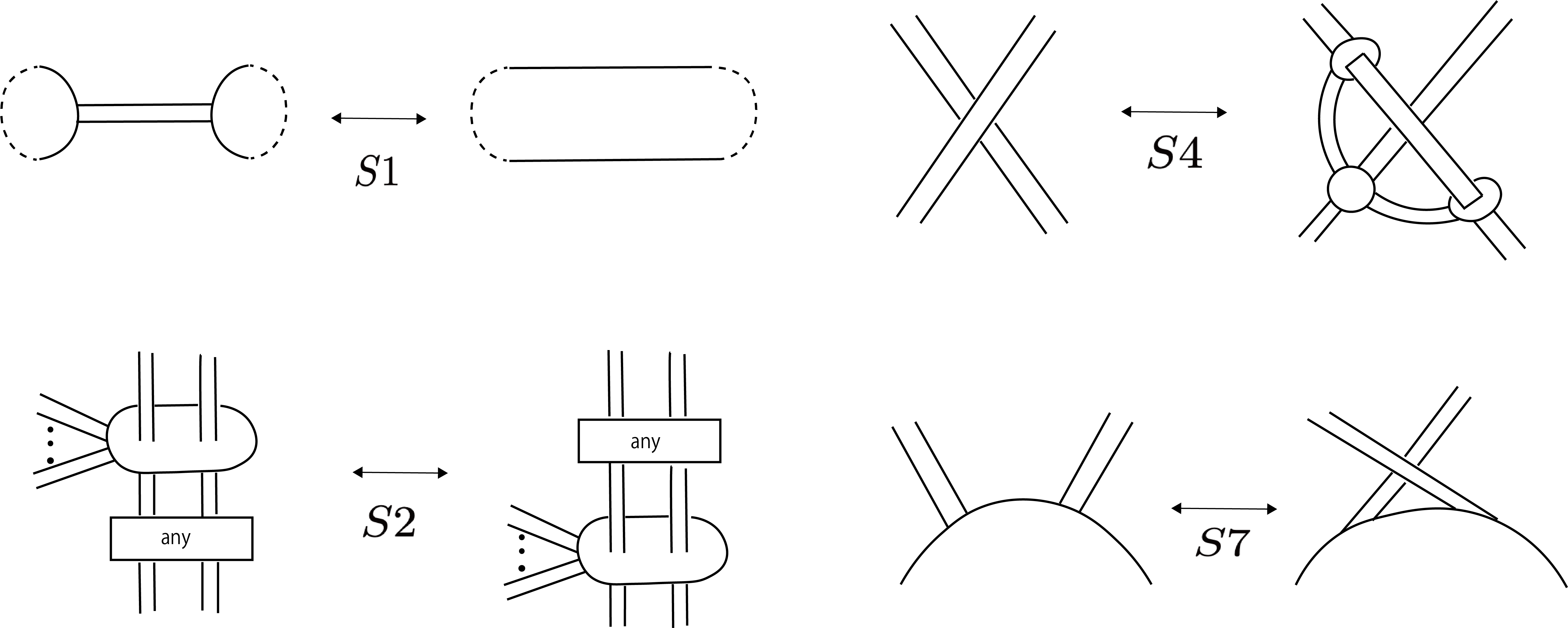}
    \caption{Example of moves of ribbon presentations}
    \label{exofmoves}
\end{figure}
\end{notation}

\begin{rem}
We have the move in Figure \ref{S4prime} by using $S1$ and $S4$ moves. We call this move $S4'$ move.
\begin{figure}[htpb]
\centering
\includegraphics[width = 12cm]{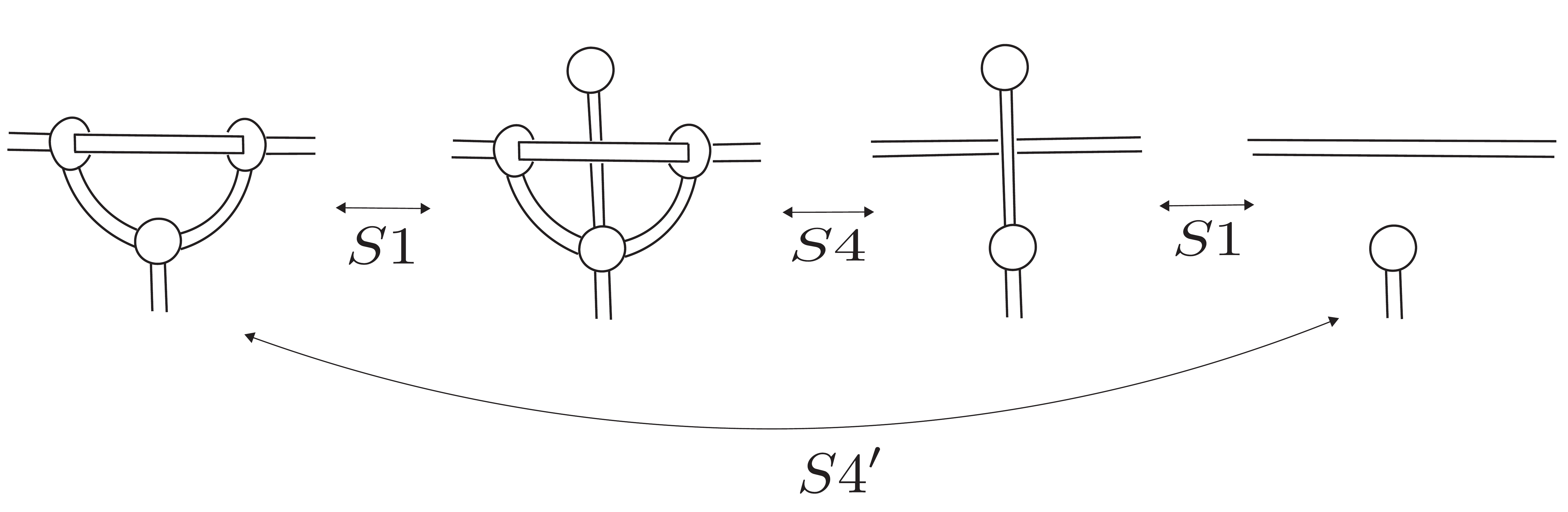}
\caption{$S4^{\prime}$ move}
\label{S4prime}
\centering
\end{figure}
\end{rem}

\begin{notation}
When $n-j=2$, we sometimes encounter a situation where we want to distinguish two types of crossings: crossings to which we perform perturbations defined in subsection  \ref{Perturbation of crossings}, and crossings to which we don't perform perturbations. We often put the label $\star$ to intersections to which we perform perturbations. 
\end{notation}

\begin{definition}\cite[Definition 4.2] {SW}
We define a \textit{ribbon $(j+1)$-disk} $V_P$ by 
\[
V_P = \left(\mathcal{B} \times [-\frac{1}{4}, \frac{1}{4}]^{j-1} \right) \bigcup  \left(\mathcal{D} \times [-\frac{1}{2}, \frac{1}{2}]^{j-1}\right) \subset \mathbb{R}^3 \times \mathbf{0} \times  \mathbb{R}^{j-1}.
\]
Note that the thicknesses of bands and that of disks are different.
\end{definition}

\begin{definition}
An embedding $\varphi_k : \mathbb{R}^j \rightarrow \mathbb{R}^n$ is defined by the connected sum $\partial V_P \# \iota(\mathbb{R}^j)$ at the base point, after the smoothing of corners of $V_P$. Here, the parametrization of the embedding is given by using the canonical path to the standard immersion in $\text{Imm}(\mathbb{R}^j, \mathbb{R}^n)$. See Definition \ref{path to standard immersion}
\end{definition}

\begin{rem}
Although $\mathcal{B} \times [-\frac{1}{4}, \frac{1}{4}]^{j-1}$ and $\mathcal{D} \times [-\frac{1}{2}, \frac{1}{2}]^{j-1}$ intersect, their boundaries do not intersect:
\begin{align*}
\partial(\mathcal{B} \times [-\frac{1}{4}, \frac{1}{4}]^{j-1}) &= \left(\partial\mathcal{B} \times [-\frac{1}{4}, \frac{1}{4}]^{j-1}\right) \cup \left(\mathcal{B} \times \partial [-\frac{1}{4}, \frac{1}{4}]^{j-1}\right) \\ 
&=   \{(x_1, x_2, x_3) \in \mathbb{R}^3\,|\,|x_1| = \frac{1}{2}, |x_2| < 3, x_3=0\} \times  [-\frac{1}{4}, \frac{1}{4}]^{j-1} \\
&\cup \{(x_1, x_2, x_3) \in \mathbb{R}^3\,|\,|x_1| \leq \frac{1}{2}, |x_2| < 3, x_3=0\} \times \partial [-\frac{1}{4}, \frac{1}{4}]^{j-1},
\end{align*}
\begin{align*}
\partial(\mathcal{D} \times [-\frac{1}{2}, \frac{1}{2}]^{j-1}) &= \left(\partial\mathcal{D} \times [-\frac{1}{2}, \frac{1}{2}]^{j-1}\right) \cup \left(\mathcal{D} \times \partial [-\frac{1}{2}, \frac{1}{2}]^{j-1}\right) \\ 
& = \{(x_1, x_2, x_3) \in \mathbb{R}^3\,|\,|x_1|^2 + |x_3|^2  = 1, x_2=0\} \times  [-\frac{1}{2}, \frac{1}{2}]^{j-1}  \\
&\cup \{(x_1, x_2, x_3) \in \mathbb{R}^3\,|\,|x_1|^2 + |x_3|^2 \leq 1, x_2=0\}  \times \partial [-\frac{1}{2}, \frac{1}{2}]^{j-1}.
\end{align*}
\end{rem}

\subsection{Cycles of $\overline{\mathcal{K}}_{n,j}$ obtained by perturbation of crossings of ribbon presentations}
\label{Perturbation of crossings}
We recall the notion of the \textit{perturbation} $\varphi^{\mathbf{v}}_k$ $(\mathbf{v} \in S^{n-j-2})$ of the  embedding $\varphi_k$. It is constructed from the perturbed ribbon presentation 
$P_{\mathbf{v}}$ of the original ribbon presentation $P$. 

Recall that the ribbon presentation $P$ is constructed in $\mathbb{R}^3 \times \mathbf{0} \times \mathbf{0} \subset \mathbb{R}^n$, and the thickened presentation $V_p$ is constructed in $\mathbb{R}^3 \times \mathbf{0} \times \mathbb{R}^{j-1}$. The perturbation is performed using the coordinates $(x_4, \dots, x_{n-j+1})$ of $\mathbb{R}^{n-j-2}$ and the coordinate $x_3$ of $\mathbb{R}^3$. 

\begin{notation}
We consider the $(n-j-2)$-dimensional sphere $S^{n-j-2}$ as
\[
S =S^{n-j-2}=\left\{(x_3, \dots, x_{n-j+1}) \in \mathbb{R}^{n-j-1}\, | \, (x_3-1)^2+ x_4^2+\dots+x^2_{n-j+1}=1\right\}.
\]
If $n-j-2=0$, $S$ is the set of two points $S^0 =\{x_3=0, x_3=2\}$.
\end{notation}

\begin{notation}
The band obtained by perturbing $B$ to $v\in S$ direction is written as $B(v)$. 
\[
B(v) = 
\begin{cases}
\left\{(x_1, x_2, \gamma(x_2)v) \in \mathbb{R}^2 \times \mathbb{R}^{n-j-1} \, | \, |x_1| \leq 1/2, |x_2|<3 \right\}.\\
\end{cases}
\]
Here $\gamma(y)$ is a test function defined as 
\[
\gamma(y)=
\begin{cases}
 \text{exp}(-y^2/\sqrt{9-y^2}) & (|y| \leq 3) \\
 0 & (|y| \geq 3)
\end{cases}.
\]
\end{notation}

\begin{definition}
We define $B_i(v_i) (v_i \in S) $ as the band in $\mathbb{R}^3 \times \mathbb{R}^{n-j-2} \times \mathbf{0} $ constructed from $B_i$ by replacing $B \times \mathbf{0}$ by $B(v_i)$.
\end{definition}

\begin{notation}
Let $\mathbf{v} = (v_1, \dots, v_k) \in  {(S^{n-j-2})}^k$. We define $\mathcal{B}_{\mathbf{v}}$ by 
\[
\mathcal{B}_{\mathbf{v}} =  B_1 (v_1) \cup B_2 (v_2) \cup \dots \cup B_k (v_k).
\]
The \textit{perturbation} $P_{\mathbf{v}} \subset \mathbb{R}^{n-j+1}$ of the ribbon presentation $P$ is defined by $P_{\mathbf{v}} = \mathcal{D} \cup \mathcal{B}_{\mathbf{v}} $. Set
\[
V_{P_{\mathbf{v}}} = \left(\mathcal{B}_{\mathbf{v}} \times [-\frac{1}{4}, \frac{1}{4}]^{j-1} \right) \bigcup  \left(\mathcal{D} \times [-\frac{1}{2}, \frac{1}{2}]^{j-1}\right) \subset \mathbb{R}^n. 
\]
\end{notation}

\begin{definition}
The \textit{perturbation} $\varphi^\mathbf{v}_k$ of $\varphi_k: \mathbb{R}^j \rightarrow  \mathbb{R}^j$ is defined by the connected sum $\partial V_{P_{\mathbf{v}}} \# \iota(\mathbb{R}^j)$. \end{definition}

\begin{notation}\cite [Definition 4.6] {SW}
The part of $\varphi^\mathbf{v}_k$ corresponding to a crossing of the ribbon presentation is also called a crossing, although it is not a crossing in a usual sense. 
\end{notation}

\begin{definition}
\label{path to standard immersion}
We define a cycle $\widetilde{c_k} : (S^{n-j-2})^k \longrightarrow \overline{\mathcal{K}}_{n,j}$ as follows. Consider the following sequential resolutions of crossings which are possible in the space of immersions.
 \begin{itemize}
\item [($m1$)] Pull the disk $D_j$ of each crossing to the $x_1$-direction so that any $D_j$ and $B_i(v_i)$ do not intersect.
\item [($m2$)] Pull back  $D_j$ and $B_i(v_i)$ to the based disk $D_0$.
\end{itemize}
This operation gives a path from $\varphi^\mathbf{v}_k$ to the trivially immersed $\mathbb{R}^j$. Using this path, we can equip each point of the submanifold $\varphi^\mathbf{v}_k$ with a coordinate. This path also gives a lift $\widetilde{\varphi^\mathbf{v}_k}$ of $\varphi^\mathbf{v}_k$ to $\overline{\mathcal{K}}_{n,j}$. We define the  cycle $\widetilde{c_k} : (S^{n-j-2})^k \longrightarrow \overline{\mathcal{K}}_{n,j}$ by $\mathbf{v} \mapsto \widetilde{\varphi^\mathbf{v}_k}$.
\end{definition}


\subsection{Planetary  systems and chord diagrams on oriented lines}
\begin{definition}\cite{Yos 1}
A set of \textit{planetary systems} $\mathcal{S}$ on $\mathbb{R}^j$ of order $k$ consists of the following data.
\begin{itemize}
\item Points $a_i$ at $(i,0,\dots,0) \in \mathbb{R}^j \quad (i=1, \dots, s)$.
\item $(j-1)$-dimensional spheres $S(L^l_i, a_i)$\quad$(l=1, \dots, t_i)$ with radius $L^j_i$ centered at $a_i$ $(i = 1,\dots, s)$.
\end{itemize}
$s$ and $t_i\ (i=1, \dots, s)$ are non-negative integers such that the total number of points and spheres is $2k$. 
$L^j_i$ are sufficiently small real numbers and satisfy $L^j_i > L^k_i $ if $j > k$. We call points of (1)  \textit{fixed stars} and call spheres of (2)  \textit{planetary orbits}. Each fixed star forms one planetary system that consists of $t_i$ planetary orbits.
\end{definition}

\begin{definition}
An \textit{ordered pairing} $\{p_i\}_{i =1, \dots, k}$ for $\mathcal{S}$ is a pairing among the set of fixed stars and planetary orbits.
Fixed stars must be paired with planetary orbits. Planetary orbits may be paired with other planetary orbits (\textit{orbit-orbit pairing}). The order of the two elements of each pair is taken so that the fixed star is the first. 
\end{definition}

A chord diagram on oriented lines can describe a set of planetary systems and its ordered pairing.
\begin{definition}\cite{Yos 1}
A \textit{chord diagram on $s$ oriented lines} of order $k$ consists of the following data.
\begin{itemize}
\item  Integers $t_i \geq 0$ ($i = 1, \dots, s$) such that $\sum_{i=1}^s (t_i +1) = 2k$. We draw $\{x = i\} \subset \mathbb{R}^2$ by a blue oriented line. 
\item  An ordered pairing $\{p_i\}_{i=1,\dots,k}$ among $2k$ points $V = \{(i,l)\in \mathbb{Z}^2\, |\, 1\leq i \leq s, 0\leq l \leq t_i\}$, which is drawn by a chord. 
\end{itemize}
Note that the set of chords is ordered, and each chord is oriented. We can order the set of vertices of a chord diagram canonically using the ordering and orientations of chords: The initial vertex of the $i$th chord is labeled by $(2i-1)$. The end vertex of the $i$th chord is labeled by $2i$. 

We always assume that at least one of the two ends of each chord is not on the $x$-axis and that any point on the $x$-axis must be the initial point of some chord. 
\end{definition}

\begin{example}
\label{exampleofchorddiagram}
\chorddiagramb{} is an example of a chord diagram on oriented lines ($s=2, t_1 = t_2 =2$). We write $O$ for the points on the $x$-axis.
\end{example} 

A set of planetary systems and its ordered pairing determines a chord diagram on oriented lines in an expected way. Namely, points on the $x$-axis are fixed stars. Other points are planetary orbits. Conversely, a chord diagram on oriented lines determines a set of planetary systems (up to rescalings) and its ordered pairing.

\begin{definition}[The ribbon presentation associated with a chord diagram with signs]
\label{ribbon presentation from diagram}
Consider a chord diagram on oriented lines with each chord given a sign $+$ or $-$. To this chord diagram, we assign a ribbon presentation as follows.
Each oriented line corresponds to a sequence of bands and nodes that connects the based disk $D_0$ and a leaf. The origin of each oriented line corresponds to this leaf.  Each chord with the sign $+$ (or $-$) corresponds to a positive (or negative) crossing. 
The initial point of each chord corresponds to the disk (leaf) of the crossing. If this initial point is not the origin of the oriented line, this leaf is connected to a node by a band. The target point of each chord corresponds to a small segment of the band of the crossing. 
\end{definition}

\begin{example}\cite{Yos 1}
We assign the ribbon presentation in Figure  \ref{exampleofribbonpresentation}
to the chord diagram in Example \ref{exampleofchorddiagram}, where we aassume the signs of the chords $(1), (2), (3)$ are $-, +, +$ respectively.
In Figure  \ref{exampleofribbonpresentation}, the crossings with label (1), (2), (3) correspond to the chords with label (1), (2), (3), respectively.

\begin{figure}[htpb]
\centering
    \includegraphics [width =5cm] {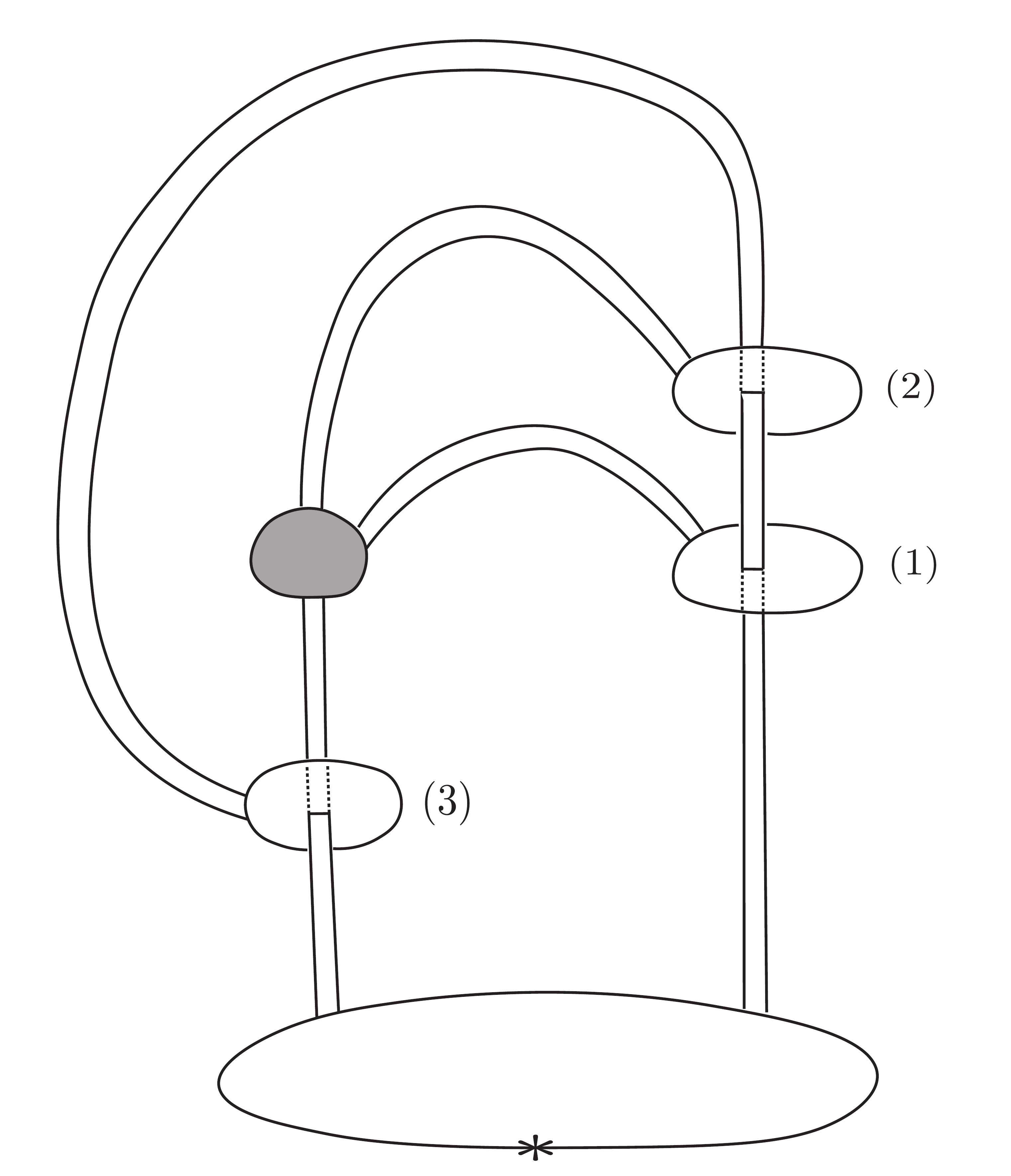}
   \caption{The ribbon presentation associated with the chord diagram in Example \ref{exampleofchorddiagram}}
    \label{exampleofribbonpresentation}
 \end{figure}

\end{example}


\section{Construction of $2$-loop and $3$-loop cycles of $\overline{\mathcal{K}}_{n,j}$}
\label{Construction of general $2$-loop and $3$-loop cycles}

Let $k\geq 1$, $p, r \geq 1,\ q \geq 0,\ p+q+r+1=k$. In this section, we construct  $(k(n-j-2)+(j-1))$-cycles
\[
d(\Theta(p,q,r)) : (S^{n-j-2})^k \times S^{j-1}  \rightarrow \overline{\mathcal{K}}_{n,j}
\]
by perturbation of ribbon presentations with one (gray) node. These cycles, after some modifications, generate the class  $\mathcal{H}_{n,j}(k,2)$. The cycle $d(\Theta(1,0,1))$ coincides with the simplest $2$-loop cycle introduced in \cite{Yos 1}. The general cycle $d(\Theta(p,q,r))$ is detected by the cocycle including the $2$-loop hairy graph $\Theta(p,q,r)$ whose three edges of $``\Theta"$ have $p$, $q$, $r$ hairs respectively. See Figure \ref{Thetapqr}.

\begin{figure}[htpb]
\centering
    \includegraphics [width =4cm] {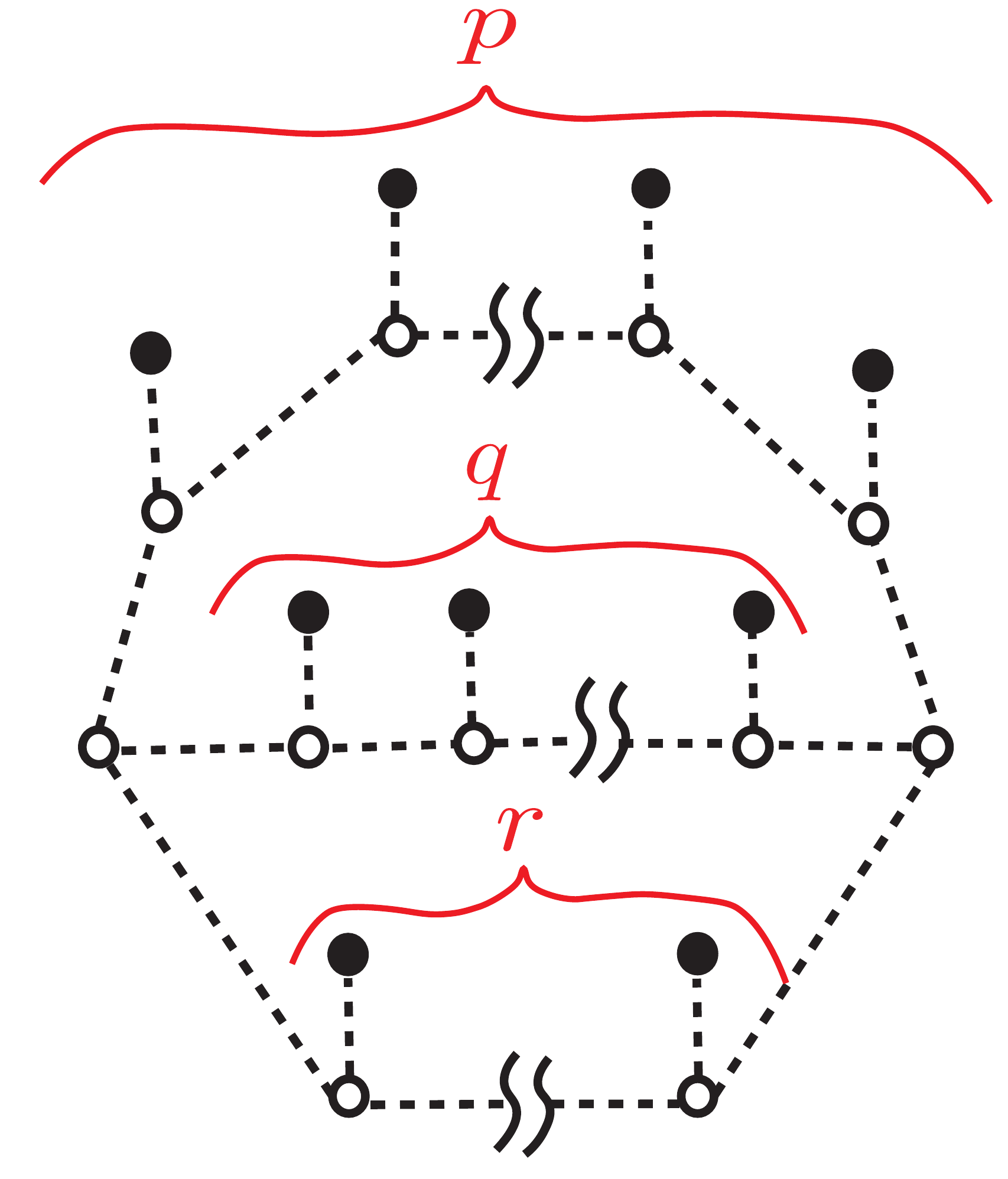}
   \caption{The hairy graph $\Theta(p,q,r)$}
    \label{Thetapqr}
 \end{figure}

We construct the cycles as follows. First, we give a chord diagram on oriented lines $D(\Theta(p,q,r))$ from the hairy graph $\Theta(p,q,r)$. From this diagram, we give a ribbon presentation  $P(\Theta(p,q,r))$. Such a ribbon presentation gives $S^{j-1} \times (S^{n-j-2})^{\times k}$ cycle of embedded submanifolds ($\approx \mathbb{R}^j$) in $\mathbb{R}^n$. Then the desired cycle $d(\Theta(p,q,r)): S^{j-1} \times (S^{n-j-2})^{\times k} \rightarrow \overline{\mathcal{K}}_{n,j}$ is obtained by giving a path of immersions to the trivial immersion. The path also gives the parameterization of the embedded submanifolds.

As mentioned before, our construction is analogous to the construction of \textit{wheel-like cycles} given by Sakai and Watanabe \cite{Wat 2, SW}, 
\[
\tilde{c}_k : (S^{n-j-2})^{\times k} \rightarrow \overline{\mathcal{K}}_{n,j},
\]
which are detected by $1$-loop graphs. The additional parameter on $S^{j-1}$ of $d(\Theta(p,q,r))$ arises from the \textit{node} (see Notation \ref{nodeleaf}), which only our ribbon presentations have. 

We also construct $(k(n-j-2)+2(j-1))$-cycles
\[
d(\ctext{Y}(p_i)_{i = 1, \dots, 6}) : (S^{n-j-2})^k \times (S^{j-1})^2  \rightarrow \overline{\mathcal{K}}_{n,j} \quad (k = 2 + \sum_{i=1}^6 p_i)
\]
by perturbation of ribbon presentations with two nodes, which are associated with the $3$-loop hairy graph $\ctext{Y}((p_i)_{i = 1, \dots, 6})$ (see the $3$-loop graph in subsection \ref{3loopcycles}). These cycles generate the class  $\mathcal{H}_{n,j}(k,3)$. 
Though we focus on $2$ or $3$-loop hairy graphs, this construction of cycles is very likely to be generalized to hairy graphs with an arbitrary number of loops.

\subsection{The chord diagrams $D(\Theta(p,q,r))$}
The chord diagram $D(\Theta(p,q,r))$ on oriented lines is obtained from the hairy graph $\Theta(p,q,r)$ as follows. See Figure \ref{figofdiagram}. 

\begin{itemize}
\item First, orient the three edges of the graph $\Theta$ as
\[
 \begin{tikzpicture}[x=2.5pt,y=2.5pt,yscale=0.15,xscale=0.15, baseline=-3pt] 

\draw (0,0) circle (100); 

\draw [-Stealth] (100,1) -- (100,0); 
\draw [-Stealth] (-100,-1) -- (-100,0); 
\draw [-Stealth] (100,0) -- (-100,0);

\draw(-130,0) node {$(I)$};
\draw (130,0) node {$(II)$};

\draw (0,120) node {$p$};
\draw (0,20) node {$q$};
\draw (0,-90) node {$r$};
\end{tikzpicture}
\]
The left vertex of $\Theta$ has two incoming edges, while the right vertex has one incoming edge. We call the first vertex type (I) and the second vertex type (II). \footnote{The terminology type (I) and type (II) is introduced in \cite{Wat 4}, though convention is different. }

\item Replace each hair with the oriented line with two open chords
\begin{tikzpicture}[x=0.75pt,y=0.75pt,yscale=0.15,xscale=0.15, baseline=-3pt] 

\draw  [-Stealth, line width =1pt, color={rgb, 255: red, 0;  green, 0; blue, 255}] (0,0)--(500,0);

\draw  [-Stealth, dash pattern = on 2pt off 3 pt, line width =1pt] (100,0)--(100,150);
\draw  [Stealth-, dash pattern = on 2pt off 3 pt, line width =1pt] (300,0)--(300,150);

\draw  [fill={rgb, 255:red, 0; green, 0; blue, 0 }  ,fill opacity=1 ]  (100,0) circle (20);
\draw  [fill={rgb, 255:red, 0; green, 0; blue, 0 }  ,fill opacity=1 ] (300,0) circle (20);
\end{tikzpicture}.
Exceptionally replace the leftmost (resp. rightmost) hair of the upper (resp. lower) edge of $\Theta(p,q,r)$ with 
\begin{tikzpicture}[x=0.75pt,y=0.75pt,yscale=0.15,xscale=0.15, baseline=-3pt] 

\draw  [-Stealth, line width =1pt, color ={rgb, 225 :red, 0; green, 0; blue, 225 } ] (0,0)--(700,0);

\draw  [-Stealth, dash pattern = on 2pt off 3 pt, line width =1pt] (100,0)--(100,150);
\draw  [-Stealth, dash pattern = on 2pt off 3 pt, line width =1pt] (300,0)--(300,150);
\draw  [Stealth-, dash pattern = on 2pt off 3 pt, line width =1pt] (500,0)--(500,150);

\draw  [fill={rgb, 255:red, 0; green, 0; blue, 0 }  ,fill opacity=1 ]  (100,0) circle (20);
\draw  [fill={rgb, 255:red, 0; green, 0; blue, 0 }  ,fill opacity=1 ] (300,0) circle (20);
\draw  [fill={rgb, 255:red, 0; green, 0; blue, 0 }  ,fill opacity=1 ] (500,0) circle (20);
\end{tikzpicture}
\quad (\text{resp.}
\begin{tikzpicture}[x=0.75pt,y=0.75pt,yscale=0.15,xscale=0.15, baseline=-3pt] 

\draw  [-Stealth, line width =1pt, color ={rgb, 255: red, 0; green, 0; blue, 255 } ] (0,0)--(700,0);

\draw  [-Stealth, dash pattern = on 2pt off 3 pt, line width =1pt] (100,0)--(100,150);
\draw  [Stealth-, dash pattern = on 2pt off 3 pt, line width =1pt] (300,0)--(300,150);
\draw  [Stealth-, dash pattern = on 2pt off 3 pt, line width =1pt] (500,0)--(500,150);

\draw  [fill={rgb, 255:red, 0; green, 0; blue, 0 }  ,fill opacity=1 ]  (100,0) circle (20);
\draw  [fill={rgb, 255:red, 0; green, 0; blue, 0 }  ,fill opacity=1 ] (300,0) circle (20);
\draw  [fill={rgb, 255:red, 0; green, 0; blue, 0 }  ,fill opacity=1 ] (500,0) circle (20);
\end{tikzpicture}).

\item Finally connect ends of chords as expected from the graph $\Theta(p,q,r)$. There are two ways to connect the two outgoing open chords of 
\begin{tikzpicture}[x=0.75pt,y=0.75pt,yscale=0.15,xscale=0.15, baseline=-3pt] 

\draw  [-Stealth, line width =1pt, color ={rgb, 225 :red, 0; green, 0; blue, 225 } ] (0,0)--(700,0);

\draw  [-Stealth, dash pattern = on 2pt off 3 pt, line width =1pt] (100,0)--(100,150);
\draw  [-Stealth, dash pattern = on 2pt off 3 pt, line width =1pt] (300,0)--(300,150);
\draw  [Stealth-, dash pattern = on 2pt off 3 pt, line width =1pt] (500,0)--(500,150);

\draw  [fill={rgb, 255:red, 0; green, 0; blue, 0 }  ,fill opacity=1 ]  (100,0) circle (20);
\draw  [fill={rgb, 255:red, 0; green, 0; blue, 0 }  ,fill opacity=1 ] (300,0) circle (20);
\draw  [fill={rgb, 255:red, 0; green, 0; blue, 0 }  ,fill opacity=1 ] (500,0) circle (20);
\end{tikzpicture}
to the other two ingoing open chords. We choose the one as in Figure \ref{figofdiagram}. 
\item We order the chords so that the upper, and tail-side if on the same level, the vertices have outgoing chords with smaller labels. We order the oriented lines by a similar rule. 
\end{itemize}

\begin{figure}[htpb]
\centering
    \includegraphics [width =11cm] {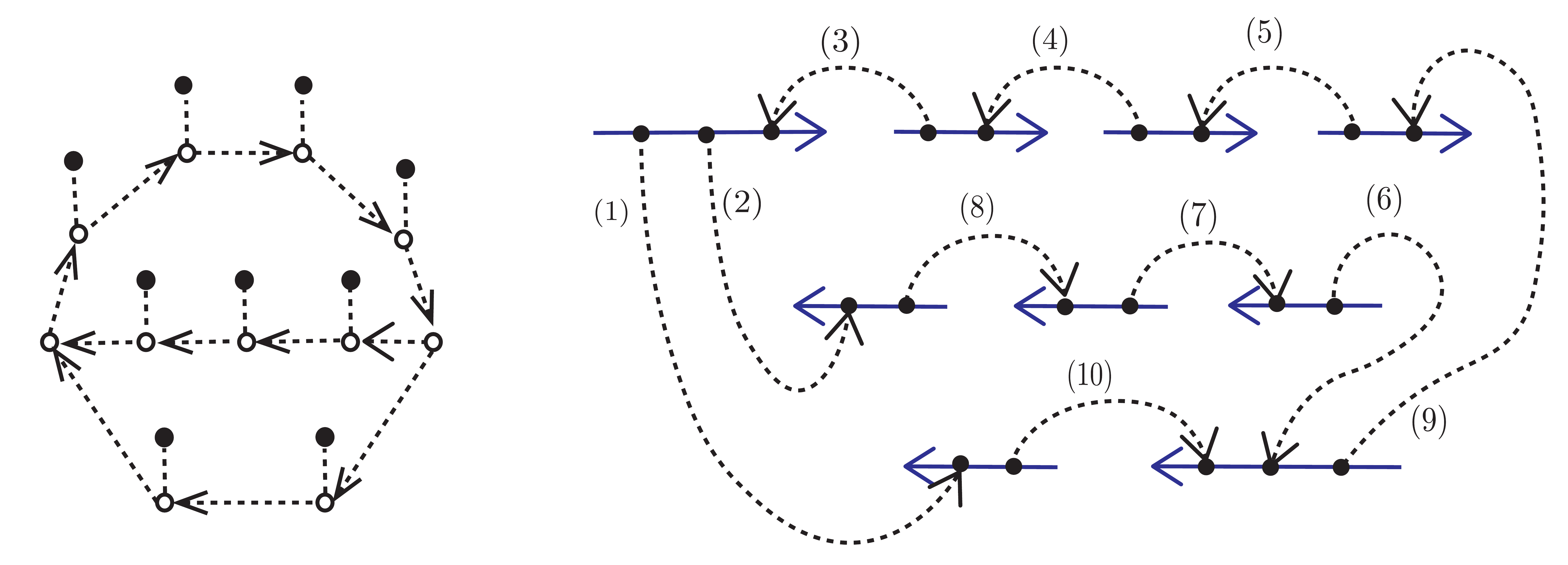}
    \caption{The diagram $D(\Theta(4,3,2))$)}
    \label{figofdiagram}
 \end{figure}
 
 Note that the diagram $D(\Theta(1,0,1))$ equals to the diagram  $C_1 = \chorddiagramb{}$ 
 
 \begin{rem}
In the construction, we implicitly use a labeling of the graph $\Theta$. We can construct a cycle even when we change this label so that $\Theta$ has one type-(I) vertex and one type-(II) vertex. 
We can show that the result of pairing does not depend,  up to signs, on this choice of labeling.  
 \end{rem}
 
 \begin{rem}
The chord diagram $D(\Theta(p,q,r))$ has $p+q+r$ oriented lines and $k = p+q+r+1$ chords. As we recalled in Section \ref{Preliminaries for construction of cycles},  this chord diagram corresponds to a set of planetary systems which has $p+q+r$ planetary systems. In this set, there are two planetary systems which have two planetary orbits. See Figure~\ref{planetarysystem2loop}. 
 \end{rem}


\begin{figure}[htpb]
\begin{center}
\tikzset{every picture/.style={line width=1pt, xscale=0.3pt, yscale = 0.3pt}}  
\begin{tikzpicture}

\draw (0,-3) rectangle (48, 8);

\draw (3,3) circle(2) [line width = 2pt] [color = {rgb, 255:red, 0; green, 0; blue, 0 }, fill opacity =1.0];
\draw (3,3) circle(1.5) [line width = 2pt] [color = {rgb, 255:red, 0; green, 0; blue, 70 }, fill opacity =1.0];
\draw (4.5, 3) circle (0.2) [fill={rgb, 255:red, 0; green, 0; blue, 0}, fill opacity =1.0] ; 
\draw (3,3) circle(0.3) [fill={rgb, 255:red, 0; green, 0; blue, 0}, fill opacity =1.0] ; 

\draw (9,3) circle(0.3) [fill={rgb, 255:red, 0; green, 0; blue, 0}, fill opacity =1.0]  ; 
\draw (9,3) circle(1.5) [line width = 2pt] [color = {rgb, 255:red, 0; green, 0; blue, 70 }, fill opacity =1.0];

\draw (12,3) node {$\cdots$};

\draw (15,3) circle(0.3) [fill={rgb, 255:red, 0; green, 0; blue, 0}, fill opacity =1.0]  ; 
\draw (15,3) circle(1.5) [line width = 2pt] [color = {rgb, 255:red, 0; green, 0; blue, 70 }, fill opacity =1.0];

\draw (3,0)..controls (9,-1).. (15,0);
\draw (9,-2) node {$p$};

\draw (21,3) circle(0.3) [fill={rgb, 255:red, 0; green, 0; blue, 0}, fill opacity =1.0]  ; 
\draw (21,3) circle(1.5) [line width = 2pt] [color = {rgb, 255:red, 0; green, 0; blue, 70 }, fill opacity =1.0];

\draw (24,3) node {$\cdots$};

\draw (27,3) circle(0.3) [fill={rgb, 255:red, 0; green, 0; blue, 0}, fill opacity =1.0]  ; 
\draw (27,3) circle(1.5) [line width = 2pt] [color = {rgb, 255:red, 0; green, 0; blue, 70 }, fill opacity =1.0];

\draw (21,0)..controls (24,-1).. (27,0);
\draw (24,-2) node {$q$};

\draw (33,3) circle(0.3) [fill={rgb, 255:red, 0; green, 0; blue, 0}, fill opacity =1.0]  ; 
\draw (33,3) circle(1.5) [line width = 2pt] [color = {rgb, 255:red, 0; green, 0; blue, 70 }, fill opacity =1.0];
\draw (33,3) circle(2) [line width = 2pt] [color = {rgb, 255:red, 0; green, 0; blue, 70 }, fill opacity =1.0];

\draw (37,3) node {$\cdots$};

\draw (39,3) circle(0.3) [fill={rgb, 255:red, 0; green, 0; blue, 0}, fill opacity =1.0]  ; 
\draw (39,3) circle(1.5) [line width = 2pt] [color = {rgb, 255:red, 0; green, 0; blue, 70 }, fill opacity =1.0];

\draw (45,3) circle(0.3) [fill={rgb, 255:red, 0; green, 0; blue, 0}, fill opacity =1.0]  ; 
\draw (45,3) circle(1.5) [line width = 2pt] [color = {rgb, 255:red, 0; green, 0; blue, 70 }, fill opacity =1.0];

\draw (33,0)..controls (39,-1).. (45,0);
\draw (39,-2) node {$r$};

\end{tikzpicture}

\caption{The set of planetary systems corresponding to $\Theta(p,q,r)$}
\label{planetarysystem2loop}
\end{center}
\end{figure}

\subsection{The ribbon presentations $P(\Theta(p,q,r))$}
We see how a ribbon presentation $P(\Theta(p,q,r))$ with one node is obtained from the diagram $D(\Theta(p,q,r))$. This construction is already explained in Definition \ref{ribbon presentation from diagram}. 

\begin{itemize}
\item Replace 
\begin{tikzpicture}[x=0.75pt,y=0.75pt,yscale=0.15pt, xscale=0.15pt, baseline= -3pt] 

\draw  [-Stealth, line width =1pt, color ={rgb, 255: red, 0; green, 0; blue, 255 }] (0,0)--(500,0);

\draw  [-Stealth, dash pattern = on 2pt off 3 pt, line width =1pt] (100,0)--(100,150);
\draw  [Stealth-, dash pattern = on 2pt off 3 pt, line width =1pt] (300,0)--(300,150);

\draw  [fill={rgb, 255:red, 0; green, 0; blue, 0 }  ,fill opacity=1 ]  (100,0) circle (20);
\draw  [fill={rgb, 255:red, 0; green, 0; blue, 0 }  ,fill opacity=1 ] (300,0) circle (20);
\end{tikzpicture}
and
\begin{tikzpicture}[x=0.75pt,y=0.75pt,yscale=0.15,xscale=0.15, baseline=-3pt] 

\draw  [-Stealth, line width =1pt, color ={rgb, 255: red, 0; green, 0; blue, 255 }] (0,0)--(700,0);

\draw  [-Stealth, dash pattern = on 2pt off 3 pt, line width =1pt] (100,0)--(100,150);
\draw  [Stealth-, dash pattern = on 2pt off 3 pt, line width =1pt] (300,0)--(300,150);
\draw  [Stealth-, dash pattern = on 2pt off 3 pt, line width =1pt] (500,0)--(500,150);

\draw  [fill={rgb, 255:red, 0; green, 0; blue, 0 }  ,fill opacity=1 ]  (100,0) circle (20);
\draw  [fill={rgb, 255:red, 0; green, 0; blue, 0 }  ,fill opacity=1 ] (300,0) circle (20);
\draw  [fill={rgb, 255:red, 0; green, 0; blue, 0 }  ,fill opacity=1 ] (500,0) circle (20);
\end{tikzpicture}
with 
\begin{tikzpicture}[x=0.75pt,y=0.75pt,yscale=0.1,xscale=0.1, baseline=-6pt] 

\draw  [line width =1pt] (35,0)--(300,0);
\draw  [line width =1pt] (35,-80)--(300,-80);

\draw  [line width =1pt]  (-30,-40) circle (80);

\end{tikzpicture},
where the vertex with an outgoing (resp. ingoing) open chord is replaced by a disk (resp. a segment of a band). Exceptionally, replace 
\begin{tikzpicture}[x=0.75pt,y=0.75pt,yscale=0.15,xscale=0.15, baseline=-3pt] 

\draw  [-Stealth, line width =1pt, color ={rgb, 255: red, 0; green, 0; blue, 255 }] (0,0)--(700,0);

\draw  [-Stealth, dash pattern = on 2pt off 3 pt, line width =1pt] (100,0)--(100,150);
\draw  [-Stealth, dash pattern = on 2pt off 3 pt, line width =1pt] (300,0)--(300,150);
\draw  [Stealth-, dash pattern = on 2pt off 3 pt, line width =1pt] (500,0)--(500,150);

\draw  [fill={rgb, 255:red, 0; green, 0; blue, 0 }  ,fill opacity=1 ]  (100,0) circle (20);
\draw  [fill={rgb, 255:red, 0; green, 0; blue, 0 }  ,fill opacity=1 ] (300,0) circle (20);
\draw  [fill={rgb, 255:red, 0; green, 0; blue, 0 }  ,fill opacity=1 ] (500,0) circle (20);
\end{tikzpicture}, which has two outgoing open chords, with 
\begin{tikzpicture}[x=0.75pt,y=0.75pt,yscale=0.1,xscale=0.1, baseline=-6pt] 

\draw  [line width =1pt] (45,120)--(200,20);
\draw  [line width =1pt] (40,60)--(190,-40);

\draw  [line width =1pt] (45,-200)--(200,-100);
\draw  [line width =1pt] (40,-140)--(190,-40);

\draw  [line width =1pt] (335,0)--(475,0);
\draw  [line width =1pt] (335,-80)--(475,-80);

\draw  [line width =1pt]  (-30,-180) circle (80);
\draw (-200, -180) node {$(2)$};
\draw  [line width =1pt]  (-30,100) circle (80);
\draw (-200, 100) node {$(1)$};
\draw  [line width =1pt, fill = gray]  (265,-40) circle (80);

\end{tikzpicture}. Again, the two vertices with an outgoing open chord are replaced by (white) disks, where the tail vertex is replaced with the disk with the label $(1)$. The vertex with an ingoing open chord is replaced by a segment of a band. 
The disk with three bands (drawn in gray) is the node of our ribbon presentation. 
\end{itemize}

\begin{itemize}
\item Intersect a disk with a band if they are connected by chords. Assign the label $\star$ to this crossing. This label is a sign of perturbation of a crossing which we later define. 

\item Connect the free ends of bands to the based disk. The order of bands is arbitrary (because of S7 move in Figure \ref{exofmoves}). This connection determines the orientations of crossings. We connect the ends to the based disk so that the orientation of each crossing is positive.
\end{itemize}

\begin{figure}[htpb]
\centering
    \includegraphics [width =7cm] {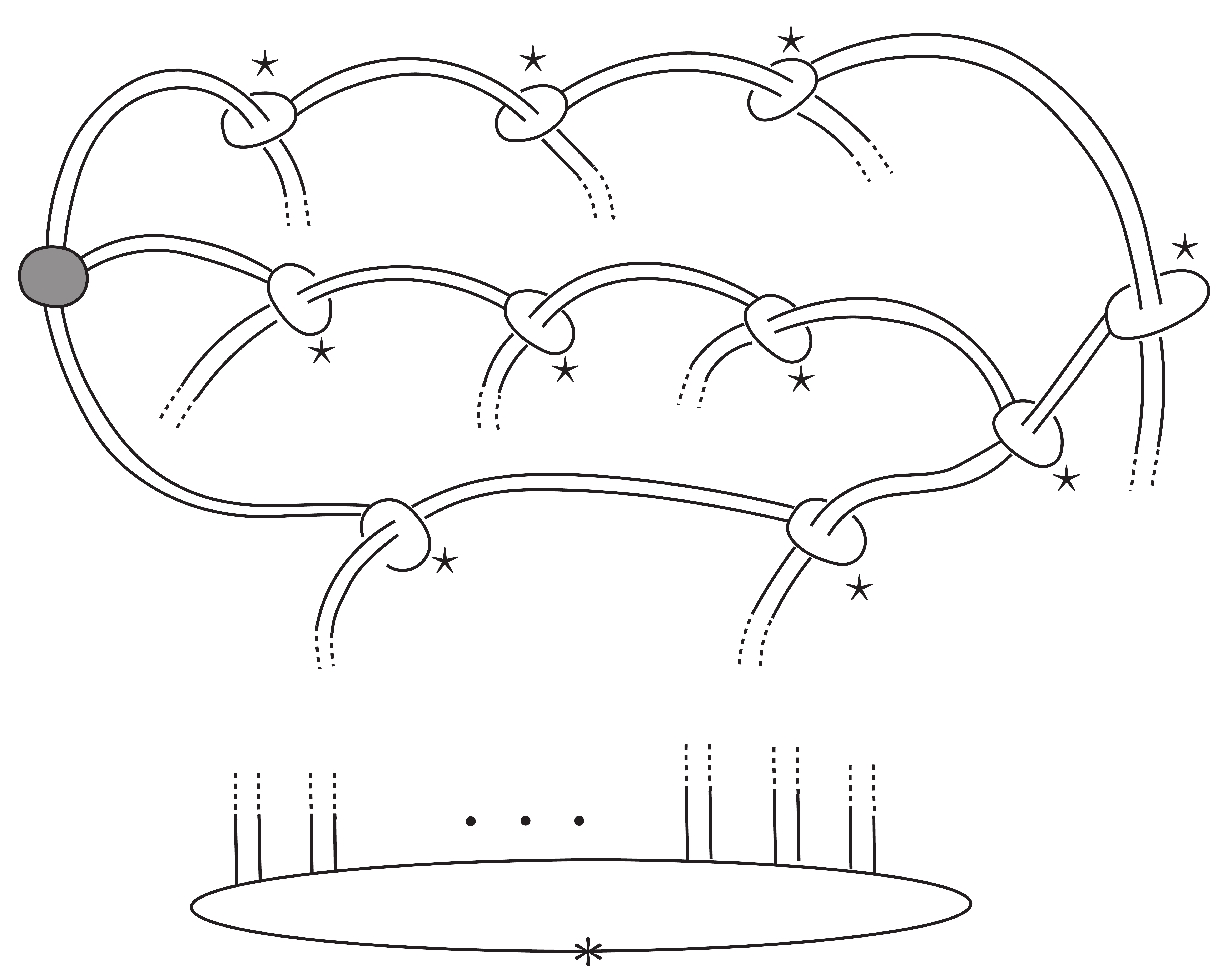}
    \caption{The ribbon presentation $P(\Theta(4,3,2))$)}
 \end{figure}
 

\subsection{The $2$-loop cycles $d(\Theta(p,q,r))$ of $\overline{\mathcal{K}}_{n,j}$}
Here, we construct cycles $d(\Theta(p,q,r)): (S^{n-j-2})^k \times S^{j-1} \rightarrow \overline{\mathcal{K}}_{n,j}$. The parameter space $(S^{n-j-2})^k$ arises from perturbation of crossings. For each parameter $\mathbf{v} = (v_1,\dots, v_k) \in (S^{n-j-2})^k$, we can give the new ribbon presentation $P_\mathbf{v}$ by 
\[
P_\mathbf{v} = \mathcal{D} \cup \mathcal{B}(\mathbf{v}) = \bigcup D_i \cup \bigcup B_j(v_1, v_2, \dots, v_k),
\]
where $B_j(v_1, v_2, \dots, v_k)$ is the band obtained by perturbating all crossings $c_k$ of $B_j$ to the direction $v_k$.  See Section \ref{Preliminaries for construction of cycles}. 

\begin{definition} 
The cycle $c(\Theta(p,q,r)): (S^{n-j-2})^k \rightarrow \mathcal{K}_{n,j}$ is defined by 
\[
\mathbf{v} \longmapsto \psi(P_{\mathbf{v}}) =  \partial V_{P_\mathbf{v}} \# \iota (\mathbb{R}^j).
\]
\end{definition} 

Recall that our ribbon presentation has a node which has two bands, say $B_i$ and $B_j$, connected to leaves 
\begin{tikzpicture}[x=0.75pt,y=0.75pt,yscale=0.1,xscale=0.1, baseline=-6pt] 

\draw  [line width =1pt] (45,120)--(200,20);
\draw  [line width =1pt] (40,60)--(190,-40);

\draw  [line width =1pt] (45,-200)--(200,-100);
\draw  [line width =1pt] (40,-140)--(190,-40);

\draw  [line width =1pt] (335,0)--(475,0);
\draw  [line width =1pt] (335,-80)--(475,-80);

\draw  [line width =1pt]  (-30,-180) circle (80);
\draw (120, -240) node {$B_j$};
\draw  [line width =1pt]  (-30,100) circle (80);
\draw (120, 160) node {$B_i$};

\draw  [line width =1pt, fill = gray]  (265,-40) circle (80);

\end{tikzpicture}.
As a long embedding, this part corresponds to a punctured sphere with two tubes $\hat{B_i}$ and $\hat{B_j}$ attached. 
The additional $S^{j-1}$ family is given by moving  one tube ($\hat{B_j}$) around the other tube ($\hat{B_i}$). In terms of planetary systems in Figure \ref{planetarysystem2loop}, we turn the planet in the second orbit of the leftmost planetary system around the fixed star. See Figure \ref{additionalfamily}. Then we obtain a $(S^{n-j-2})^k \times S^{j-1}$ cycle of submanifolds in $\mathbb{R}^n$. 

Note that, as (images of) immersions, there is a path to the standard immersion. Using this path, we obtain the desired cycle
\[
d(\Theta(p,q,r)) : (S^{n-j-2})^k \times S^{j-1}  \rightarrow \overline{\mathcal{K}}_{n,j}.
\]

\begin{figure}[htpb]
   \centering
    \includegraphics [width = 11cm] {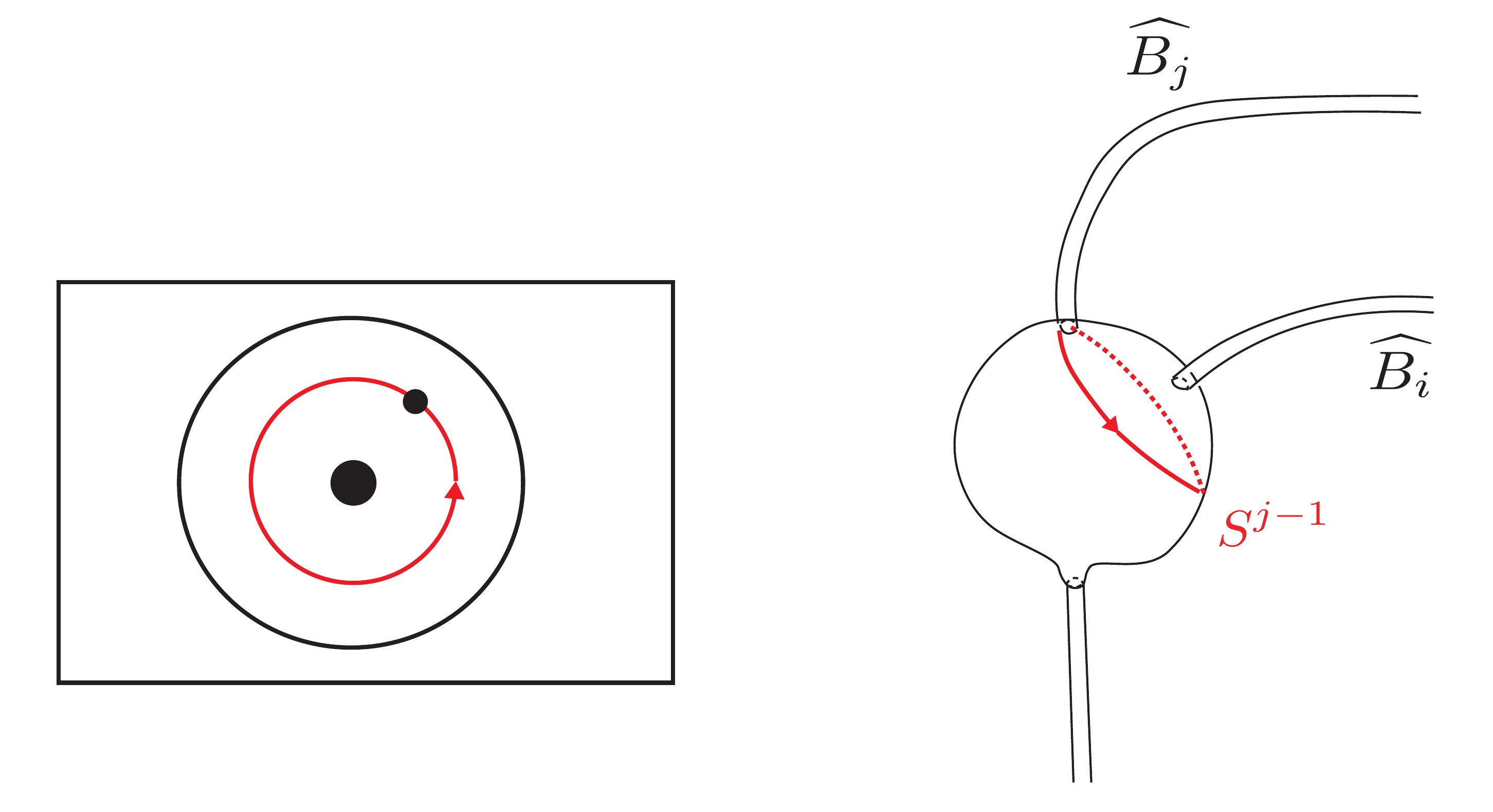}
    \caption{The additional $S^{j-1}$ family}
    \label{additionalfamily}
 \end{figure}

\subsection{Properties of the ribbon presentations $P(\Theta(p,q,r))$}
\label{Properties of ribbon presentation}
We want to take the cycles from the unknot component. (See Notation \ref{def unknot component} for the definition of the unknot component of $\mathcal{K}_{n,j}$ and $\overline{\mathcal{K}}_{n,j}$.)
  When $n-j \geq 3$, the cycles $d(\Theta(p,q,r))$ are already in the unknot component. 
When $n-j = 2$, we must make more efforts. 
We introduce an important property for this purpose and introduce a modified ribbon presentation $P^{\prime}(\Theta(p,q,r))$
 \begin{notation}
Let $\varepsilon_i = \pm 1$. 
Write $P(\Theta(p,q,r)) (\varepsilon_1, \varepsilon_2, \dots,\varepsilon_k)$ for the ribbon presentation obtained by changing
the $j$th crossing
\begin{tikzpicture}[x=0.5pt,y=0.5pt,yscale=0.1,xscale=0.1, baseline=-6pt] 
 \draw [line width =1pt, fill = white, fill opacity =1] (-100,0) rectangle (100,-300) ;
 \draw [line width =1pt, fill = white, fill opacity =1]  (0,0) circle (220 and 140);
 \draw [line width =1pt, fill = white, fill opacity =1] (-100,300) rectangle (100,0) ;
\draw (300, 0) node  {$\star$};
\end{tikzpicture}
to 
\begin{tikzpicture}[x=0.5pt,y=0.5pt,yscale=0.1,xscale=0.1, baseline=-6pt] 
 \draw [line width =1pt, fill = white, fill opacity =1] (-100,0) rectangle (100,-300) ;
 \draw [line width =1pt, fill = white, fill opacity =1]  (0,0) circle (220 and 140);
 \draw [line width =1pt, fill = white, fill opacity =1] (-100,300) rectangle (100,0) ;
\end{tikzpicture}
(without $\star$) when $\varepsilon_j = 1$, and to 
\begin{tikzpicture}[x=0.5pt,y=0.5pt,yscale=0.1,xscale=0.1, baseline=-6pt] 
 \draw [line width =1pt, fill = white, fill opacity =1] (-100,300) rectangle (100,-300) ;
 \draw [line width =1pt, fill = white, fill opacity =1]  (-400,0) circle (220 and 140);
\end{tikzpicture}
when $\varepsilon_j = -1$.
\end{notation}

\begin{notation}
We define cross-change moves of ribbon presentations as in Figure \ref{crosschangemove}. That is, we make a copy disk of a node by $S1$ move and intersect this disk with some band.   Note that we do not put the label $\star$ to the new crossing. In other words, we do not perform perturbation to this crossing. Note that this move is meaningless when $n-j \geq 3$. 
\begin{figure}[htpb]
   \centering
    \includegraphics [width =8cm] {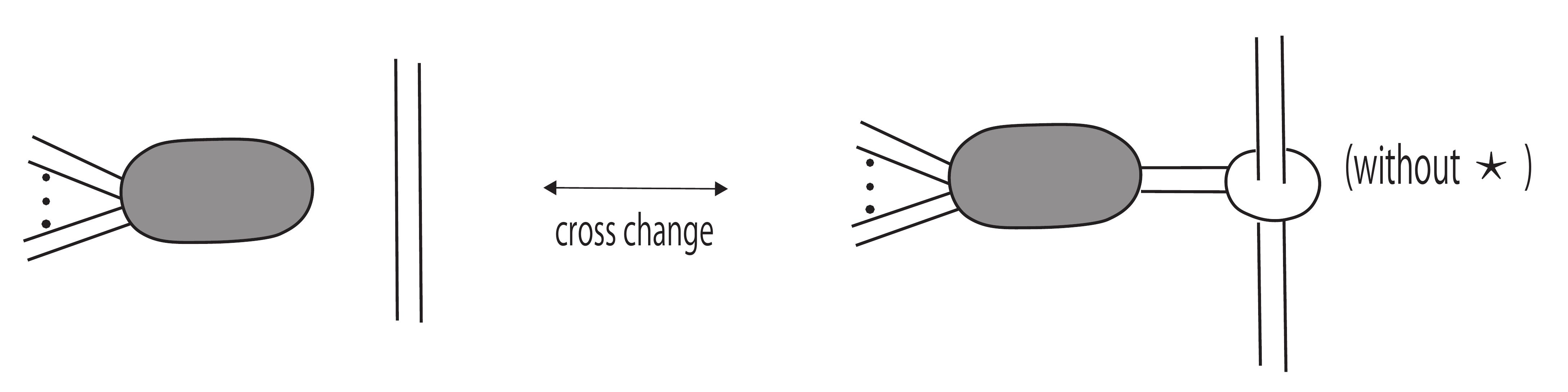}
    \caption{Cross-change move}
    \label{crosschangemove}
\end{figure}
\end{notation}

\begin{prop}
\label{propcrosschangemove}
After several cross-change moves are performed to $P(\Theta(p,q,r))$, the presentation $P(\Theta(p,q,r))(1,1, \dots,1)$ is equivalent to the trivial presentation.
\end{prop}
\begin{proof}
We perform cross-change moves as in Figure \ref{crosschangedpresentation}. Namely, we add the new red and blue branches from the node. They play a role in resolving the two white disks  of $Q = $
\begin{tikzpicture}[x=0.75pt,y=0.75pt,yscale=0.1,xscale=0.1, baseline=-6pt] 

\draw  [line width =1pt] (45,120)--(200,20);
\draw  [line width =1pt] (40,60)--(190,-40);

\draw  [line width =1pt] (45,-200)--(200,-100);
\draw  [line width =1pt] (40,-140)--(190,-40);

\draw  [line width =1pt] (335,0)--(475,0);
\draw  [line width =1pt] (335,-80)--(475,-80);

\draw  [line width =1pt]  (-30,-180) circle (80);
\draw  [line width =1pt]  (-30,100) circle (80);
\draw  [line width =1pt, fill = gray]  (265,-40) circle (80);

\end{tikzpicture}
by $S4^{\prime}$ move. See Figure \ref{crosschangedpresentation2}. Then, the resulting presentation becomes trivial.

\begin{figure}[htpb]
\centering
     \includegraphics [width =8cm] {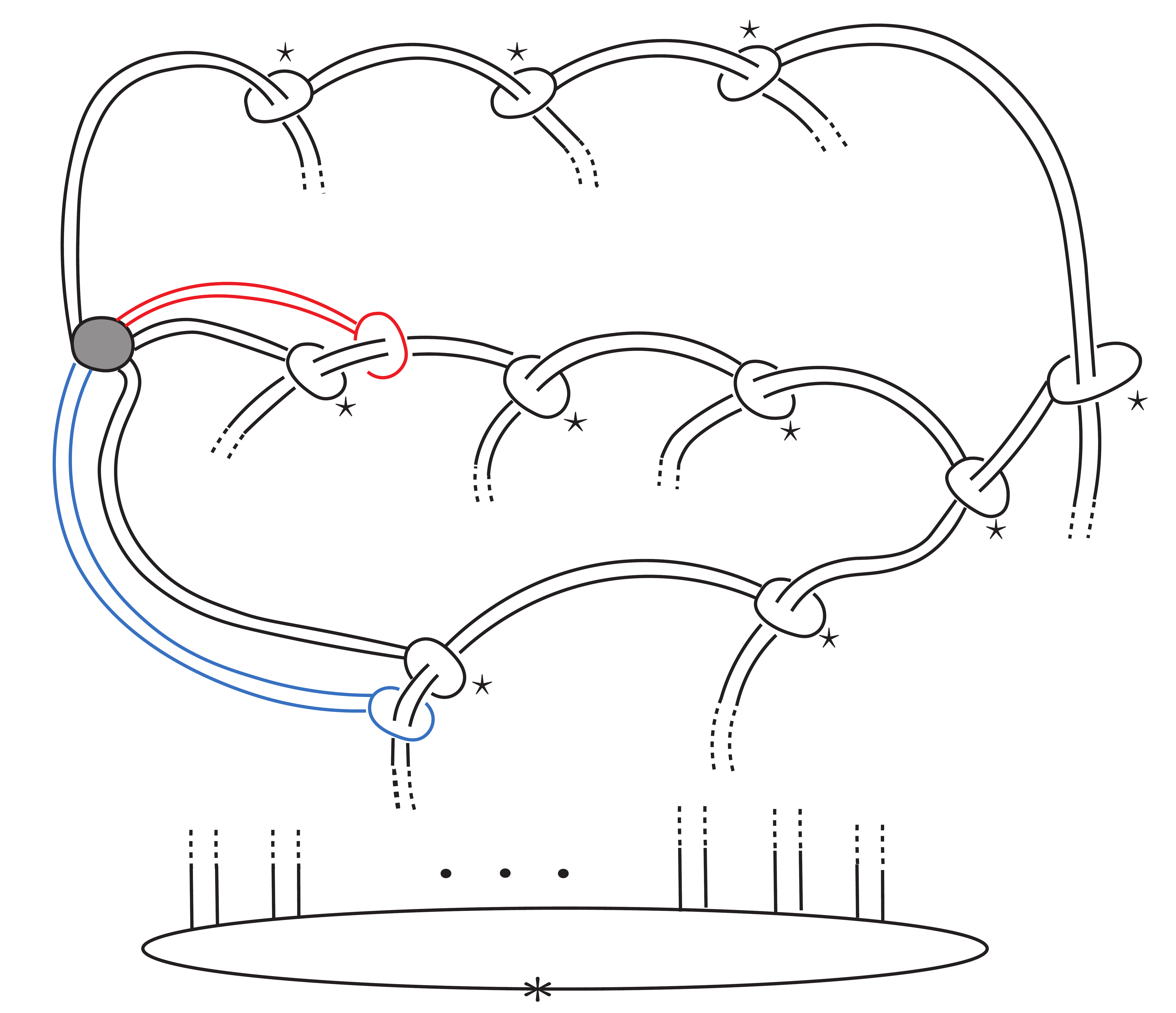}
     \caption{The ribbon presentation  $P^{\prime}(\Theta(p,q,r))$ after cross-change moves}
     \label{crosschangedpresentation}
\end{figure}

\begin{figure}[htpb]
\centering
     \includegraphics [width =13cm] {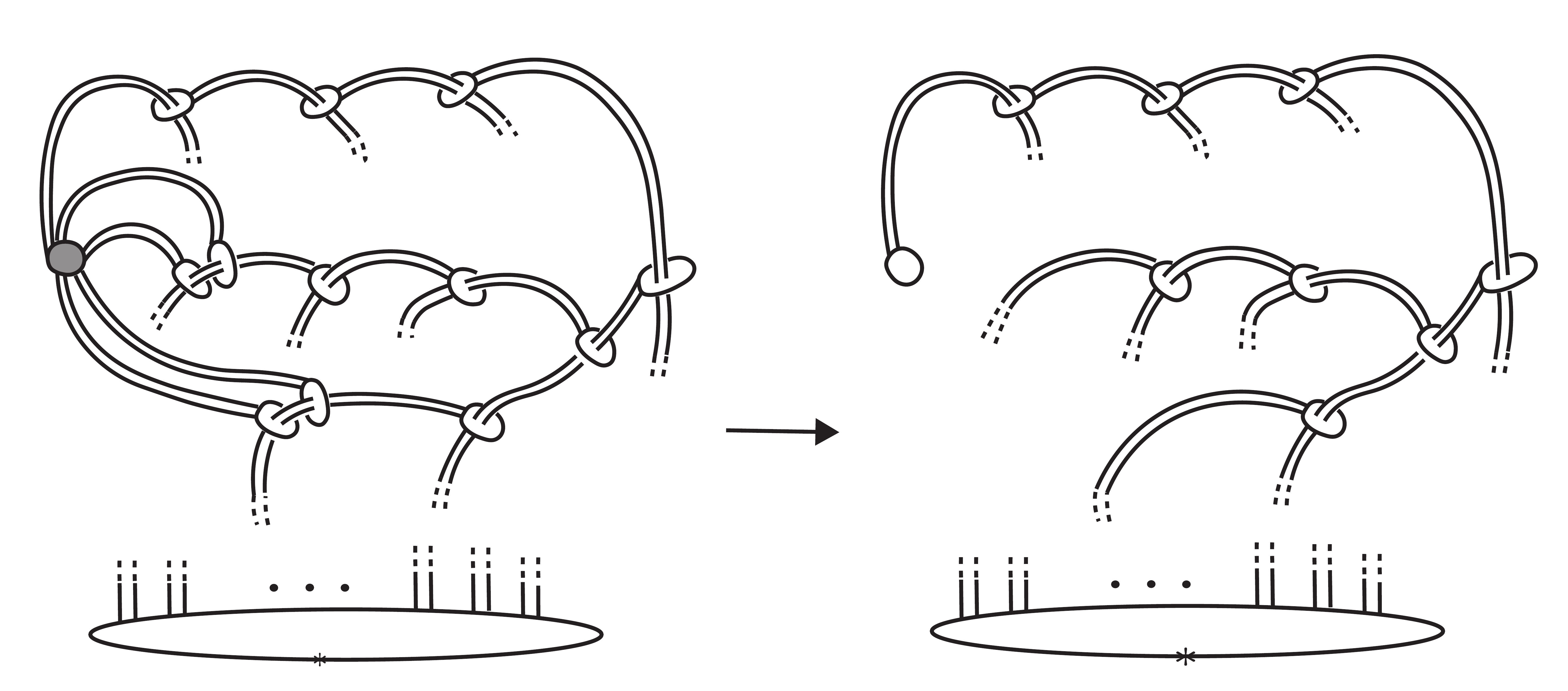}
     \caption{The ribbon presentation before and after $S4^{\prime}$ move}
     \label{crosschangedpresentation2}
\end{figure}

\end{proof}

We write $P^{\prime}(\Theta(p,q,r))$ for the ribbon presentation after the cross-change moves.
Note that the cross-change move might change the homology classes of the cycles. However, the move does not affect the pairing argument we later discuss. 


\begin{prop}
\label{presentation spherical cycles}
If at least one of $\varepsilon_i$ is $-1$, $P(\Theta(p,q,r)) (\varepsilon_1, \varepsilon_2, \dots,\varepsilon_k)$ and $P^{\prime}(\Theta(p,q,r)) (\varepsilon_1, \varepsilon_2, \dots,\varepsilon_k)$ are equivalent to a presentation such that at least one of the two white disks of  $Q = $
\begin{tikzpicture}[x=0.75pt,y=0.75pt,yscale=0.1,xscale=0.1, baseline=-6pt] 

\draw  [line width =1pt] (45,120)--(200,20);
\draw  [line width =1pt] (40,60)--(190,-40);

\draw  [line width =1pt] (45,-200)--(200,-100);
\draw  [line width =1pt] (40,-140)--(190,-40);

\draw  [line width =1pt] (335,0)--(475,0);
\draw  [line width =1pt] (335,-80)--(475,-80);

\draw  [line width =1pt]  (-30,-180) circle (80);
\draw  [line width =1pt]  (-30,100) circle (80);
\draw  [line width =1pt, fill = gray]  (265,-40) circle (80);

\end{tikzpicture}
is a node. 
\end{prop}

\begin{proof}
Let $\varepsilon_i = -1$ for some $i \in \{1, \dots, k\}$. Let $B_i$ be the branch (band) of $P(\Theta(p,q,r))$ or $P^{\prime}(\Theta(p,q,r))$ which is attached to a disk corresponding to the $i$-th crossing. 

First, if $B_i$ is the upper left or lower left branch of the part $Q =$
\begin{tikzpicture}[x=0.75pt,y=0.75pt,yscale=0.1,xscale=0.1, baseline=-6pt] 

\draw  [line width =1pt] (45,120)--(200,20);
\draw  [line width =1pt] (40,60)--(190,-40);

\draw  [line width =1pt] (45,-200)--(200,-100);
\draw  [line width =1pt] (40,-140)--(190,-40);

\draw  [line width =1pt] (335,0)--(475,0);
\draw  [line width =1pt] (335,-80)--(475,-80);

\draw  [line width =1pt]  (-30,-180) circle (80);
\draw  [line width =1pt]  (-30,100) circle (80);
\draw  [line width =1pt, fill = gray]  (265,-40) circle (80);

\end{tikzpicture}, there is nothing to show. 
Second, assume that the branch $B_i$ belongs to the middle (resp. lower) edge of $\Theta(p,q,r)$. The branch has a crossing with the disk attached to the next (left) branch. However, the crossing is resolved since the disk attached to $B_i$ has no crossing. We repeat this process. Then, all the crossings of the branches after $B_i$ along the middle (resp. lower) edge are resolved. Consequently, either the upper left disk or the lower left disk of $Q$ is resolved. Then,  the argument is reduced to the first case.  See Figure \ref{equivalenttotrivial}. Finally, assume $B_i$ belongs to the upper edge of $\Theta(p,q,r)$.  Then, similarly to the second case, the disk attached to the rightmost branch of the lower edge of $\Theta$ is resolved. Hence, the argument is reduced to the second case. 
 \end{proof}
 
 \begin{figure}[htpb]
\centering
    \includegraphics [width =12cm] {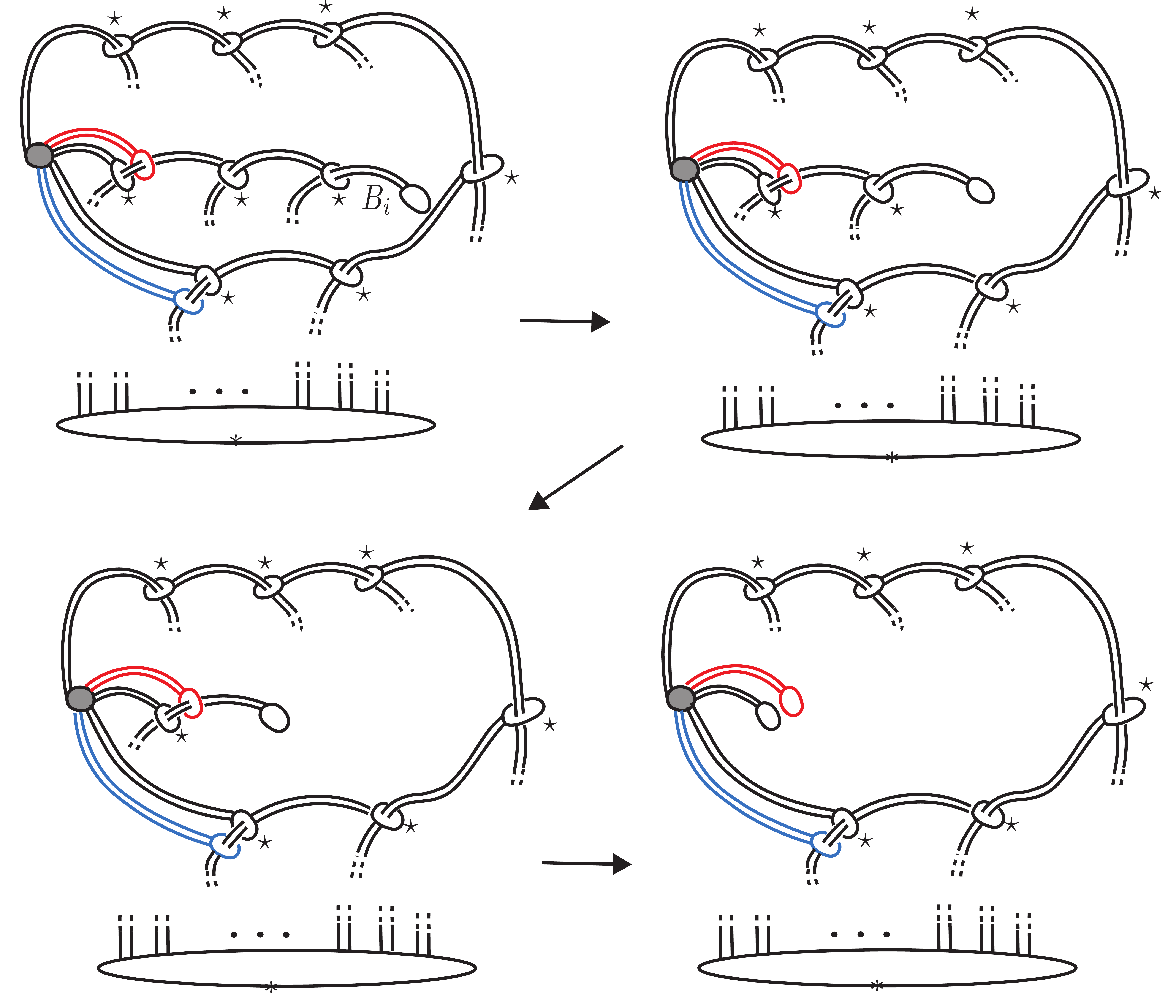}
      \caption{$P^{\prime}(\Theta(p,q,r)) (\varepsilon_1, \varepsilon_2, \dots,\varepsilon_k)$ is trivial if some $\varepsilon_i = -1$}
       \label{equivalenttotrivial}
\end{figure}

\subsection{The class $\mathcal{H}_{n,j}(g=2)$ of $2$-loop cycles}
\label{the class of 2loop cycles}
Let $d^{\prime}(\Theta(p,q,r))$ be the cycle obtained from the ribbon presentation $P^{\prime}(\Theta(p,q,r))$ (when $n-j-2$). 

\begin{definition}
We write $\mathcal{H}_{n,j}(k,g=2)$ for the subspace of $H_{k(n-j-2)+(j-1)}(\overline{\mathcal{K}}_{n,j}, \mathbb{R})$ generated by the cycles $d^{\prime}(\Theta(p,q,r))$ which satisfy $k = p+q+r+1, p, r \geq 1, q\geq 0$. 
We write $\mathcal{H}_{n,j}(g=2)$ for the subspace of $H_{\ast}(\overline{\mathcal{K}}_{n,j}, \mathbb{R})$  generated by the cycles $d^{\prime}(\Theta(p,q,r))$ which satisfy $p, r \geq 1, q\geq 0$. 
\end{definition}

Note that when $n-j = 2$, all the cycles $d^{\prime}(\Theta(p,q,r))$ lie in degree $j-1$. 
Below, we show that $\mathcal{H}_{n,j}(g=2)$ lies in the homotopy group when $n-j=2$. 

\begin{prop}
Let $n-j = 2$. 
Let 
\[
d(\Theta(p,q,r))(\varepsilon_1, \varepsilon_2, \dots, \varepsilon_k) : S^{j-1} \rightarrow \overline{\mathcal{K}}_{n,j}
\] 
be the cycle corresponding to the ribbon presentation $P(\Theta(p,q,r))(\varepsilon_1, \varepsilon_2, \dots, \varepsilon_k)$. Note that 
\[
d(\Theta(p,q,r)) = \sum_{(\varepsilon_1, \varepsilon_2, \dots, \varepsilon_k) \in \{+1, -1\}^{\times k}}  \varepsilon_1 \varepsilon_2 \cdots \varepsilon_k\ d(\Theta(p,q,r))(\varepsilon_1, \varepsilon_2, \dots, \varepsilon_k).
\]
Then, if at least one of $\varepsilon_i$ is $-1$, the cycle $d(\Theta(p,q,r))(\varepsilon_1, \varepsilon_2, \dots, \varepsilon_k)$ is a degenerate cycle. More precisely, the cycle is homotopic to a constant function. A similar result also holds for $d^{\prime}(\Theta(p,q,r))$.
\end{prop}

\begin{proof}
By Proposition \ref{presentation spherical cycles}, either the upper left disk or the lower left disk of 
\begin{tikzpicture}[x=0.75pt,y=0.75pt,yscale=0.1,xscale=0.1, baseline=-6pt] 

\draw  [line width =1pt] (45,120)--(200,20);
\draw  [line width =1pt] (40,60)--(190,-40);

\draw  [line width =1pt] (45,-200)--(200,-100);
\draw  [line width =1pt] (40,-140)--(190,-40);

\draw  [line width =1pt] (335,0)--(475,0);
\draw  [line width =1pt] (335,-80)--(475,-80);

\draw  [line width =1pt]  (-30,-180) circle (80);
\draw  [line width =1pt]  (-30,100) circle (80);
\draw  [line width =1pt, fill = gray]  (265,-40) circle (80);
\end{tikzpicture}
of the presentation $P(\Theta(p,q,r))(\varepsilon_1, \varepsilon_2, \dots, \varepsilon_k)$ does not have a crossing. Then, one of the two tubes of Figure \ref{additionalfamily} which gives the $S^{j-1}$ family, vanishes. Hence, the corresponding cycle is homotopic to a constant function.
\end{proof}

\begin{coro}
\label{from homotopy group}
Let $n-j =2$. Then $d(\Theta(p,q,r))$ is homologous to the cycle
\[
d(\Theta(p,q,r))((\varepsilon_i = 1)_i): S^{j-1} \rightarrow \overline{\mathcal{K}}_{n,j}.
\]
A similar result also holds for $d^{\prime}(\Theta(p,q,r))$. Moreover, the cycle $d^{\prime}(\Theta(p,q,r))((\varepsilon_i = 1)_i)$ gives an element of the homotopy group  $\pi_{j-1} ((\overline{\mathcal{K}}_{j+2, j})_{\overline{\iota}})$.
\end{coro}
\begin{proof}
The cycle $d^{\prime}(\Theta(p,q,r))((\varepsilon_i = 1)_i)$ is in the unknot component because $P^{\prime}(\Theta(p,q,r))((\varepsilon_i = 1)_i)$ is equivalent to the trivial presentation. 
\end{proof}

\subsection{The class $\mathcal{H}_{n,j}(g=3)$ of $3$-loop cycles}
\label{3loopcycles}
We write $\ctext{Y}((p_i)_i) = \ctext{Y}(p_1, p_2, p_3, p_4, p_5, p_6)$ for the $3$-loop hairy graph
\[
\begin{tikzpicture}[x=5pt,y=5pt,yscale=0.09,xscale=0.09, baseline=-3pt] 
\draw (0,0) circle (100); 
\draw (0,0) -- (0,100);
\draw (0,0) -- (70,-70);
\draw (0,0) -- (-70, -70);

\draw (-120,0) node {$p_1$};
\draw (20,50) node {$p_6$};
\draw (55,-30) node {$p_4$};
\draw (-55,-30) node {$p_5$};
\draw (0,-120) node {$p_3$};
\draw (120,0) node {$p_2$};

\end{tikzpicture},
\]
whree $p_i$ stands for $p_i$ hairs on the edge. Note that the order of the graph $\ctext{Y}((p_i)_i)$ equals to $k = 2 + \sum_{i=1}^6 p_i$. 

We construct the  $k(n-j-2) + 2(j-1)$-cycles
\[
d(\ctext{Y}((p_i)_i)) : (S^{n-j-2})^k \times (S^{j-1})^2  \rightarrow \overline{\mathcal{K}}_{n,j}, 
\] 
when one of the following conditions is satisfied. 
\begin{itemize}
\item [(1)] $p_1, p_6, p_3, p_4 \geq 1$, 
\item [(2)] $p_1, p_6, p_4 \geq 1$ and $p_2, p_3, p_5 = 0$, 
\item [(3)] $p_1, p_6 \geq 1$ and $p_2, p_3, p_4, p_5 = 0$.
\end{itemize}

First, we give a diagram $D(\ctext{Y}((p_i)_i))$. Suppose condition (1) is satisfied. Orient the $3$-loop graph $\ctext{Y}$ as follows. 
\[
\begin{tikzpicture}[x=5pt,y=5pt,yscale=0.09,xscale=0.09, baseline=-3pt] 

\draw (0,0) circle (100); 

\draw [-Stealth] (-1,-1) -- (0,0); 
\draw [-Stealth] (1,-1) -- (0,0); 
\draw [-Stealth] (-71,-70) -- (-70,-71); 
\draw [-Stealth] (70,-71) -- (71,-70); 
\draw [- Stealth] (0,0) -- (0,100);
\draw [- Stealth] (1,100) -- (0,100);

\draw (0,0) -- (70,-70);
\draw (0,0) -- (-70, -70);

\draw (-120,0) node {$p_1$};
\draw (20,50) node {$p_6$};
\draw (55,-30) node {$p_4$};
\draw (-55,-30) node {$p_5$};
\draw (0,-120) node {$p_3$};
\draw (120,0) node {$p_2$};

\draw (0,130) node {$(I)$};
\draw (0,-40) node {$(I)$};
\draw (-110,-70) node {$(II)$};
\draw (110,-70) node {$(II)$};

\end{tikzpicture}
\]
Note that there are two type-(I) vertices and two type-(II) vertices.
We replace each hair with 
\begin{tikzpicture}[x=1.5pt,y=1.5pt,yscale=0.1,xscale=0.1, baseline=-3pt] 

\draw  [ color = {rgb, 255:red, 0; green, 0; blue, 255 }  ] [-Stealth, line width =1pt] (0,0)--(500,0);

\draw  [-Stealth, dash pattern = on 2pt off 3 pt, line width =1pt] (100,0)--(100,150);
\draw  [Stealth-, dash pattern = on 2pt off 3 pt, line width =1pt] (300,0)--(300,150);

\draw  [fill={rgb, 255:red, 0; green, 0; blue, 0 }  ,fill opacity=1 ]  (100,0) circle (20);
\draw  [fill={rgb, 255:red, 0; green, 0; blue, 0 }  ,fill opacity=1 ] (300,0) circle (20);
\end{tikzpicture}.
Exceptionally, we change the hairs at the tails of the first and sixth edges to 
\begin{tikzpicture}[x=1.5pt,y=1.5pt,yscale=0.1,xscale=0.1, baseline=-3pt] 

\draw   [color = {rgb, 255:red, 0; green, 0; blue, 255 }]  [-Stealth, line width =1pt] (0,0)--(700,0);

\draw  [-Stealth, dash pattern = on 2pt off 3 pt, line width =1pt] (100,0)--(100,150);
\draw  [-Stealth, dash pattern = on 2pt off 3 pt, line width =1pt] (300,0)--(300,150);
\draw  [Stealth-, dash pattern = on 2pt off 3 pt, line width =1pt] (500,0)--(500,150);

\draw  [fill={rgb, 255:red, 0; green, 0; blue, 0 }  ,fill opacity=1 ]  (100,0) circle (20);
\draw  [fill={rgb, 255:red, 0; green, 0; blue, 0 }  ,fill opacity=1 ] (300,0) circle (20);
\draw  [fill={rgb, 255:red, 0; green, 0; blue, 0 }  ,fill opacity=1 ] (500,0) circle (20);
\end{tikzpicture}.
We change the hairs at the tails of the fourth and fifth edges to 
\begin{tikzpicture}[x=1.5pt,y=1.5pt,yscale=0.1,xscale=0.1, baseline=-3pt] 

\draw   [color = {rgb, 255:red, 0; green, 0; blue, 255 }]  [-Stealth, line width =1pt] (0,0)--(700,0);

\draw  [-Stealth, dash pattern = on 2pt off 3 pt, line width =1pt] (100,0)--(100,150);
\draw  [Stealth-, dash pattern = on 2pt off 3 pt, line width =1pt] (300,0)--(300,150);
\draw  [Stealth-, dash pattern = on 2pt off 3 pt, line width =1pt] (500,0)--(500,150);

\draw  [fill={rgb, 255:red, 0; green, 0; blue, 0 }  ,fill opacity=1 ]  (100,0) circle (20);
\draw  [fill={rgb, 255:red, 0; green, 0; blue, 0 }  ,fill opacity=1 ] (300,0) circle (20);
\draw  [fill={rgb, 255:red, 0; green, 0; blue, 0 }  ,fill opacity=1 ] (500,0) circle (20);
\end{tikzpicture}.
Finally, connect ends of chords as expected from the graph $\ctext{Y}((p_i)_i)$. See Figure \ref{diagram2} for an example. 

 \begin{figure}[htpb]
\centering
    \includegraphics [width =10cm] {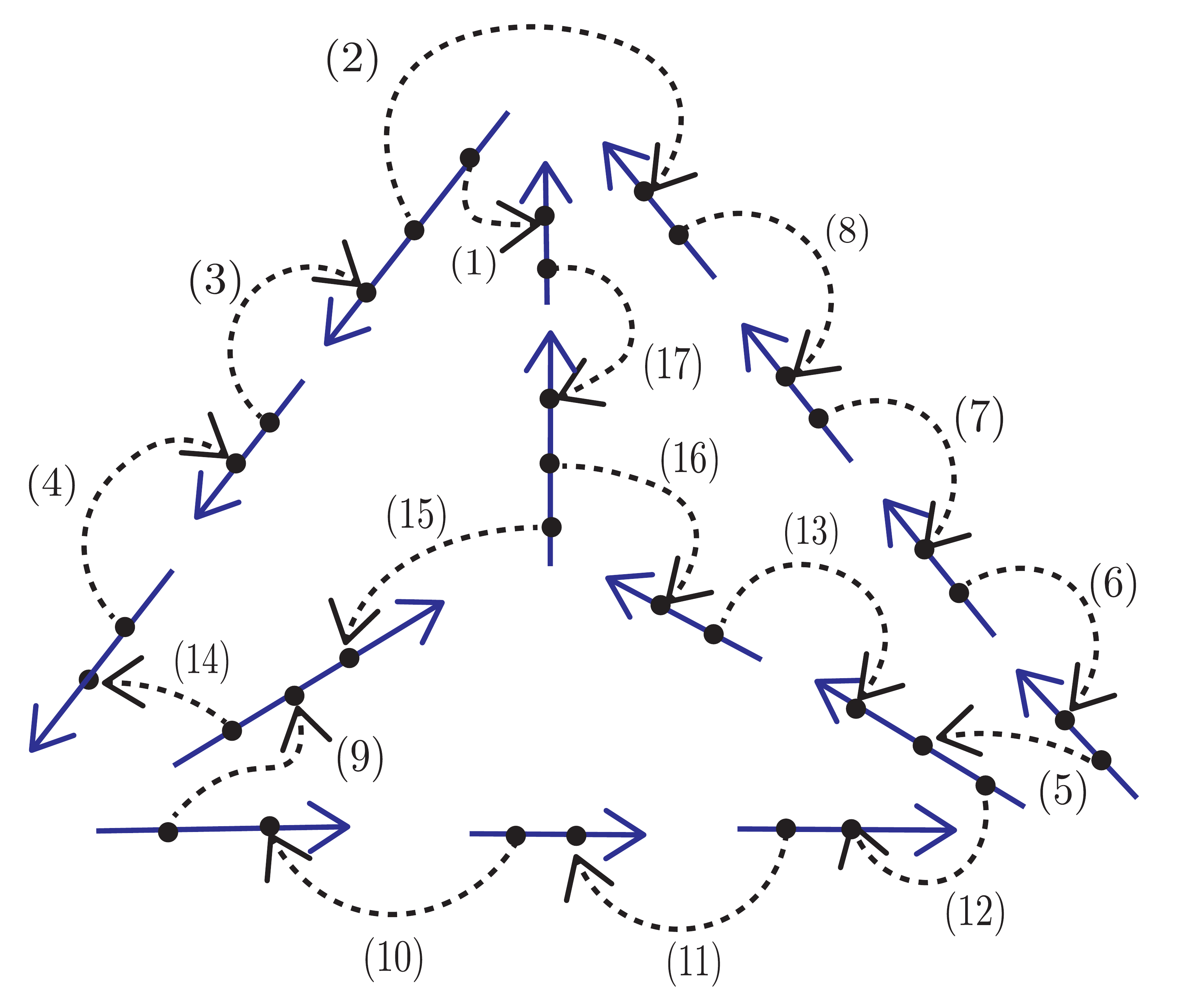}
      \caption{The diagram $D(\ctext{Y}(3, 4, 3, 2, 1, 2))$}
       \label{diagram2}
\end{figure}

 \begin{figure}[htpb]
\centering
    \includegraphics [width =12cm] {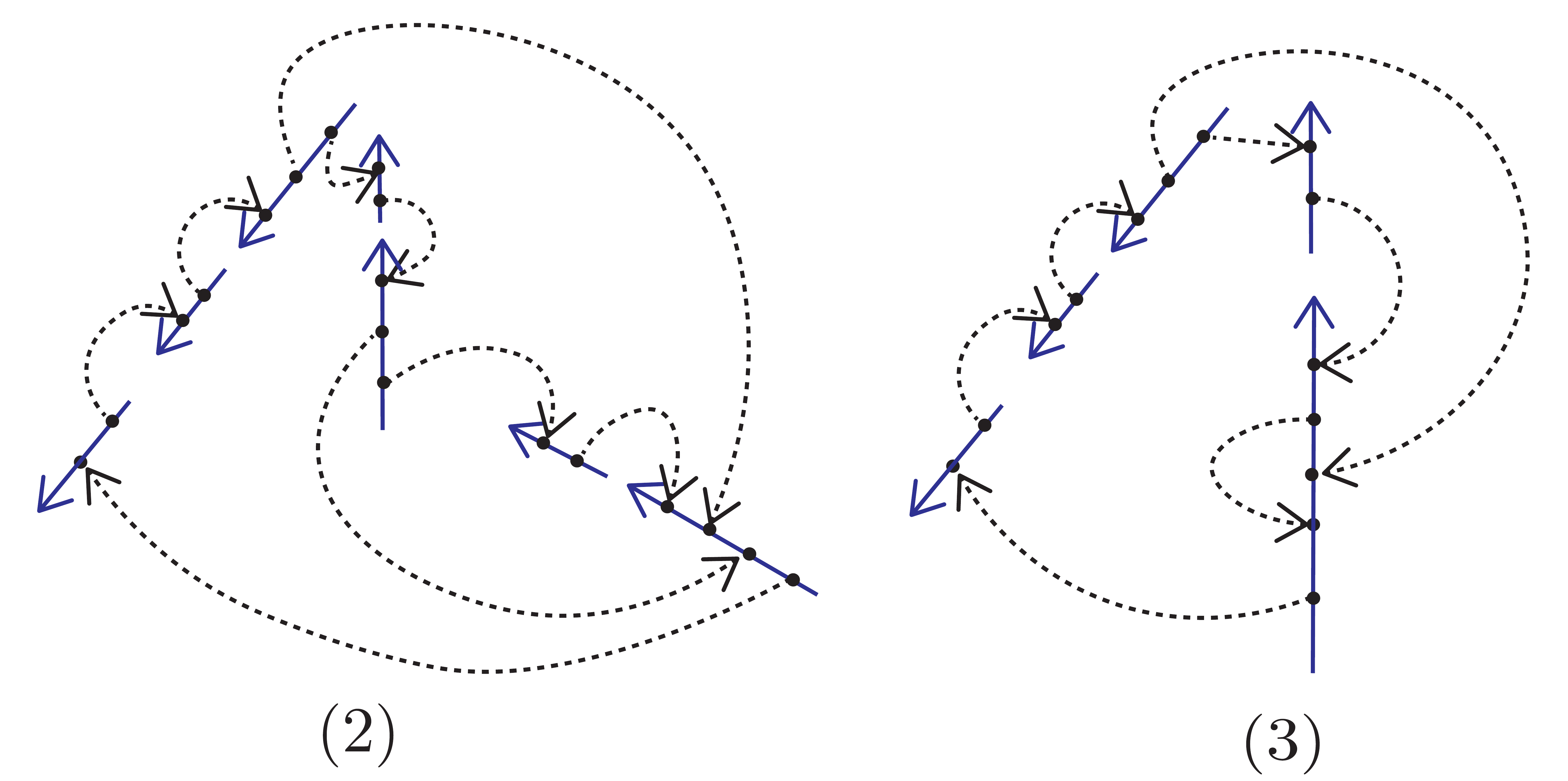}
     \caption{The diagrams $D(\ctext{Y}(3, 0, 0, 2, 0, 2))$ and $D(\ctext{Y}(3, 0, 0, 0, 0, 2))$}
       \label{diagram3}
  \end{figure}

Cases (2) and (3) are exceptional. We give chord diagrams whose oriented lines are along three and two edges of \ctext{Y}, respectively. See Figure \ref{diagram3} .

From the diagram $D(\ctext{Y}((p_i)_i))$, we can construct a ribbon presentation $P(\ctext{Y}((p_i)_i))$. 
This presentation has two nodes and $k$ crossings of bands and disks. See Figure \ref{3loop-presentation}. Hence the corresponding cycle $d(\ctext{Y}((p_i)_i))$ has the parameter space
\[
(S^{n-j-2})^k \times (S^{j-1})^2.
\]

\begin{figure}[htpb]
\centering
    \includegraphics [width =9cm] {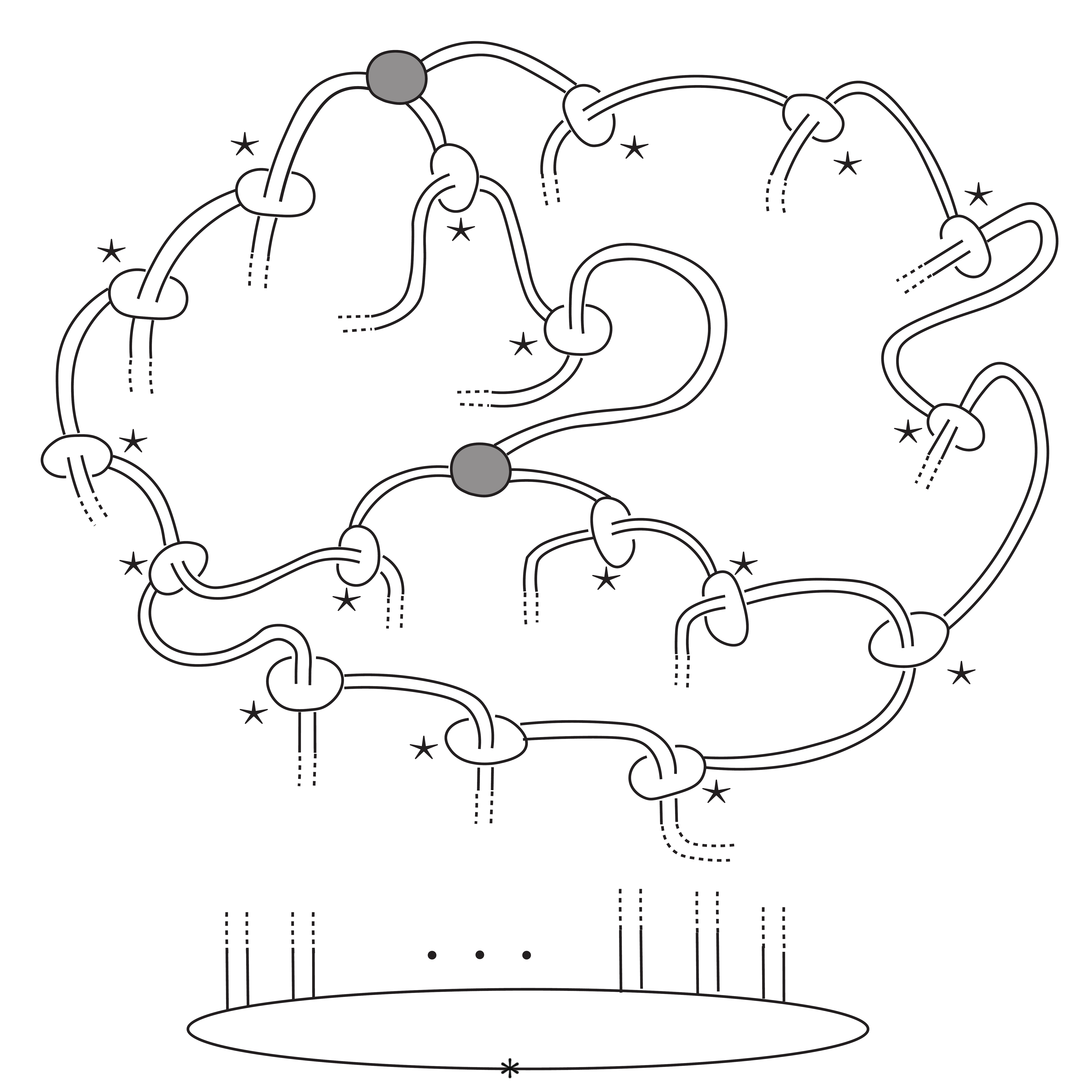}
   \caption{The presentation $P(\ctext{Y}(3, 4, 3, 2, 1, 2))$}
     \label{3loop-presentation}
\end{figure}

By cross-change moves, we can replace $P(\ctext{Y}((p_i)_i))$ with $P^{\prime}(\ctext{Y}((p_i)_i))$ so that $P^{\prime}(\ctext{Y}((p_i)_i)) ((\varepsilon_i =1)_i)$ is equivalent to the trivial presentation. See Figure \ref{3loop-presentation2}.
This ribbon presentation $P^{\prime}(\ctext{Y}((p_i)_i)) ((\varepsilon_i =1)_i)$ gives the cycle $d^{\prime}(\ctext{Y}((p_i)_i))$. 

\begin{figure}[htpb]
\centering
    \includegraphics [width =9cm] {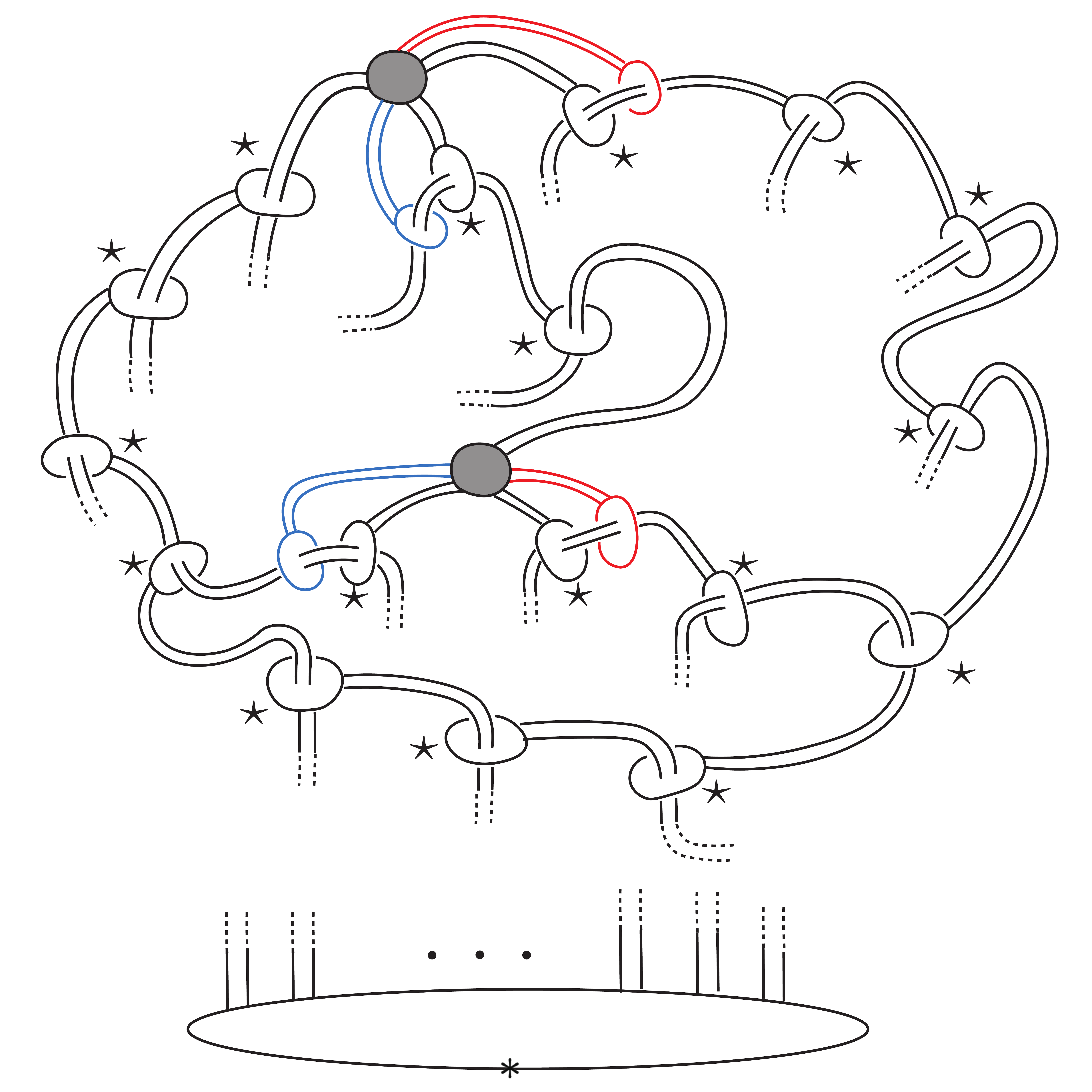}
   \caption{The presentation $P^{\prime}(\ctext{Y}(3, 4, 3, 2, 1, 2))$}
     \label{3loop-presentation2}
\end{figure}

\begin{notation}
We write $\mathcal{H}_{n,j}(k,g=3)$ for the subspace of $H_{k(n-j-2)+2(j-1)}(\overline{\mathcal{K}}_{n,j}, \mathbb{R})$ generated by the cycles $d^{\prime}(\ctext{Y}((p_i)_i)) $ which satisfy 
$k = 2 + \sum_{i=1}^6 p_i$ and one of the conditions (1)(2)(3) on $(p_i)_i$. 
We write $\mathcal{H}_{n,j}(g=3)$ for the subspace of $H_{\ast}(\overline{\mathcal{K}}_{n,j}, \mathbb{R})$  generated by the cycles $d^{\prime}(\ctext{Y}((p_i)_i))$ which satisfy one of the conditions (1)(2)(3) on  $(p_i)_i$ . 
\end{notation}

\begin{prop}
Let $n-j = 2$. 
Then if at least one of $\varepsilon_i$ is $-1$, the cycle 
\[
d(\ctext{Y}((p_i)_i))(\varepsilon_1, \varepsilon_2, \dots, \varepsilon_k): (S^{j-1})^2 \rightarrow \overline{\mathcal{K}}_{n,j}
\]
 is a degenerate cycle. More precisely, the cycle is homotopic to a cycle that is constant with respect to at least one of $S^{j-1}$ of the parameter space. A similar result also holds for $d^{\prime}(\ctext{Y}((p_i)_i))$.
 \end{prop}

Hence, when $n-j = 2$, this cycle $d^{\prime}(\ctext{Y}((p_i)_i))$ is homologous to the cycle 
\[
d^{\prime}(\ctext{Y}((p_i)_i))((\varepsilon_i = 1)_i): (S^{j-1})^2 \rightarrow \overline{\mathcal{K}}_{n,j}.
\]
This cycle is in the unknot component since  $P^{\prime}(\ctext{Y}((p_i)_i)$ is the trivial presentation. Furthermore, we have
\begin{prop}
\label{prop for 3loop from homotopy group}
The cycle $d^{\prime}(\ctext{Y}((p_i)_i))((\varepsilon_i = 1)_i)$ is homotopic to a constant function if it is restricted to $S^{j-1}\times \ast$ or $\ast \times S^{j-1}$ . 
\end{prop}

\begin{coro}
\label{3loop from homotopy group}
The cycle $d^{\prime}(\ctext{Y}((p_i)_i))((\varepsilon_i = 1)_i)$ gives an element of the homotopy group  $\pi_{2(j-1)} ((\overline{\mathcal{K}}_{j+2, j})_{\overline{\iota}})$. \footnote{Note that Proposition \ref{prop for 3loop from homotopy group} is enough for obtaining Cororally \ref{3loop from homotopy group}. We would need to check compatibility, if the parameter space were the product of more than two spheres. This is the case when $g\geq 4$.}
\end{coro}


\section{Non-triviality of $2$-loop and $3$-loop cycles of $\overline{\mathcal{K}}_{n,j}$}
\label{Pairing for general cases}

In \cite{Yos 3}, we established  a zigzag of cochain maps
\[
HGC_{n,j} \xleftarrow[p_1]{\simeq } PGC^{\prime}_{n,j} \xrightarrow[p_2]{\simeq} PGC_{n,j} \xleftarrow[p_3]{\simeq} DGC_{n,j} \xrightarrow[\overline{I}]{} \Omega^{\ast}_{dR}(\overline{\mathcal{K}}_{n,j}).
\]
that connects the hairy graph complex $HGC_{n,j}$ with the modified graph complex, the \textit{decorated graph complex} $DGC_{n,j}$. The intermediate complexes $PGC^{\prime}_{n,j}$ and $PGC_{n,j}$ are called the \textit{plain graph complexes}. The rightmost map is given by the modified configuration integrals, which is a cochain map at least when $j \geq 3$.

In this section, we perform pairing between general $2$-loop cycles $d(\Theta(p,q,r))$ of $\overline{\mathcal{K}}_{n,j}$ in Section \ref{Construction of general $2$-loop and $3$-loop cycles} and cocycles obtained by the $2$-loop part of the modified configuration space integrals $\overline{I}$. The pairing is reduced to pairing between graphs and chord diagrams (counting formula). Theorem \ref{main theorem on general cases} is shown in this section. 

\subsection{Review of the modified configuration space integrals}

\begin{definition}\cite{Yos 3}
\label{plain graphs}
\textit{Plain graphs} have three types of vertices
(white \begin{tikzpicture}[baseline = -3pt]  \draw (0, 0) circle (0.05); \end{tikzpicture},
external black \begin{tikzpicture} [baseline = -3pt]  \draw (0, 0) circle (0.05) [fill={rgb, 255:red, 0; green, 0; blue, 0}, fill opacity =1.0]; \end{tikzpicture},
internal black \begin{tikzpicture}  [baseline = -2pt]   \draw (0, 0) rectangle (0.1, 0.1) [fill={rgb, 255:red, 0; green, 0; blue, 0}, fill opacity =1.0]; \end{tikzpicture})
and two types of edges
(dashed \begin{tikzpicture}[x=0.75pt,y=0.75pt,yscale=0.3,xscale=0.3, baseline=-3pt]  \draw [dash pattern={on 4pt off 3pt}, line width = 1pt]   (0,0)--(90,0); \end{tikzpicture}, 
solid \begin{tikzpicture} \draw [x=0.75pt,y=0.75pt,yscale=0.3,xscale=0.3, baseline=-3pt]  [line width = 1pt](0,0)--(90,0); \end{tikzpicture}). 
White vertices have at least three dashed edges and no solid edge, while internal black vertices have at least three solid edges and no dashed edge. External black vertices have an arbitrary number of solid and dashed edges. We assume that each component has at least one external black vertex. 
Double edges and loop edges are allowed. 

A plain graph is \textit{admissible} if every external black vertex has at least one dashed edge. See Figure \ref {figofplaingraph}.
A plain graph is a \textit{good plain graph} if it has no internal black vertex and the solid component is a disjoint union of broken lines such as
\begin{tikzpicture}[x=0.5pt,y=0.5pt,yscale=0.5,xscale=0.5, baseline=0pt, line width = 1pt] 
\draw  [dash pattern = on 3pt off 2pt] (0, 0)-- (0,50);
\draw  [dash pattern = on 3pt off 2pt] (50, 0)-- (50,50);
\draw  [dash pattern = on 3pt off 2pt] (100, 0)-- (100,50) ;
\draw [dash pattern = on 3pt off 2pt]  (150, 0)--(150,50);
\draw  [dash pattern = on 3pt off 2pt] (200, 0)--(200,50);
\draw [dash pattern = on 3pt off 2pt]  (250, 0)--(250,50);

\draw  (0, 0)-- (100,0);
\draw  (150, 0)-- (200,0);

\draw [fill = black] (0, 0) circle (7);
\draw [fill = black] (50, 0) circle (7);
\draw [fill = black]  (100, 0) circle (7);

\draw [fill = black] (150, 0) circle (7);
\draw [fill = black](200, 0) circle (7);

\draw [fill = black]  (250, 0) circle (7);
\end{tikzpicture}. 

 \begin{figure}[htpb]
   \centering
    \includegraphics [width =7cm] {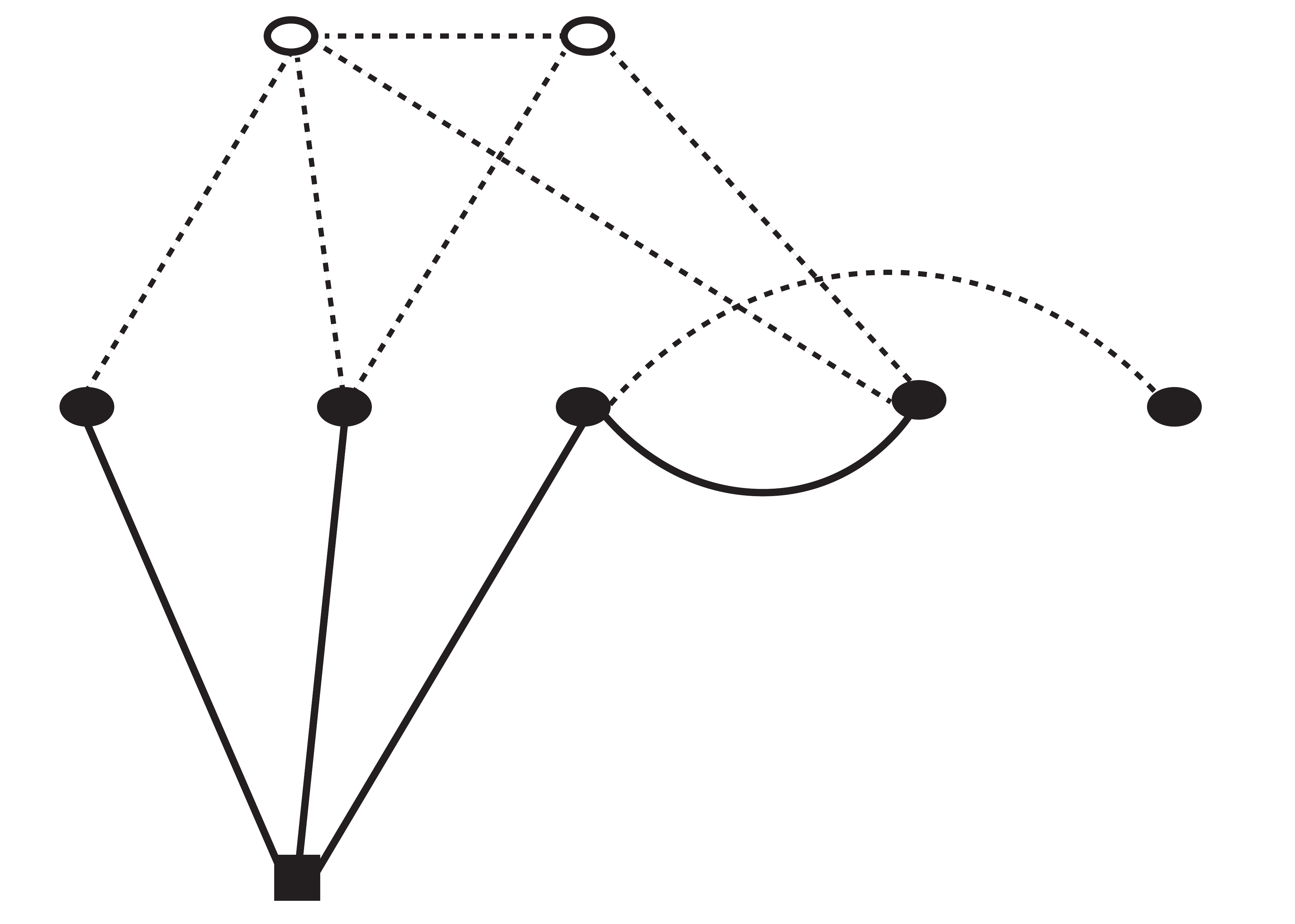}
    \caption{Example of an admissible (but not good) plain graph}
    \label{figofplaingraph}
\end{figure}
\end{definition}

\begin{notation}
Write $E(\Gamma) (= E_{\theta}(\Gamma) \cup E_{\eta}(\Gamma))$, $E_{\theta}(\Gamma)$, $E_{\eta}(\Gamma)$ for the sets of edges, dashed edges and solid edges.
Write $V(\Gamma) (= W(\Gamma) \cup B(\Gamma))$, $W(\Gamma)$, $B(\Gamma) (= B_i(\Gamma) \cup B_e(\Gamma))$, $B_i(\Gamma)$ and $B_e(\Gamma)$ for the set of vertices, white vertices, black vertices, internal black vertices and external black vertices. 
\end{notation}

\begin{definition}
The degrees of dashed edges, solid edges, white vertices and black vertices are $n-1$, $j-1$, $-n$, $-j$ respectively. 
The degree $|\Gamma| = deg(\Gamma)$ of a plain graph $\Gamma$ is defined by the sum of degrees
\[
|\Gamma| = (n-1) |E_{\theta}(\Gamma)| 
+ (j-1) |E_{\eta}(\Gamma)| 
-n |W(\Gamma)| 
-j |B(\Gamma)|.
\]
\end{definition}

\begin{definition}
Let $\Gamma$ be a plain graph. Define the \textit{defect} of a vertex $v$ of $\Gamma$ by
\begin{equation*}
l(v) = \begin{cases}
|E_{\theta}(v)| -3 & (v\in W(\Gamma)) \\
|E_{\theta}(v)|-1 & (v\in B(\Gamma)).\\
\end{cases}
\end{equation*}
Here $|E_{\theta}(v)|$ is the number of dashed (half) edges that $v$ has. Note that an admissible plain graph is allowed to have an internal black vertex of negative defect, which is $-1$. 
The \textit{defect} $l(\Gamma)$ of $\Gamma$ is defined as $l(\Gamma) = \sum_{v\in V(\Gamma)}l(v)$ and is equal to $2 |E_{\theta}(\Gamma)| - 3|W(\Gamma)| - |B(\Gamma)|$. 
\end{definition}

 \begin{definition}\cite{Yos 3}
 A label of a plain graph consists of a choice of an ordering of the set
 \[
 o(\Gamma) = E(\Gamma) \cup V(\Gamma)
 \]
 and a choice of orientations of edges. Each label gives an orientation of the underlying graph.
   \end{definition}
   
Admissible plain graphs generate the plain graph complexes $PGC_{n,j}$ and  $PGC^{\prime}_{n,j}$, where the degrees of the complexes are defined in terms of the degrees of graphs. 
The latter complex $PGC^{\prime}_{n,j}$ has a relation that graphs with double or loop edges vanish.

We call an element of  $PGC_{n,j}$ or $PGC^{\prime}_{n,j}$ of \textit{top degree} if the element is generated by plain graphs of defect $0$. If an admissible plain graph has no internal black vertices, the graph is of defect $0$ if and only if each white vertex has exactly three dashed edges and each external black vertex has exactly one dashed edge.

The \textit{order}, or \textit{complextity}, of plain graphs is defined by  $k(\Gamma) = |E_{\theta}(\Gamma)| - |W(\Gamma)|$. The plain graph complexes have a decomposition with respect to the order and the first Betti number. 
   
\begin{definition} \cite{Yos 3}
A \textit{decorated graph} is a plain graph whose external black vertices are decorated by an element of the acyclic bar complex $Z$ of an dg algebra $A_{n,j}$ called the hidden face dg algebra \cite{Yos 3}. A decorated graph is labeled if its plain part is labeled. Let  $\Gamma$ be a labeled decorated graph. We write the plain part as $P(\Gamma)$ and write the decorated part as $D(\Gamma) = z_1 \otimes \cdots \otimes z_l$, if $\Gamma$ has $l$ external black vertices and the $i$-th external vertex is decorated by $z_i \in Z$. We write $\Gamma = P(\Gamma) \otimes D(\Gamma)$. See Figure \ref{figofdecoratedgraph}. 

A decoration $z$ of an external black vertex is called \textit{trivial} if $z$ is obtained by scalar multiplication of the unit of $A_{n,j}$. A decorated graph is admissible if any external black vertex of it has either a dashed edge or a non-trivial decoration. 
\end{definition}

Admissible decorated graphs generate the decorated graph complex $DGC_{n,j}$. (See \cite{Yos 3} for the definition of $g$ and $k$ of decorated graphs.)

\begin{figure}[htpb]
   \centering
    \includegraphics [width =7cm] {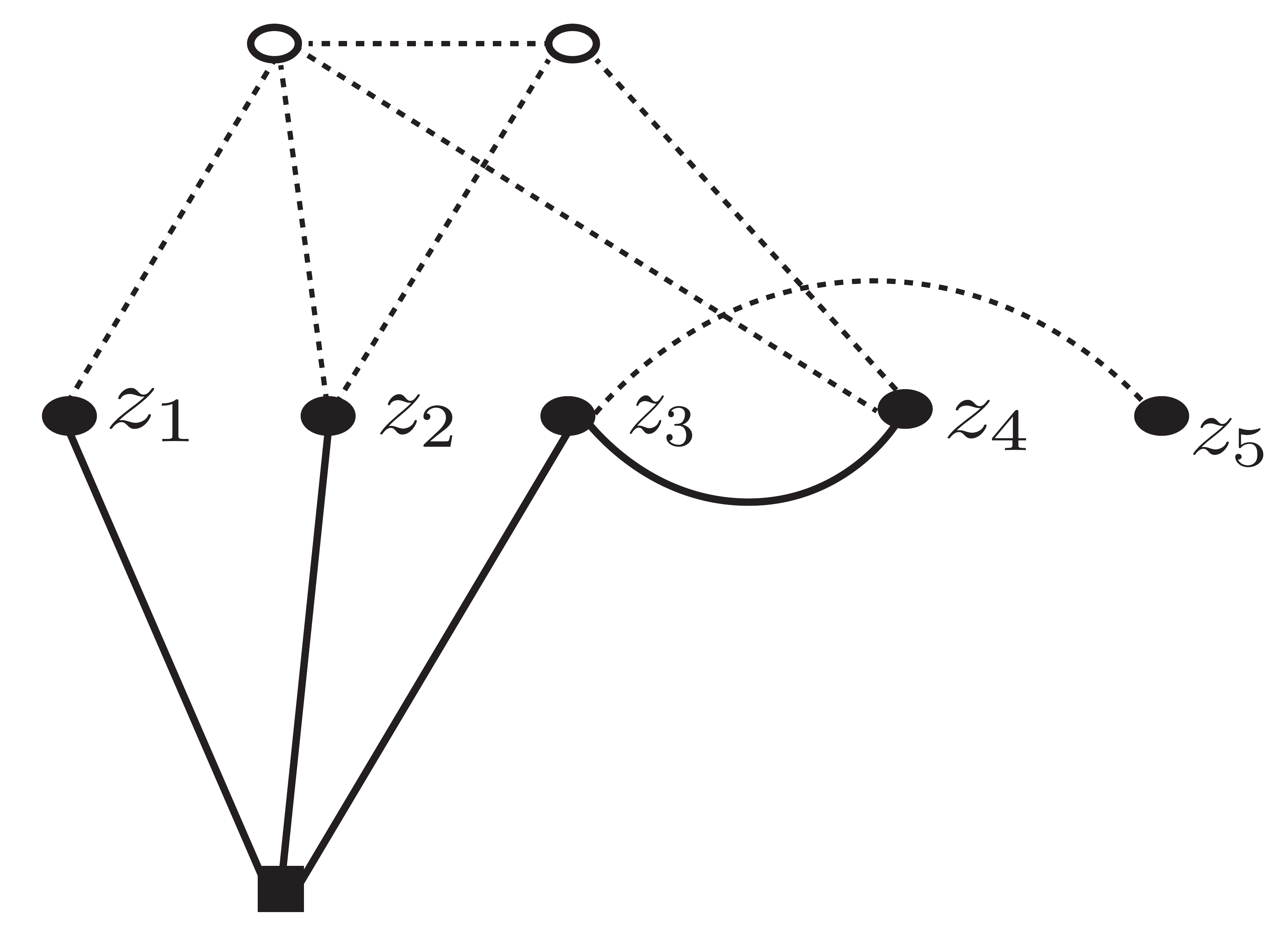}
    \caption{Example of an admissible decorated graph}
    \label{figofdecoratedgraph}
\end{figure}

The following is Main Theorem of \cite{Yos 3}.
\begin{theorem}\cite{Yos 3}
\label{main theorem 1 on hidden faces}
When $n-j \geq 2$ and $j \geq 3$, there  exists a cochain map
\[
\overline{I}: DGC_{n,j} \rightarrow \Omega_{dR} (\overline{\mathcal{K}}_{n,j}). 
\] 
It holds even when $j = 2$, if $\overline{I}$ is restricted to the subspace of $DGC_{n,j}$ generated by decorated graphs with the first Betti number $g \leq 3$. 
\end{theorem}   
   
 \begin{prop}\cite{Yos 3}
\label{keypropforpairing}
Let $\Gamma$ be a decorated graph of order $\leq k$.
Assume that the plain part $P(\Gamma)$ has at least $2k$ external black vertices with dashed edges. 
Then $\Gamma$ is a plain graph of order exactly $k$, without white vertices. Moreover, $\Gamma$ has exactly $k$ dashed edges and $2k$ external black vertices. 
\end{prop}

\subsection {The counting formula of configuration space integrals}

We recall the counting formula established in \cite{Yos 1}. In \cite{Yos 1}, this formula is stated in terms of BCR graphs, but it is applicable to plain graphs. First, for convenience, we recall some notation to state the counting formula.

\begin{definition}[Graph-chord pairing \cite{Yos 1}]
\label{Graph-chord pairing}
Let $C$ be a chord diagram on $s$ oriented lines of order $k$.  We write $V(C)$ and $E(C)$ for the set of vertices and the set of chords, respectively. Let $\Gamma$ be a labeled plain graph of order $k$.
Then the pairing $\langle \Gamma, C \rangle$ is defined by counting maps $|B(\Gamma)| \rightarrow  |V(C)|$ which satisfy suitable conditions as follows. First, we only count when
\begin{itemize}
\item [(I)] $ |B(\Gamma)| (= |V(C)|)   = 2k $.
\end{itemize}
We count permutations $\sigma : |B(\Gamma)| \rightarrow  |V(C)| $ which satisfy (I) and all of the following.
\begin{itemize}
\item [(II)] $ \sigma$ induces a bijective map  $ \overline{\sigma} : E_{\theta}(\Gamma) \rightarrow E(C)$.
\item [(III)] If two black vertices $v_1$ and $v_2$ are in the same solid component, $\sigma(v_1)$ and $\sigma(v_2)$ are on the same oriented line.
\item [(IV)] If a vertex $w$ of $V(C)$ is not on the $x$-axis, $w$ has an exactly one ingoing solid edge: there exists exactly one vertex $w^{\prime}$ lower on the same oriented line such that $\sigma^{-1}(w)$ and $\sigma^{-1}(w^{\prime})$ are connected by exactly one solid edge of $\Gamma$.
\end{itemize}
The sign of the counting is $+1$ if and only if the orientation of $\Gamma$ coincides with the induced orientation determined by the induced label from $\sigma$. (See Definition~\ref{Induced color on BCR diagrams}.)
\end{definition}

\begin{definition}[Induced label on graphs \cite{Yos 1}]
\label{Induced color on BCR diagrams}
Let $\sigma: B(\Gamma) \rightarrow V(C) $ satisfy all the conditions (I, II, III, IV) of Definition \ref{Graph-chord pairing}.
Then $\sigma$ defines an induced label on the underlying graph $\overline{\Gamma}$:
we order black vertices using $\sigma$ and the ordering of $V(C)$. 
We orient solid edges so that the orientations are compatible with oriented lines.
We orient dashed edges so that the orientations are compatible with chords.
We order edges in the following order.
\begin{itemize}
\item [(1)] Ingoing solid edge (if it exists) from the initial vertex of the $i$th chord ($i = 1, \dots, k$). There are $(g-1)$ solid edges ordered by this first step, where $g-1=k-s$.
\item [(2)] The $\overline{\sigma}^{-1}(i)$th dashed edge and the ingoing solid edge from the target point of the $i$th chord ($i = 1, \dots, k$).
\end{itemize}
\end{definition}

\begin{example}
In Figure \ref{exampleofpairing}, we compute the pairing of a plain graph $\Gamma$ and a chord diagram $C$ when $n$ and $j$ are odd. (Most labels of $\Gamma$ and $C$ are omitted.) Then, two permutations $|B(\Gamma)| \rightarrow  |V(C)| $ are allowed, and they have the same sign when $n$ and $j$ are odd. So we have $\langle \Gamma, C \rangle = \pm 2$, where the sign depends on labels of $\Gamma$ and $C$. Note that the permutations in Figure \ref{nonexampleofpairing} are not allowed because they do not satisfy the condition (IV) of Definition \ref{Graph-chord pairing}.

\begin{figure}[htpb]
\centering
     \includegraphics [width =10cm] {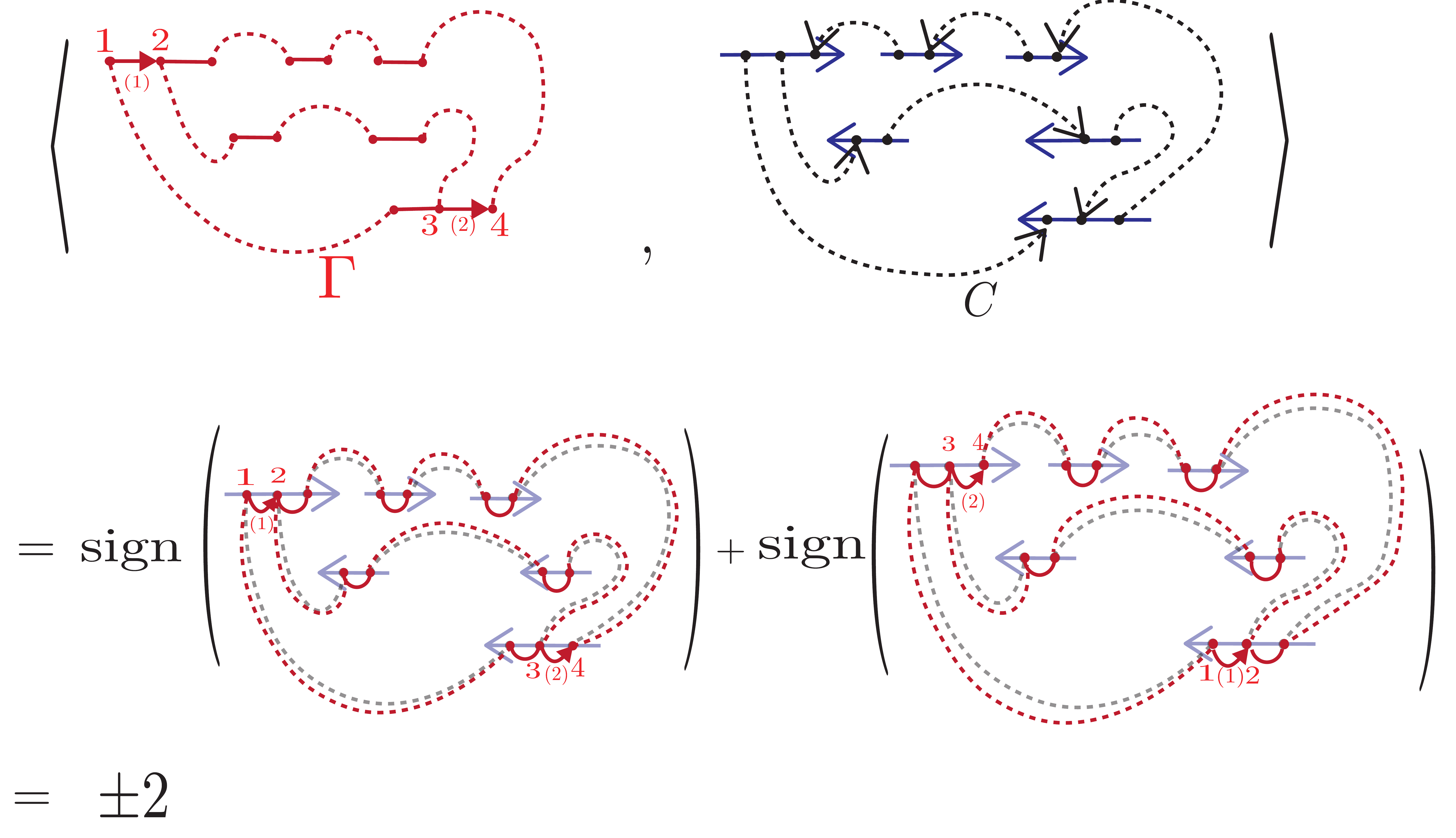}
    \caption{Example of a pairing. The graph $\Gamma$ is drawn in red.}
    \label{exampleofpairing}
\end{figure}

\begin{figure}[htpb]
\centering
     \includegraphics [width =11cm] {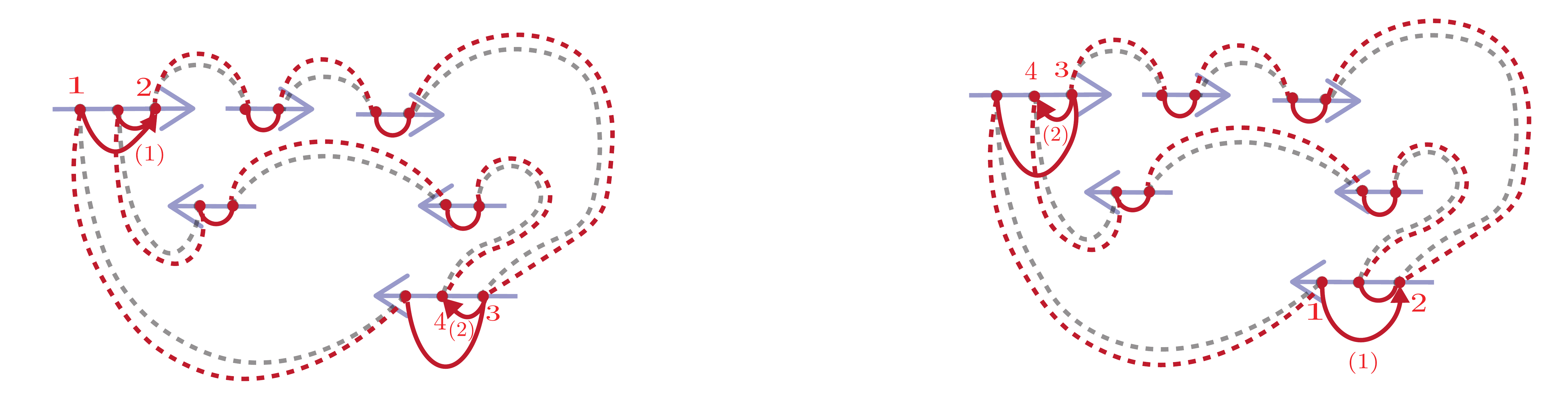}
    \caption{Example of permutations not counted}
    \label{nonexampleofpairing}
\end{figure}

\end{example}

\begin{notation}
Let $C = (\{t_i\}_{i = 1, \dots, s}, \{p_i\}_{i=1, \dots, k})$ be a chord diagram on oriented lines. 
Let $G(C)$ be the set of admissible plain graphs (without a label) mapped on $C$ satisfying all the conditions of Definition \ref{Graph-chord pairing}. 
Let $G^{\prime}(C)$ be the set of good plain graphs (without a label) mapped on $C$ satisfying all the conditions of Definition \ref{Graph-chord pairing}. 
The set $G(C)$ consists of $\prod_{i=1}^s  2^{t_i -1}$ graphs while $G^{\prime}(C)$ consists of $\prod_{i=1}^s  (t_i)!$ graphs. 
\end{notation}

\begin{theorem}[Counting formula \cite{Yos 1}]
\label{original counting formula}
Let $H = \sum \frac{w(\Gamma_i)}{|\text{Aut}(\Gamma_i)|} \Gamma_i$ be a linear combination of admissible plain graphs without internal black vertices,  of order $\leq k$ and of top degree. 
Let $C$ be a chord diagram on oriented lines of order $k$, with $r(C)$ chords negative sign. Let $\psi$ be the cycle constructed from $C$.
Then 
\begin{equation*}
\langle I(H), \psi \rangle = (-1)^{r(C)} \sum_{\overline{\Gamma} \in G(C)}  w(\Gamma)  s(\Gamma, \overline{\Gamma}).
\end{equation*}
Here $\Gamma$ is a labeled graph in $H$ such that the underlying graph of $\Gamma$ is $\overline{\Gamma}$.
The sign $s(\Gamma, \overline{\Gamma}) \in \{+1, -1\}$ is $+1$ if and only if the orientation of $\Gamma$ coincides with that of $\overline{\Gamma}$ with the induced label.
Similarly, if $ \sum \frac{w(\Gamma_i)}{|\text{Aut}(\Gamma_i)|} \Gamma_i$ is a linear combination of good plain graphs of order $\leq k$ and of top degree, then 
\begin{equation*}
 \langle I(H), \psi \rangle = (-1)^{r(C)} \sum_{\overline{\Gamma} \in G^{\prime}(C)}  w(\Gamma)   s(\Gamma, \overline{\Gamma}).
\end{equation*}
\end{theorem}

\subsection{The counting formula for $2$-loop (co)cycles}

Suppose $k \geq 1$ is an integer and suppose $p, q, r$ are integers which satisfy $p+q+r+1 = k$ and $p,r \geq 1$, $q \geq 0$. Let $H$ be a $2$-loop graph cocycle of order $\leq k$ and of top degree, expressed as
\[
H = \sum_i \frac{w(\Gamma_i)}{|\text{Aut}(\Gamma_i)|} \Gamma_i. 
\]
where $\Gamma_i$ are admissible plain graphs without internal black vertices. $w(\Gamma_i)$ is the coefficient of a graph $\Gamma_i$, divided by the number of automorphisms $|\text{Aut}(\Gamma_i)|$.
We always assume each $\Gamma_i$ has no orientation-reversing automorphisms. 

For our $2$-loop cycles $d^{\prime}(\Theta(p,q,r))$, the counting formula \ref{original counting formula} becomes more simple:

\begin{theorem}[Counting formula for $2$-loop (co)cycles]
\label{keyprop}
\[
\langle I(H), d^{\prime}(\Theta(p,q,r)) \rangle = \langle I(H), d(\Theta(p,q,r)) \rangle = \pm w(\Theta(p,q,r)).
\]
where $\pm$ depends only on the orientation of  $\Theta(p,q,r)$ in the graph cocycle $H$. \footnote{We can show that a similar result holds even when we allow a plain graph with an internal black vertex. Refer to proof of Theorem \ref{keyprop2}}
\end{theorem}

We proceed to show Theorem \ref{keyprop}. First, the following lemma is a consequence of Localizing lemma and Pairing Lemma in \cite{Yos 1}, which say a crossing is detected when it is connected by a dashed edge. 

\begin{lemma}
\[
\langle I(H), d^{\prime}(\Theta(p,q,r)) - d(\Theta(p,q,r)) \rangle = 0. 
\]
\end{lemma}

Then, we show the second equality of Theorem \ref{keyprop}. 
Consider the four graphs $D_1$, $D_2$, $D_3$, $D_4$ in Figure \ref{D1D2D3D4}, which are obtained by performing $STU$ relations \cite{Yos 2} to the oriented hairy graph $\Theta(p,q,r)$. Note that some of $D_1$, $D_2$, $D_3$, $D_4$ might be the same graph.   We assume that the labels of the four graphs are the same if we contract the two labeled edges of each graph in Figure \ref{D1D2D3D4}. 
We can take such labels of these graphs so that the relation 
\[
w(\Theta(p,q,r)) =w(D_1) + w(D_2) +  w(D_3) + w(D_4).
\]
is satisfied in the graph cocycle $H = \sum_i \frac{w(\Gamma_i)}{|\text{Aut}(\Gamma_i)|} \Gamma_i $. 
We show the integrals $I(\Gamma_i)$ which do not vanish on the cycle $d(\Theta(p,q,r))$, are only the integrals when $\Gamma_i$  are either of the graphs $D_1$, $D_2$, $D_3$, $D_4$. 

\begin{figure}[htpb]
\centering
     \includegraphics [width =10cm] {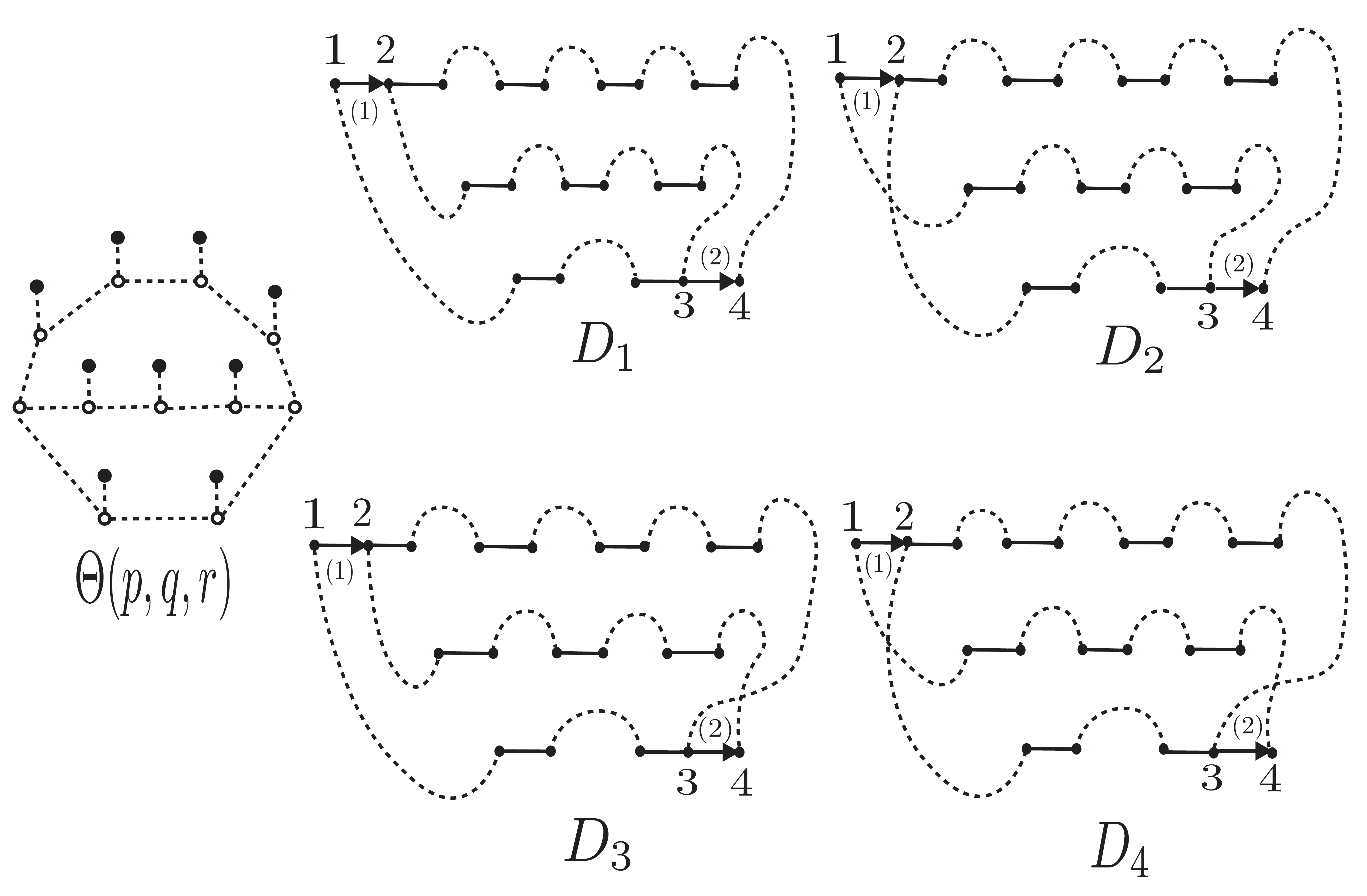}
    \caption{Graphs $D_1$, $D_2$, $D_3$ and $D_4$}
    \label{D1D2D3D4}
\end{figure}

\begin{notation}
Suppose $\Gamma_i$ has no orientation-reversing automorphism. Define the paring of $\Gamma_i$ and $D_j$ by
\begin{equation*}
\langle \Gamma_i, D_j \rangle =
\begin{cases}
0 & (\text{if $\Gamma_i$ is not isomorphic to $ D_j$}) \\
\pm1 & (\text{if $\Gamma_i$ is isomorphic to $D_j$}) \\
\end{cases} 
\end{equation*}
The sign is positive (resp. negative) if the isomorphism preserves (resp. reverses) the orientation.
\end{notation}

The counting formula Theorem \ref{keyprop} for the $2$-loop cycles $d(\Theta(p,q,r))$ is a consequence of the following. 
\begin{lemma}
\label{keylemmaforpairing}
If the order of a graph $\Gamma_i$ is less than or equal to $k$, we have
\[
 \langle I(\Gamma_i), d(\Theta(p,q,r)) \rangle = \pm \sum_{j = 1,2,3,4} |\text{Aut}(D_j)| \langle \Gamma_i, D_j \rangle,
\]
where $\pm$ depends only on the orientation of  $\Theta(p,q,r)$ in the graph cocycle $H$.
\end{lemma}

\begin{proof}[Proof of Theorem \ref{keyprop}]
Assuming Lemma \ref{keylemmaforpairing}, we have 
\begin{align*}
\langle I(H), d(\Theta(p,q,r)) \rangle &= \pm \sum_i \frac{w(\Gamma_i)}{|\text{Aut}(\Gamma_i)|}  \sum_{j = 1,2,3,4} |\text{Aut}(D_j)| \langle \Gamma_i, D_j \rangle \\
& = \pm (w(D_1) + w(D_2) +w(D_3) + w(D_4)) \\
& = \pm w(\Theta(p,q,r)).
\end{align*}
\end{proof}

\begin{proof}[Proof of Lemma \ref{keylemmaforpairing}]
By the counting formula \ref{original counting formula}, the pairing $ \langle I(\Gamma_i), d(\Theta(p,q,r)) \rangle$ is equal to counting graphs on the diagram $D = D(\Theta(p,q,r))$. On the segment \segmenta, only
\begin{tikzpicture}[x=1pt,y=1pt,yscale=0.15,xscale=0.15, baseline=-3pt] 

\draw  [line width =1pt] (100,0)--(300,0);

\draw  [dash pattern = on 2pt off 3 pt, line width =1pt] (100,0)--(100,150);
\draw  [dash pattern = on 2pt off 3 pt, line width =1pt] (300,0)--(300,150);

\draw  [fill={rgb, 255:red, 0; green, 0; blue, 0 }  ,fill opacity=1 ]  (100,0) circle (20);
\draw  [fill={rgb, 255:red, 0; green, 0; blue, 0 }  ,fill opacity=1 ] (300,0) circle (20);
\end{tikzpicture}
is counted.
On the segment \segmentb, only
\begin{tikzpicture}[x=1pt,y=1pt,yscale=0.15,xscale=0.15, baseline=-3pt] 

\draw  [line width =1pt] (100,0)..controls (200, 50)..(300,0);
\draw  [line width =1pt] (300,0)..controls (400, 50)..(500,0);

\draw  [dash pattern = on 2pt off 3 pt, line width =1pt] (100,0)--(100,150);
\draw  [dash pattern = on 2pt off 3 pt, line width =1pt] (300,0)--(300,150);
\draw  [dash pattern = on 2pt off 3 pt, line width =1pt] (500,0)--(500,150);

\draw  [fill={rgb, 255:red, 0; green, 0; blue, 0 }  ,fill opacity=1 ]  (100,0) circle (20);
\draw  [fill={rgb, 255:red, 0; green, 0; blue, 0 }  ,fill opacity=1 ] (300,0) circle (20);
\draw  [fill={rgb, 255:red, 0; green, 0; blue, 0 }  ,fill opacity=1 ] (500,0) circle (20);

\end{tikzpicture}
and
\begin{tikzpicture}[x=1pt,y=1pt,yscale=0.15,xscale=0.15, baseline=-3pt] 

\draw  [line width =1pt] (100,0)..controls (200, 50)..(300,0);
\draw  [line width =1pt] (100,0)..controls (250, 100)..(500,0);

\draw  [dash pattern = on 2pt off 3 pt, line width =1pt] (100,0)--(100,150);
\draw  [dash pattern = on 2pt off 3 pt, line width =1pt] (300,0)--(300,150);
\draw  [dash pattern = on 2pt off 3 pt, line width =1pt] (500,0)--(500,150);

\draw  [fill={rgb, 255:red, 0; green, 0; blue, 0 }  ,fill opacity=1 ]  (100,0) circle (20);
\draw  [fill={rgb, 255:red, 0; green, 0; blue, 0 }  ,fill opacity=1 ] (300,0) circle (20);
\draw  [fill={rgb, 255:red, 0; green, 0; blue, 0 }  ,fill opacity=1 ] (500,0) circle (20);

\end{tikzpicture}
are allowed, and 
\begin{tikzpicture}[x=1pt,y=1pt,yscale=0.15,xscale=0.15, baseline=-3pt] 

\draw  [line width =1pt] (300,0)..controls (400, 50)..(500,0);
\draw  [line width =1pt] (100,0)..controls (350, 100)..(500,0);

\draw  [dash pattern = on 2pt off 3 pt, line width =1pt] (100,0)--(100,150);
\draw  [dash pattern = on 2pt off 3 pt, line width =1pt] (300,0)--(300,150);
\draw  [dash pattern = on 2pt off 3 pt, line width =1pt] (500,0)--(500,150);

\draw  [fill={rgb, 255:red, 0; green, 0; blue, 0 }  ,fill opacity=1 ]  (100,0) circle (20);
\draw  [fill={rgb, 255:red, 0; green, 0; blue, 0 }  ,fill opacity=1 ] (300,0) circle (20);
\draw  [fill={rgb, 255:red, 0; green, 0; blue, 0 }  ,fill opacity=1 ] (500,0) circle (20);

\end{tikzpicture}
is not counted.  Then, there are four plain graphs counted, which are $D_1, \dots, D_4$. Figure \ref{D2iscounted} shows how the graph $D_2$ (drawn in red) is counted on the diagram $D(\Theta(p,q,r))$. 
So we have
\begin{align*}
I(H)(d(\Theta(p,q,r))) = (-1)^{r(D(\Theta(p.q.r))} \sum_{\overline{\Gamma_i} \in G(D(\Theta(p,q,r))}  w(\Gamma_i) s(\Gamma_i, \overline{\Gamma})
=  \sum_{i=1}^4 w(D_i) s(D_i, \overline{D}_i)
\end{align*}
On the other hand, the sign $s(D_i, \overline{D}_i)$ is the same for all $i = 1, 2, 3, 4$ because of the way of labeling in Figure \ref{D1D2D3D4}. 

\begin{figure}[htpb]
\begin{center}
     \includegraphics [width =7.5cm] {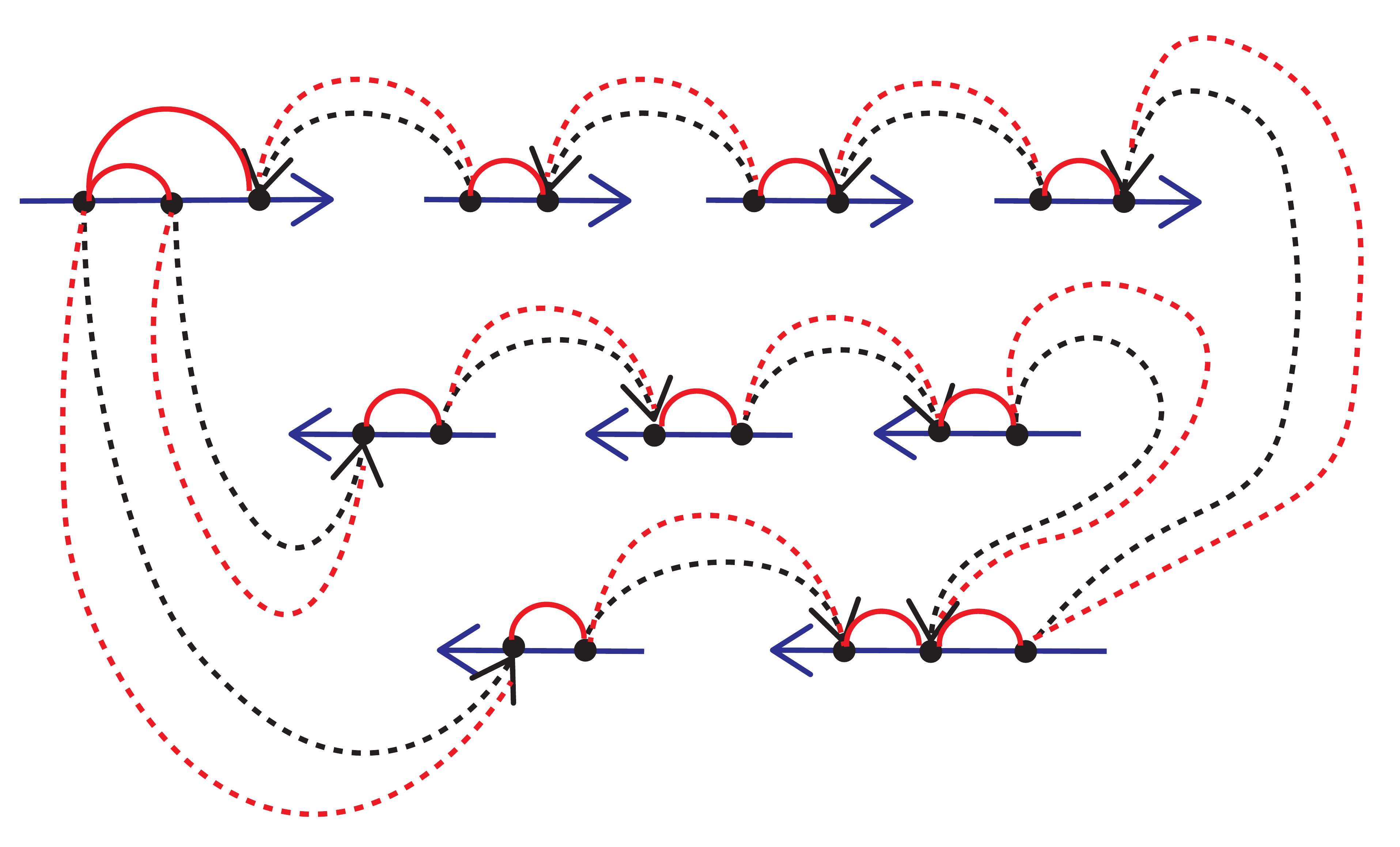}
    \caption{Graph $D_2$ (drawn in red) is counted on the diagram $D(\Theta(p,q,r))$}
    \label{D2iscounted}
\end{center}
\end{figure}

\end{proof}

\subsection{Non-triviality of $2$-loop cycles}
Configuration integrals of graph cocycles of plain graphs do not necessarily give cocycles of $\overline{\mathcal{K}}_{n,j}$. 
So, for the remainder of this section, we use the decorated graph complex and the modified configuration space integrals. 
Fortunately, we can show that added graphs do not affect pairing argument.

\begin{theorem}[Theorem \ref{main theorem on general cases}]
\label{nontrivialityof2loopcycles}
Let $\mathcal{H}_{n,j}(g=2)$ be the class of cycles defined in Section \ref{the class of 2loop cycles}. 
Then, the pairing map 
\[
\langle, \rangle: H^{\ast}(DGC_{n,j}(g=2))\otimes H_{\ast}(\overline{\mathcal{K}}_{n,j}, \mathbb{R}) \rightarrow \mathbb{R}, \quad (H, c) \longmapsto \langle \overline{I}(H), c \rangle
\]
induces the map
\[
Pair: H^{top}(HGC_{n,j}(g=2)) \otimes \mathcal{H}_{n,j}(g=2) \rightarrow \mathbb{R}
\]
which is non-degenerate with respect to the first term. That is, we can give an injective map from the top hairy graph cohomology to the dual of this class of cycles.
\end{theorem}

In Cororally \ref{from homotopy group}, we showed that $\mathcal{H}_{n,j}(g=2)$ lies in the homotopy group  $\pi_{j-1} ((\overline{\mathcal{K}}_{j+2, j})_{\overline{\iota}})$. 
Since $H^{top}(HGC_{n,j}(g=2))$ is known to be infinite-dimensional \cite{CCTW}, we have the following. This result was first given by Budney--Gabai \cite {BG} and Watanabe \cite{Wat 5} by studying the diffeomorphism group
$\text{Diff}_{\partial}(D^{j+1} \times S^1)$. 
\footnote{$\text{Diff}_{\partial}(D^{j+1} \times S^1)$ and $ \text{Emb}_{\partial}(D^j, D^{j+2})$ are related by the map $\text{Diff}_{\partial} (D^{j+1} \times S^1) \xrightarrow{ev} \text{Emb}_{\partial}(D^{j+1}, D^{j+1} \times S^1) \xrightarrow[\simeq]{scan}\ \Omega \text{Emb}_{\partial}(D^j, D^{j+2})$, where the first map is the evaluation and the second map is scanning. }

\begin{coro}
\label{non-finite generation}
$\pi_{j-1} ((\overline{\mathcal{K}}_{j+2, j})_{\overline{\iota}}) \otimes \mathbb{Q}$ is infinite-dimensional. Hence, $\pi_{j-1} ((\mathcal{K}_{j+2, j})_{\iota}) \otimes \mathbb{Q}$ is infinte-dimensional. 
\end{coro}

Suppose $k \geq 1$ is an integer and suppose $p, q, r$ are integers which satisfy $p+q+r+1 = k$, $p,r \geq 1$, $q \geq 0$. Let $H$ be a $2$-loop, top graph cocycle of order $\leq k$, expressed as
\[
H = \sum_i \frac{w(\Gamma_i)}{|\text{Aut}(P(\Gamma_i))|} \Gamma_i. 
\]
where $\Gamma_i$ are decorated graphs with different plain parts. $w(\Gamma_i)$ is the coefficient of a graph $\Gamma_i$, divided by the number of automorphisms $|\text{Aut}(P(\Gamma_i))|$ of the plain part. The following is the key proposition to show Theorem \ref{nontrivialityof2loopcycles} and is analogous to Proposition \ref{keyprop}. 

\begin{prop}[Counting formula for $2$-loop (co)cycles]
\label{keyprop2}
\[
\langle \overline{I}(H), d^{\prime}(\Theta(p,q,r)) \rangle  = \langle \overline{I}(H), d(\Theta(p,q,r)) \rangle = \pm w(\Theta(p,q,r)).
\]
where $\pm$  depends only on the orientation of  $\Theta(p,q,r)$ in the graph cocycle $H$.
\end{prop}

\begin{proof}[Poof of Theorem \ref{nontrivialityof2loopcycles}]
Let $h$ be a top graph cocycle of $HGC_{n,j}(g=2)$. Let  $H $ be a lift of $h$ to $DGC_{n,j}(g=2)$. \footnote{We call a graph cocycle $H \in DGC_{n,j}(g=2)$ is a lift of a graph cocycle $h \in HGC_{n,j}(g=2)$ if they represent the same cohomology class of $H^{\ast}( DGC_{n,j}(g=2)) \cong H^{\ast} (HGC_{n,j}(g=2))$ and the coefficients of hairy graphs without double edges or loop edges coincide. }
Assume $\overline{I}(H)$ vanishes on any cycle $d^{\prime}(\Theta(p,q,r))$. Then any coefficient of $\Theta(p,q,r)$ in $h$ must vanish by Proposition \ref{keyprop2}. This means $h= 0$ in $H^{top}(HGC_{n,j}(g=2))$. 
\end{proof}

\begin{proof}[Poof of Theorem \ref{keyprop2}]
We only have to show Lemma \ref{keylemmaforpairing} holds for decorated graphs. 
First, observe that Localizing Lemma and Pairing Lemma in \cite{Yos 1}  hold even for a decorated graph.
Thus, if the pairing $\langle \overline{I}(\Gamma_i), d(\Theta(p,q,r)) \rangle$ is non-trivial, the graph $\Gamma_i$ must have at least $2k$ external vertices which have dashed edges. 
On the other hand, since the graph $\Gamma_i$ has order $\leq k$, we can show that $\Gamma_i$ has no decoration and no white vertices by Proposition \ref{keypropforpairing}. Moreover, we can see that $\Gamma_i$ has exactly $2k$ external vertices and exactly $k$ chords. By computing the degree, internal black vertices are not allowed. Then, only plain graphs without internal black vertices can survive. Hence, we can apply the counting formula \ref{original counting formula}.
\end{proof}

\begin{rem}
In \cite{Wat 5}, Watanabe constrructed equivalent Kontsevich characteristic classes of $\text{BDiff}_{\partial}(D^{j+1} \times S^1)$ through configuration space integrals. 
These invariants take values in the space $\mathcal{A}_{\Theta}( \mathbb{R}[t^{\pm 1}])$ whose elements are described as
\[
 \begin{tikzpicture}[x=2.5pt,y=2.5pt,yscale=0.15,xscale=0.15, baseline=-3pt] 
\draw (0,0) circle (100); 

\draw [] (100,1) -- (100,0); 
\draw [] (-100,-1) -- (-100,0); 
\draw [] (100,0) -- (-100,0); 

\draw (0,120) node {$f(t_1)$};
\draw (0,20) node {$g(t_2)$};
\draw (0,-80) node {$h(t_3)$};
\end{tikzpicture}
\quad (f(t), g(t), h(t) \in  \mathbb{R}[t^{\pm 1}]),
\]
which have relations 
\[
(f(t_1), g(t_2), h(t_3)) \sim (t_1f(t_1), t_2g(t_2),t_3 h(t_3)).
\]
It is interesting to compare our invariants with these invariants by substitution $t = e^x$.
\end{rem}

\subsection{Non-triviality of $3$-loop cycles}

\begin{theorem}[Theorem \ref{main theorem on general cases}]
\label{nontrivialityof3loopcycles}
Let $\mathcal{H}_{n,j}(g=3)$ be the class of $3$-loop cycles defined in Section \ref{3loopcycles}. 
Then, the pairing map 
\[
\langle , \rangle: H^{\ast}(DGC_{n,j}(g=3))\otimes H_{\ast}(\overline{\mathcal{K}}_{n,j}, \mathbb{R}) \rightarrow \mathbb{R} \quad (H, c) \longmapsto \langle \overline{I}(H), c \rangle
\]
induces
\[
Pair: H^{top}(HGC_{n,j}(g=3)) \otimes \mathcal{H}_{n,j}(g=3) \rightarrow \mathbb{R} ,
\]
which is non-degenerate with respect to the first term. That is, we can give an injective map from the top hairy graph cohomology $H^{top}(HGC_{n,j}(g=3))$ to the dual of the class $\mathcal{H}_{n,j}(g=3)$ of cycles.
\end{theorem}

In Cororally \ref{3loop from homotopy group}, we showed that $\mathcal{H}_{n,j}(g=3)$ lies in the homotopy group  $\pi_{2(j-1)} ((\overline{\mathcal{K}}_{j+2, j})_{\overline{\iota}})$. 
Since $H^{top}(HGC_{n,j}(g=3))$ is known to be infinite-dimensional (see \cite{MO,Yos 3}), we have the following. 
This result is about a higher range than Budney--Gabai \cite {BG} and Watanabe \cite{Wat 5}, though their approach is very likely to be extended to this range. 

\begin{coro}
\label{3loop non-finite generation}
$\pi_{2(j-1)} ((\overline{\mathcal{K}}_{j+2, j})_{\overline{\iota}}) \otimes \mathbb{Q}$ is infinite-dimensional. Hence, $\pi_{2(j-1)} ((\mathcal{K}_{j+2, j})_{\iota}) \otimes \mathbb{Q}$ is infinte-dimensional. 
\end{coro}

Suppose $k \geq 1$ is an integer and suppose $p_1, p_2, \dots, p_6$ are integers which satisfy $2 + \sum_i p_i = k$ and either of the following conditions:
\begin{itemize}
\item [(1)] $p_1, p_6, p_3, p_4 \geq 1$, 
\item [(2)] $p_1, p_6, p_4 \geq 1$ and $p_2, p_3, p_5 = 0$, 
\item [(3)] $p_1, p_6 \geq 1$ and $p_2, p_3, p_4, p_5 = 0$.
\end{itemize}

Let $H$ be a $3$-loop graph cocycle of order $\leq k$, of top degree, expressed as
\[
H = \sum_i \frac{w(\Gamma_i)}{|\text{Aut}(P(\Gamma_i))|} \Gamma_i. 
\]
where $\Gamma_i$ are decorated graphs with different plain parts. 
Theorem \ref{nontrivialityof3loopcycles} is shown by establishing the $3$-loop analogue of Proposition \ref{keyprop2}.

\begin{prop}[Counting formula for $3$-loop (co)cycles]
\label{keyprop3}
\[
\langle \overline{I}(H), d(\ctext{Y}((p_i)_i)) \rangle = \pm w(\ctext{Y}((p_i)_i)), 
\]
where $\pm$ depends  only on the orientation of $\ctext{Y}((p_i)_i)$ in the graph cocycle $H$.
\end{prop}

\begin{proof}[Poof of Theorem \ref{nontrivialityof3loopcycles}]
Let $h$ be a top graph cocycle of $HGC_{n,j}(g=3)$. Then, there exists a lift $H$ of $h$ to $DGC_{n,j}(g=3)$. Assume $\overline{I}(H)$ vanishes on any cycle  $d(\ctext{Y}((p_i)_i))$ . Then, by Proposition \ref{keyprop3}, any coefficient of $\ctext{Y}((p_i)_i)$ in $h$ must vanish, for $p_1, p_2, \dots, p_6$ satisfying either of (1)(2)(3). This means $h= 0$ in $H^{\ast}(DGC_{n,j}(g=3))$. 
\end{proof}

\begin{proof}[Proof of Proposition \ref{keyprop3}] 
First, suppose condition (1) is satisfied. In this case, the proof is similar to the proof of Proposition \ref{keyprop} and  \ref{keyprop2}. 
Recall that $D(\ctext{Y}((p_i)_i))$ has four oriented lines with three vertices, \segmentb{}. For this part, two  subgraphs 
\begin{tikzpicture}[x=1pt,y=1pt,yscale=0.15,xscale=0.15, baseline=-3pt] 

\draw  [line width =1pt] (100,0)..controls (200, 50)..(300,0);
\draw  [line width =1pt] (300,0)..controls (400, 50)..(500,0);

\draw  [dash pattern = on 2pt off 3 pt, line width =1pt] (100,0)--(100,150);
\draw  [dash pattern = on 2pt off 3 pt, line width =1pt] (300,0)--(300,150);
\draw  [dash pattern = on 2pt off 3 pt, line width =1pt] (500,0)--(500,150);

\draw  [fill={rgb, 255:red, 0; green, 0; blue, 0 }  ,fill opacity=1 ]  (100,0) circle (20);
\draw  [fill={rgb, 255:red, 0; green, 0; blue, 0 }  ,fill opacity=1 ] (300,0) circle (20);
\draw  [fill={rgb, 255:red, 0; green, 0; blue, 0 }  ,fill opacity=1 ] (500,0) circle (20);

\end{tikzpicture}
and 
\begin{tikzpicture}[x=1pt,y=1pt,yscale=0.15,xscale=0.15, baseline=-3pt] 

\draw  [line width =1pt] (100,0)..controls (200, 50)..(300,0);
\draw  [line width =1pt] (100,0)..controls (250, 100)..(500,0);

\draw  [dash pattern = on 2pt off 3 pt, line width =1pt] (100,0)--(100,150);
\draw  [dash pattern = on 2pt off 3 pt, line width =1pt] (300,0)--(300,150);
\draw  [dash pattern = on 2pt off 3 pt, line width =1pt] (500,0)--(500,150);

\draw  [fill={rgb, 255:red, 0; green, 0; blue, 0 }  ,fill opacity=1 ]  (100,0) circle (20);
\draw  [fill={rgb, 255:red, 0; green, 0; blue, 0 }  ,fill opacity=1 ] (300,0) circle (20);
\draw  [fill={rgb, 255:red, 0; green, 0; blue, 0 }  ,fill opacity=1 ] (500,0) circle (20);

\end{tikzpicture}
are counted. On the other hand, by STU relation, we have 
\[
w(
\begin{tikzpicture}[x=1pt,y=1pt,yscale=0.15,xscale=0.15, baseline=-3pt] 

\draw  [line width =1pt] (100,0)..controls (200, 50)..(300,0);
\draw  [line width =1pt] (300,0)..controls (400, 50)..(500,0);

\draw  [dash pattern = on 2pt off 3 pt, line width =1pt] (100,0)--(100,150);
\draw  [dash pattern = on 2pt off 3 pt, line width =1pt] (300,0)--(300,150);
\draw  [dash pattern = on 2pt off 3 pt, line width =1pt] (500,0)--(500,150);

\draw  [fill={rgb, 255:red, 0; green, 0; blue, 0 }  ,fill opacity=1 ]  (100,0) circle (20);
\draw  [fill={rgb, 255:red, 0; green, 0; blue, 0 }  ,fill opacity=1 ] (300,0) circle (20);
\draw  [fill={rgb, 255:red, 0; green, 0; blue, 0 }  ,fill opacity=1 ] (500,0) circle (20);

\end{tikzpicture})
+ 
w(
\begin{tikzpicture}[x=1pt,y=1pt,yscale=0.15,xscale=0.15, baseline=-3pt] 

\draw  [line width =1pt] (100,0)..controls (200, 50)..(300,0);
\draw  [line width =1pt] (100,0)..controls (250, 100)..(500,0);

\draw  [dash pattern = on 2pt off 3 pt, line width =1pt] (100,0)--(100,150);
\draw  [dash pattern = on 2pt off 3 pt, line width =1pt] (300,0)--(300,150);
\draw  [dash pattern = on 2pt off 3 pt, line width =1pt] (500,0)--(500,150);

\draw  [fill={rgb, 255:red, 0; green, 0; blue, 0 }  ,fill opacity=1 ]  (100,0) circle (20);
\draw  [fill={rgb, 255:red, 0; green, 0; blue, 0 }  ,fill opacity=1 ] (300,0) circle (20);
\draw  [fill={rgb, 255:red, 0; green, 0; blue, 0 }  ,fill opacity=1 ] (500,0) circle (20);

\end{tikzpicture})
= 
w(
\begin{tikzpicture}[x=1pt,y=1pt,yscale=0.15,xscale=0.15, baseline=-3pt] 
\draw  [line width =1pt] (100,0)..controls (200, 50)..(300,0);

\draw  [dash pattern = on 2pt off 3 pt, line width =1pt] (100,170)--(60,300);
\draw  [dash pattern = on 2pt off 3 pt, line width =1pt] (100,170)--(140,300);
\draw  [dash pattern = on 2pt off 3 pt, line width =1pt] (100,130)--(100,0);
\draw  [dash pattern = on 2pt off 3 pt, line width =1pt] (300,0)--(300,150);

\draw  [fill={rgb, 255:red, 0; green, 0; blue, 0 }  ,fill opacity=1 ]  (100,0) circle (20);
\draw  [fill={rgb, 255:red, 0; green, 0; blue, 0 }  ,fill opacity=1 ] (300,0) circle (20);
\draw  (100,150) circle (20);
\end{tikzpicture}).
\]
Other oriented lines of $D(\ctext{Y}((p_i)_i))$ are the ones with two vertices, \segmenta{}. And we have, 
\[
w(
\begin{tikzpicture} [x=1pt,y=1pt,yscale=0.15,xscale=0.15, baseline=-3pt] 
\draw  [line width =1pt] (100,0)--(300,0);
\draw  [dash pattern = on 2pt off 3 pt, line width =1pt] (100,0)--(100,150);
\draw  [dash pattern = on 2pt off 3 pt, line width =1pt] (300,0)--(300,150);
\draw  [fill={rgb, 255:red, 0; green, 0; blue, 0 }  ,fill opacity=1 ]  (100,0) circle (20);
\draw  [fill={rgb, 255:red, 0; green, 0; blue, 0 }  ,fill opacity=1 ] (300,0) circle (20);
\end{tikzpicture})
=
w(
\begin{tikzpicture}[x=1pt,y=1pt,yscale=0.15,xscale=0.15, baseline=-3pt] 

\draw  [dash pattern = on 2pt off 3 pt, line width =1pt] (100,170)--(60,300);
\draw  [dash pattern = on 2pt off 3 pt, line width =1pt] (100,170)--(140,300);
\draw  [dash pattern = on 2pt off 3 pt, line width =1pt] (100,130)--(100,0);

\draw  [fill={rgb, 255:red, 0; green, 0; blue, 0 }  ,fill opacity=1 ]  (100,0) circle (20);
\draw  (100,150) circle (20);
\end{tikzpicture}).
\]
Then, we can show  $\langle \overline{I}(H), d(\ctext{Y}((p_i)_i)) \rangle$ is equal to the sum of the weights of $16 = 2^4$ graphs which are obtained by performing STU relations to $\ctext{Y}((p_i)_i)$. So we have the desired result.

Suppose the condition (2) is satisfied. Recall that in this case the diagram $D(\ctext{Y}((p_i)_i))$ has the part 
\begin{tikzpicture}[x=1pt,y=1pt,yscale=0.15,xscale=0.15, baseline=-3pt] 
\draw  [-Stealth, line width =1pt, color={rgb, 255:red, 0; green, 0; blue, 255 } ] (0,0)--(900,0);
\draw  [dash pattern = on 2pt off 3 pt, line width =1pt] (100,0)--(100,150);
\draw  [dash pattern = on 2pt off 3 pt, line width =1pt] (300,0)--(300,150);
\draw  [dash pattern = on 2pt off 3 pt, line width =1pt] (500,0)--(500,150);
\draw  [dash pattern = on 2pt off 3 pt, line width =1pt] (700,0)--(700,150);
\draw  [fill={rgb, 255:red, 0; green, 0; blue, 0 }  ,fill opacity=1 ]  (100,0) circle (20);
\draw  [fill={rgb, 255:red, 0; green, 0; blue, 0 }  ,fill opacity=1 ] (300,0) circle (20);
\draw  [fill={rgb, 255:red, 0; green, 0; blue, 0 }  ,fill opacity=1 ] (500,0) circle (20);
\draw  [fill={rgb, 255:red, 0; green, 0; blue, 0 }  ,fill opacity=1 ] (700,0) circle (20);
\end{tikzpicture}.
For this part, six subgraphs 
$
\Gamma_1 = \begin{tikzpicture}[x=1pt,y=1pt,yscale=0.15,xscale=0.15, baseline=-3pt]  
\draw  [line width =1pt] (100,0)..controls (200, 50)..(300,0);
\draw  [line width =1pt] (300,0)..controls (400, 50)..(500,0);
\draw  [line width =1pt] (500,0)..controls (600, 50)..(700,0);
\draw  [dash pattern = on 2pt off 3 pt, line width =1pt] (100,0)--(100,150);
\draw  [dash pattern = on 2pt off 3 pt, line width =1pt] (300,0)--(300,150);
\draw  [dash pattern = on 2pt off 3 pt, line width =1pt] (500,0)--(500,150);
\draw  [dash pattern = on 2pt off 3 pt, line width =1pt] (700,0)--(700,150);
\draw  [fill={rgb, 255:red, 0; green, 0; blue, 0 }  ,fill opacity=1 ]  (100,0) circle (20);
\draw  [fill={rgb, 255:red, 0; green, 0; blue, 0 }  ,fill opacity=1 ] (300,0) circle (20);
\draw  [fill={rgb, 255:red, 0; green, 0; blue, 0 }  ,fill opacity=1 ] (500,0) circle (20);
\draw  [fill={rgb, 255:red, 0; green, 0; blue, 0 }  ,fill opacity=1 ] (700,0) circle (20);
\end{tikzpicture}
$,
$
\Gamma_2 = \begin{tikzpicture}[x=1pt,y=1pt,yscale=0.15,xscale=0.15, baseline=-3pt] 
\draw  [line width =1pt] (100,0)..controls (200, 50)..(300,0);
\draw  [line width =1pt] (100,0)..controls (300, 100)..(500,0);
\draw  [line width =1pt] (500,0)..controls (600, 50)..(700,0);
\draw  [dash pattern = on 2pt off 3 pt, line width =1pt] (100,0)--(100,150);
\draw  [dash pattern = on 2pt off 3 pt, line width =1pt] (300,0)--(300,150);
\draw  [dash pattern = on 2pt off 3 pt, line width =1pt] (500,0)--(500,150);
\draw  [dash pattern = on 2pt off 3 pt, line width =1pt] (700,0)--(700,150);
\draw  [fill={rgb, 255:red, 0; green, 0; blue, 0 }  ,fill opacity=1 ]  (100,0) circle (20);
\draw  [fill={rgb, 255:red, 0; green, 0; blue, 0 }  ,fill opacity=1 ] (300,0) circle (20);
\draw  [fill={rgb, 255:red, 0; green, 0; blue, 0 }  ,fill opacity=1 ] (500,0) circle (20);
\draw  [fill={rgb, 255:red, 0; green, 0; blue, 0 }  ,fill opacity=1 ] (700,0) circle (20);
\end{tikzpicture}
$,
$
\Gamma_3 = \begin{tikzpicture}[x=1pt,y=1pt,yscale=0.15,xscale=0.15, baseline=-3pt]  
\draw  [line width =1pt] (100,0)..controls (200, 50)..(300,0);
\draw  [line width =1pt] (300,0)..controls (400, 50)..(500,0);
\draw  [line width =1pt] (100,0)..controls (400, 150)..(700,0);
\draw  [dash pattern = on 2pt off 3 pt, line width =1pt] (100,0)--(100,150);
\draw  [dash pattern = on 2pt off 3 pt, line width =1pt] (300,0)--(300,150);
\draw  [dash pattern = on 2pt off 3 pt, line width =1pt] (500,0)--(500,150);
\draw  [dash pattern = on 2pt off 3 pt, line width =1pt] (700,0)--(700,150);
\draw  [fill={rgb, 255:red, 0; green, 0; blue, 0 }  ,fill opacity=1 ]  (100,0) circle (20);
\draw  [fill={rgb, 255:red, 0; green, 0; blue, 0 }  ,fill opacity=1 ] (300,0) circle (20);
\draw  [fill={rgb, 255:red, 0; green, 0; blue, 0 }  ,fill opacity=1 ] (500,0) circle (20);
\draw  [fill={rgb, 255:red, 0; green, 0; blue, 0 }  ,fill opacity=1 ] (700,0) circle (20);
\end{tikzpicture}
$,
$
\Gamma_4 = \begin{tikzpicture}[x=1pt,y=1pt,yscale=0.15,xscale=0.15, baseline=-3pt] 
\draw  [line width =1pt] (100,0)..controls (200, 50)..(300,0);
\draw  [line width =1pt] (100,0)..controls (300, 100)..(500,0);
\draw  [line width =1pt] (300,0)..controls (500, 100)..(700,0);
\draw  [dash pattern = on 2pt off 3 pt, line width =1pt] (100,0)--(100,150);
\draw  [dash pattern = on 2pt off 3 pt, line width =1pt] (300,0)--(300,150);
\draw  [dash pattern = on 2pt off 3 pt, line width =1pt] (500,0)--(500,150);
\draw  [dash pattern = on 2pt off 3 pt, line width =1pt] (700,0)--(700,150);
\draw  [fill={rgb, 255:red, 0; green, 0; blue, 0 }  ,fill opacity=1 ]  (100,0) circle (20);
\draw  [fill={rgb, 255:red, 0; green, 0; blue, 0 }  ,fill opacity=1 ] (300,0) circle (20);
\draw  [fill={rgb, 255:red, 0; green, 0; blue, 0 }  ,fill opacity=1 ] (500,0) circle (20);
\draw  [fill={rgb, 255:red, 0; green, 0; blue, 0 }  ,fill opacity=1 ] (700,0) circle (20);
\end{tikzpicture}
$,
$
\Gamma_5 = \begin{tikzpicture}[x=1pt,y=1pt,yscale=0.15,xscale=0.15, baseline=-3pt] 
\draw  [line width =1pt] (100,0)..controls (200, 50)..(300,0);
\draw  [line width =1pt] (100,0)..controls (300, 100)..(500,0);
\draw  [line width =1pt] (100,0)..controls (300, 150)..(700,0);
\draw  [dash pattern = on 2pt off 3 pt, line width =1pt] (100,0)--(100,150);
\draw  [dash pattern = on 2pt off 3 pt, line width =1pt] (300,0)--(300,150);
\draw  [dash pattern = on 2pt off 3 pt, line width =1pt] (500,0)--(500,150);
\draw  [dash pattern = on 2pt off 3 pt, line width =1pt] (700,0)--(700,150);
\draw  [fill={rgb, 255:red, 0; green, 0; blue, 0 }  ,fill opacity=1 ]  (100,0) circle (20);
\draw  [fill={rgb, 255:red, 0; green, 0; blue, 0 }  ,fill opacity=1 ] (300,0) circle (20);
\draw  [fill={rgb, 255:red, 0; green, 0; blue, 0 }  ,fill opacity=1 ] (500,0) circle (20);
\draw  [fill={rgb, 255:red, 0; green, 0; blue, 0 }  ,fill opacity=1 ] (700,0) circle (20);
\end{tikzpicture}
$,
$
\Gamma_6=\begin{tikzpicture}[x=1pt,y=1pt,yscale=0.15,xscale=0.15, baseline=-3pt] 
\draw  [line width =1pt] (100,0)..controls (200, 50)..(300,0);
\draw  [line width =1pt] (300,0)..controls (400, 50)..(500,0);
\draw  [line width =1pt] (300,0)..controls (500, 100)..(700,0);
\draw  [dash pattern = on 2pt off 3 pt, line width =1pt] (100,0)--(100,150);
\draw  [dash pattern = on 2pt off 3 pt, line width =1pt] (300,0)--(300,150);
\draw  [dash pattern = on 2pt off 3 pt, line width =1pt] (500,0)--(500,150);
\draw  [dash pattern = on 2pt off 3 pt, line width =1pt] (700,0)--(700,150);
\draw  [fill={rgb, 255:red, 0; green, 0; blue, 0 }  ,fill opacity=1 ]  (100,0) circle (20);
\draw  [fill={rgb, 255:red, 0; green, 0; blue, 0 }  ,fill opacity=1 ] (300,0) circle (20);
\draw  [fill={rgb, 255:red, 0; green, 0; blue, 0 }  ,fill opacity=1 ] (500,0) circle (20);
\draw  [fill={rgb, 255:red, 0; green, 0; blue, 0 }  ,fill opacity=1 ] (700,0) circle (20);
\end{tikzpicture}
$
are counted. On the other hand, the sum of the weights of these six graphs
coincides with 
\[
w(
\begin{tikzpicture}[x=1pt,y=1pt,yscale=0.15,xscale=0.15, baseline=-3pt] 
\draw  [line width =1pt] (100,0)..controls (200, 50)..(300,0);
\draw  [line width =1pt] (300,0)..controls (400, 50)..(500,0);
\draw  [dash pattern = on 2pt off 3 pt, line width =1pt] (100,170)--(60,300);
\draw  [dash pattern = on 2pt off 3 pt, line width =1pt] (100,170)--(140,300);
\draw  [dash pattern = on 2pt off 3 pt, line width =1pt] (100,130)--(100,0);
\draw  [dash pattern = on 2pt off 3 pt, line width =1pt] (300,0)--(300,150);
\draw  [dash pattern = on 2pt off 3 pt, line width =1pt] (500,0)--(500,150);
\draw  [fill={rgb, 255:red, 0; green, 0; blue, 0 }  ,fill opacity=1 ]  (100,0) circle (20);
\draw  [fill={rgb, 255:red, 0; green, 0; blue, 0 }  ,fill opacity=1 ] (300,0) circle (20);
\draw  [fill={rgb, 255:red, 0; green, 0; blue, 0 }  ,fill opacity=1 ] (500,0) circle (20);
\draw  (100,150) circle (20);
\end{tikzpicture}
)+
w(
\begin{tikzpicture}[x=1pt,y=1pt,yscale=0.15,xscale=0.15, baseline=-3pt] 
\draw  [line width =1pt] (100,0)..controls (200, 50)..(300,0);
\draw  [line width =1pt] (100,0)..controls (300, 100)..(500,0);
\draw  [dash pattern = on 2pt off 3 pt, line width =1pt] (100,170)--(60,300);
\draw  [dash pattern = on 2pt off 3 pt, line width =1pt] (100,170)--(140,300);
\draw  [dash pattern = on 2pt off 3 pt, line width =1pt] (100,130)--(100,0);
\draw  [dash pattern = on 2pt off 3 pt, line width =1pt] (300,0)--(300,150);
\draw  [dash pattern = on 2pt off 3 pt, line width =1pt] (500,0)--(500,150);
\draw  [fill={rgb, 255:red, 0; green, 0; blue, 0 }  ,fill opacity=1 ]  (100,0) circle (20);
\draw  [fill={rgb, 255:red, 0; green, 0; blue, 0 }  ,fill opacity=1 ] (300,0) circle (20);
\draw  [fill={rgb, 255:red, 0; green, 0; blue, 0 }  ,fill opacity=1 ] (500,0) circle (20);
\draw  (100,150) circle (20);
\end{tikzpicture})
+
w(
\begin{tikzpicture}[x=1pt,y=1pt,yscale=0.15,xscale=0.15, baseline=-3pt] 
\draw  [line width =1pt] (100,0)--(300,-150);
\draw  [line width =1pt] (300,0)--(300,-150);
\draw  [line width =1pt] (500,0)--(300,-150);
\draw  [dash pattern = on 2pt off 3 pt, line width =1pt] (100,0)--(70,160);
\draw  [dash pattern = on 2pt off 3 pt, line width =1pt] (100,0)--(130,160);
\draw  [dash pattern = on 2pt off 3 pt, line width =1pt] (300,0)--(300,150);
\draw  [dash pattern = on 2pt off 3 pt, line width =1pt] (500,0)--(500,150);
\draw  [fill={rgb, 255:red, 0; green, 0; blue, 0 }  ,fill opacity=1 ]  (100,0) circle (20);
\draw  [fill={rgb, 255:red, 0; green, 0; blue, 0 }  ,fill opacity=1 ] (300,0) circle (20);
\draw  [fill={rgb, 255:red, 0; green, 0; blue, 0 }  ,fill opacity=1 ] (500,0) circle (20);
\draw  (280,-130) rectangle (320,-170) [fill={rgb, 255:red, 0; green, 0; blue, 0 }  ,fill opacity=1 ] ;
\end{tikzpicture}
)-
w(
\begin{tikzpicture}[x=1pt,y=1pt,yscale=0.15,xscale=0.15, baseline=-3pt] 
\draw  [line width =1pt] (100,0)--(300,-150);
\draw  [line width =1pt] (300,0)--(300,-150);
\draw  [line width =1pt] (500,0)--(300,-150);
\draw  [dash pattern = on 2pt off 3 pt, line width =1pt] (100,0)--(70,160);
\draw  [dash pattern = on 2pt off 3 pt, line width =1pt] (100,0)--(130,160);
\draw  [dash pattern = on 2pt off 3 pt, line width =1pt] (300,0)--(300,150);
\draw  [dash pattern = on 2pt off 3 pt, line width =1pt] (500,0)--(500,150);
\draw  [fill={rgb, 255:red, 0; green, 0; blue, 0 }  ,fill opacity=1 ]  (100,0) circle (20);
\draw  [fill={rgb, 255:red, 0; green, 0; blue, 0 }  ,fill opacity=1 ] (300,0) circle (20);
\draw  [fill={rgb, 255:red, 0; green, 0; blue, 0 }  ,fill opacity=1 ] (500,0) circle (20);
\draw  (280,-130) rectangle (320,-170) [fill={rgb, 255:red, 0; green, 0; blue, 0 }  ,fill opacity=1 ] ;
\end{tikzpicture}
) = 
w(
\begin{tikzpicture}[x=1pt,y=1pt,yscale=0.15,xscale=0.15, baseline=-3pt] 
\draw  [line width =1pt] (100,0)..controls (200, 50)..(300,0);

\draw  [dash pattern = on 2pt off 3 pt, line width =1pt] (60,330)--(20,460);
\draw  [dash pattern = on 2pt off 3 pt, line width =1pt] (60,330)--(100,460);
\draw  [dash pattern = on 2pt off 3 pt, line width =1pt] (140,280)--(100,160);
\draw  [dash pattern = on 2pt off 3 pt, line width =1pt] (60,280)--(100,160);
\draw  [dash pattern = on 2pt off 3 pt, line width =1pt] (100,130)--(100,0);
\draw  [dash pattern = on 2pt off 3 pt, line width =1pt] (300,0)--(300,150);

\draw  [fill={rgb, 255:red, 0; green, 0; blue, 0 }  ,fill opacity=1 ]  (100,0) circle (20);
\draw  [fill={rgb, 255:red, 0; green, 0; blue, 0 }  ,fill opacity=1 ] (300,0) circle (20);
\draw  (60,300) circle (20);
\draw  (100,150) circle (20);
\end{tikzpicture}).
\]
Here, the subgraph with an internal black vertex appears twice with the opposite signs and they are cancelled. See Figure \ref{sumofweights}.
Note that the diagram $D(\ctext{Y}((p_i)_i))$ has two oriented lines with three vertices, \segmentb{}, as well as oriented lines with two vertices, \segmenta{}.
Then, we can show  $\langle \overline{I}(H), d(\ctext{Y}((p_i)_i) \rangle$ is equal to the sum of the weights of $24 = 2^2 \times 6$ graphs which are obtained by performing STU relations to $\ctext{Y}((p_i)_i)$.
The case (3) is similar, where $D(\ctext{Y}((p_i)_i))$ has an oriented line with five vertices. 

\end{proof}

\begin{figure}[htpb]
\begin{center}
     \includegraphics [width =11cm] {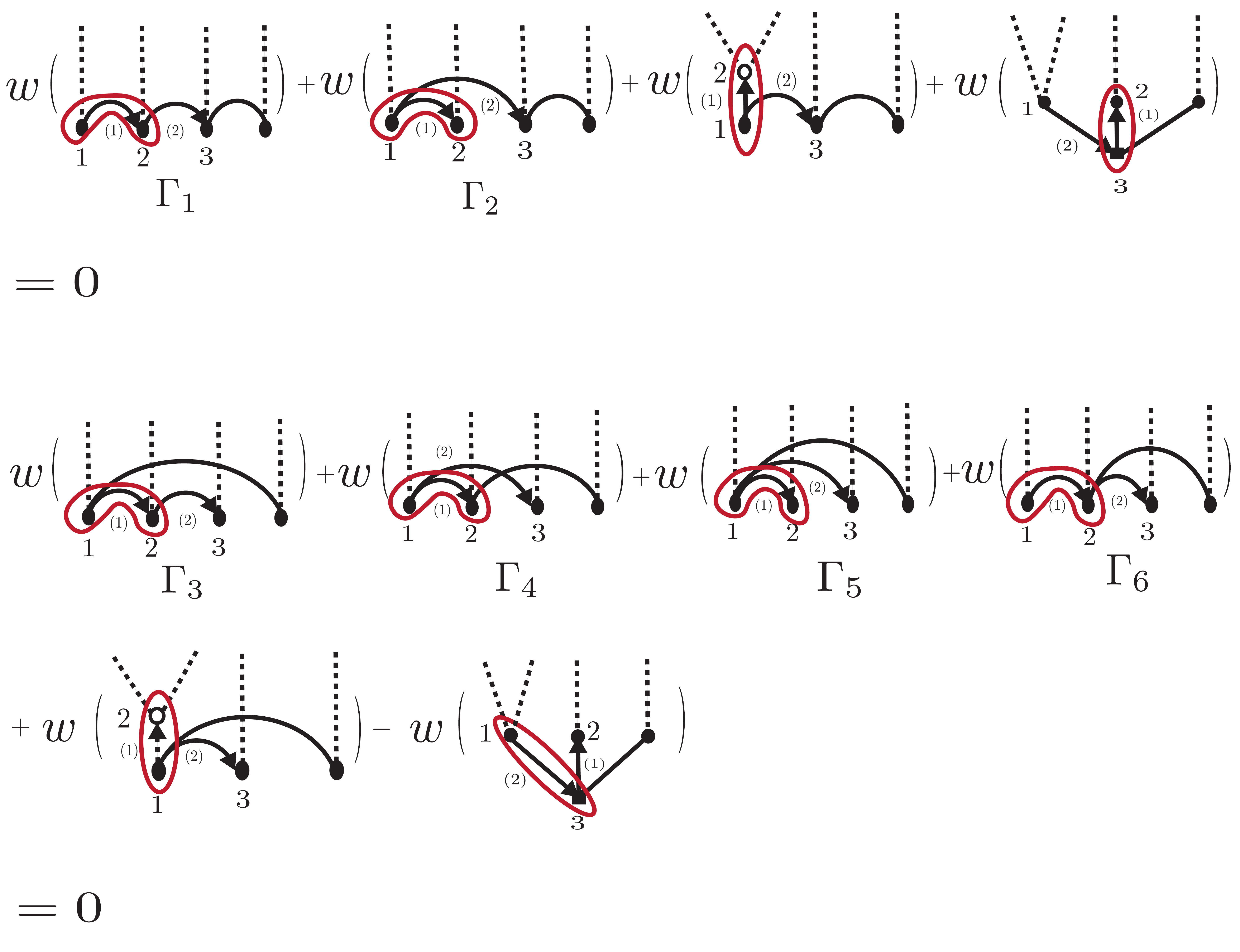}
    \caption{Relations of the weights of the six subgraphs. In each equation, the contraction of the circled edges gives the same graphs.}
    \label{sumofweights}
\end{center}
\end{figure}

We can give another proof of Theorem \ref{nontrivialityof3loopcycles} when $n-j$ is even, which is likely to be extended to the case $n-j$ is odd.
In this proof, we can reduce the number of graphs to consider. Consider the following assumption on the graph cocycle $H$. 
\begin{itemize}
\item [(A)] If $\Gamma_i$ is a plain graph, the solid part of $\Gamma_i$ is a disjoint union of broken lines. 
\end{itemize}
Instead of Proposition \ref{keyprop3}, we use the following weaker version. 
\begin{prop}
\label{keyprop4}
Let $H$ be a $3$-loop graph cocycle of order $\leq k$, of top degree which satisfies the condition (A). Then, we have
\[
\langle \overline{I}(H), d(\ctext{Y}((p_i)_i)) \rangle = \pm w(\ctext{Y}((p_i)_i)), 
\]
where $\pm$ depends only on the oriented graph $\ctext{Y}((p_i)_i)$.
\end{prop}

\begin{proof}[Another proof of Theorem \ref{nontrivialityof3loopcycles}]
Let $h$ be a top graph cocycle of $HGC_{n,j}(g=3)$ of order $\leq k$. Then, there exists a lift $H$ of $h$ to $DGC_{n,j}(g=3)$ which satisfies the condition $(A)$. This is a consequence of  \cite{Yos 2}. 
Then, the value of the pairing between the cocycle $\overline{I}(H)$ and the $3$-loop cycles $d(\ctext{Y}((p_i)_i))$ depends only on the coefficients of $\ctext{Y}((p_i)_i)$ if $k \leq 2 + \sum_i p_i$, by Proposition \ref{keyprop4}. The rest of the proof is the same as the original proof. 
\end{proof}

\begin{proof}[proof of \ref{keyprop4}]
Consider the part
\begin{tikzpicture}[x=1pt,y=1pt,yscale=0.15,xscale=0.15, baseline=-3pt] 
\draw  [-Stealth, line width =1pt, color={rgb, 255:red, 0; green, 0; blue, 255 } ] (0,0)--(900,0);
\draw  [dash pattern = on 2pt off 3 pt, line width =1pt] (100,0)--(100,150);
\draw  [dash pattern = on 2pt off 3 pt, line width =1pt] (300,0)--(300,150);
\draw  [dash pattern = on 2pt off 3 pt, line width =1pt] (500,0)--(500,150);
\draw  [dash pattern = on 2pt off 3 pt, line width =1pt] (700,0)--(700,150);
\draw  [fill={rgb, 255:red, 0; green, 0; blue, 0 }  ,fill opacity=1 ]  (100,0) circle (20);
\draw  [fill={rgb, 255:red, 0; green, 0; blue, 0 }  ,fill opacity=1 ] (300,0) circle (20);
\draw  [fill={rgb, 255:red, 0; green, 0; blue, 0 }  ,fill opacity=1 ] (500,0) circle (20);
\draw  [fill={rgb, 255:red, 0; green, 0; blue, 0 }  ,fill opacity=1 ] (700,0) circle (20);
\end{tikzpicture}.
For this part, by assumption (A), only four subgraphs
\begin{tikzpicture}[x=1pt,y=1pt,yscale=0.15,xscale=0.15, baseline=-3pt]  
\draw  [line width =1pt] (100,0)..controls (200, 50)..(300,0);
\draw  [line width =1pt] (300,0)..controls (400, 50)..(500,0);
\draw  [line width =1pt] (500,0)..controls (600, 50)..(700,0);

\draw  [dash pattern = on 2pt off 3 pt, line width =1pt] (100,0)--(100,150);
\draw  [dash pattern = on 2pt off 3 pt, line width =1pt] (300,0)--(300,150);
\draw  [dash pattern = on 2pt off 3 pt, line width =1pt] (500,0)--(500,150);
\draw  [dash pattern = on 2pt off 3 pt, line width =1pt] (700,0)--(700,150);

\draw  [fill={rgb, 255:red, 0; green, 0; blue, 0 }  ,fill opacity=1 ]  (100,0) circle (20);
\draw  [fill={rgb, 255:red, 0; green, 0; blue, 0 }  ,fill opacity=1 ] (300,0) circle (20);
\draw  [fill={rgb, 255:red, 0; green, 0; blue, 0 }  ,fill opacity=1 ] (500,0) circle (20);
\draw  [fill={rgb, 255:red, 0; green, 0; blue, 0 }  ,fill opacity=1 ] (700,0) circle (20);
\end{tikzpicture},
\begin{tikzpicture}[x=1pt,y=1pt,yscale=0.15,xscale=0.15, baseline=-3pt] 
\draw  [line width =1pt] (100,0)..controls (200, 50)..(300,0);
\draw  [line width =1pt] (100,0)..controls (300, 100)..(500,0);
\draw  [line width =1pt] (500,0)..controls (600, 50)..(700,0);

\draw  [dash pattern = on 2pt off 3 pt, line width =1pt] (100,0)--(100,150);
\draw  [dash pattern = on 2pt off 3 pt, line width =1pt] (300,0)--(300,150);
\draw  [dash pattern = on 2pt off 3 pt, line width =1pt] (500,0)--(500,150);
\draw  [dash pattern = on 2pt off 3 pt, line width =1pt] (700,0)--(700,150);

\draw  [fill={rgb, 255:red, 0; green, 0; blue, 0 }  ,fill opacity=1 ]  (100,0) circle (20);
\draw  [fill={rgb, 255:red, 0; green, 0; blue, 0 }  ,fill opacity=1 ] (300,0) circle (20);
\draw  [fill={rgb, 255:red, 0; green, 0; blue, 0 }  ,fill opacity=1 ] (500,0) circle (20);
\draw  [fill={rgb, 255:red, 0; green, 0; blue, 0 }  ,fill opacity=1 ] (700,0) circle (20);
\end{tikzpicture},
\begin{tikzpicture}[x=1pt,y=1pt,yscale=0.15,xscale=0.15, baseline=-3pt]  
\draw  [line width =1pt] (100,0)..controls (200, 50)..(300,0);
\draw  [line width =1pt] (300,0)..controls (400, 50)..(500,0);
\draw  [line width =1pt] (100,0)..controls (400, 150)..(700,0);

\draw  [dash pattern = on 2pt off 3 pt, line width =1pt] (100,0)--(100,150);
\draw  [dash pattern = on 2pt off 3 pt, line width =1pt] (300,0)--(300,150);
\draw  [dash pattern = on 2pt off 3 pt, line width =1pt] (500,0)--(500,150);
\draw  [dash pattern = on 2pt off 3 pt, line width =1pt] (700,0)--(700,150);

\draw  [fill={rgb, 255:red, 0; green, 0; blue, 0 }  ,fill opacity=1 ]  (100,0) circle (20);
\draw  [fill={rgb, 255:red, 0; green, 0; blue, 0 }  ,fill opacity=1 ] (300,0) circle (20);
\draw  [fill={rgb, 255:red, 0; green, 0; blue, 0 }  ,fill opacity=1 ] (500,0) circle (20);
\draw  [fill={rgb, 255:red, 0; green, 0; blue, 0 }  ,fill opacity=1 ] (700,0) circle (20);
\end{tikzpicture},
\begin{tikzpicture}[x=1pt,y=1pt,yscale=0.15,xscale=0.15, baseline=-3pt] 
\draw  [line width =1pt] (100,0)..controls (200, 50)..(300,0);
\draw  [line width =1pt] (100,0)..controls (300, 100)..(500,0);
\draw  [line width =1pt] (300,0)..controls (500, 100)..(700,0);

\draw  [dash pattern = on 2pt off 3 pt, line width =1pt] (100,0)--(100,150);
\draw  [dash pattern = on 2pt off 3 pt, line width =1pt] (300,0)--(300,150);
\draw  [dash pattern = on 2pt off 3 pt, line width =1pt] (500,0)--(500,150);
\draw  [dash pattern = on 2pt off 3 pt, line width =1pt] (700,0)--(700,150);

\draw  [fill={rgb, 255:red, 0; green, 0; blue, 0 }  ,fill opacity=1 ]  (100,0) circle (20);
\draw  [fill={rgb, 255:red, 0; green, 0; blue, 0 }  ,fill opacity=1 ] (300,0) circle (20);
\draw  [fill={rgb, 255:red, 0; green, 0; blue, 0 }  ,fill opacity=1 ] (500,0) circle (20);
\draw  [fill={rgb, 255:red, 0; green, 0; blue, 0 }  ,fill opacity=1 ] (700,0) circle (20);
\end{tikzpicture},
are counted. On the other hand, the sum of the weights of these four graphs coincides with 
\[
w(
\begin{tikzpicture}[x=1pt,y=1pt,yscale=0.15,xscale=0.15, baseline=-3pt] 
\draw  [line width =1pt] (100,0)..controls (200, 50)..(300,0);
\draw  [line width =1pt] (300,0)..controls (400, 50)..(500,0);

\draw  [dash pattern = on 2pt off 3 pt, line width =1pt] (100,170)--(60,300);
\draw  [dash pattern = on 2pt off 3 pt, line width =1pt] (100,170)--(140,300);
\draw  [dash pattern = on 2pt off 3 pt, line width =1pt] (100,130)--(100,0);
\draw  [dash pattern = on 2pt off 3 pt, line width =1pt] (300,0)--(300,150);
\draw  [dash pattern = on 2pt off 3 pt, line width =1pt] (500,0)--(500,150);

\draw  [fill={rgb, 255:red, 0; green, 0; blue, 0 }  ,fill opacity=1 ]  (100,0) circle (20);
\draw  [fill={rgb, 255:red, 0; green, 0; blue, 0 }  ,fill opacity=1 ] (300,0) circle (20);
\draw  [fill={rgb, 255:red, 0; green, 0; blue, 0 }  ,fill opacity=1 ] (500,0) circle (20);
\draw  (100,150) circle (20);
\end{tikzpicture}
)+
w(
\begin{tikzpicture}[x=1pt,y=1pt,yscale=0.15,xscale=0.15, baseline=-3pt] 
\draw  [line width =1pt] (100,0)..controls (200, 50)..(300,0);
\draw  [line width =1pt] (100,0)..controls (300, 100)..(500,0);

\draw  [dash pattern = on 2pt off 3 pt, line width =1pt] (100,170)--(60,300);
\draw  [dash pattern = on 2pt off 3 pt, line width =1pt] (100,170)--(140,300);
\draw  [dash pattern = on 2pt off 3 pt, line width =1pt] (100,130)--(100,0);
\draw  [dash pattern = on 2pt off 3 pt, line width =1pt] (300,0)--(300,150);
\draw  [dash pattern = on 2pt off 3 pt, line width =1pt] (500,0)--(500,150);

\draw  [fill={rgb, 255:red, 0; green, 0; blue, 0 }  ,fill opacity=1 ]  (100,0) circle (20);
\draw  [fill={rgb, 255:red, 0; green, 0; blue, 0 }  ,fill opacity=1 ] (300,0) circle (20);
\draw  [fill={rgb, 255:red, 0; green, 0; blue, 0 }  ,fill opacity=1 ] (500,0) circle (20);
\draw  (100,150) circle (20);
\end{tikzpicture})
= 
w(
\begin{tikzpicture}[x=1pt,y=1pt,yscale=0.15,xscale=0.15, baseline=-3pt] 
\draw  [line width =1pt] (100,0)..controls (200, 50)..(300,0);

\draw  [dash pattern = on 2pt off 3 pt, line width =1pt] (60,330)--(20,460);
\draw  [dash pattern = on 2pt off 3 pt, line width =1pt] (60,330)--(100,460);
\draw  [dash pattern = on 2pt off 3 pt, line width =1pt] (140,280)--(100,160);
\draw  [dash pattern = on 2pt off 3 pt, line width =1pt] (60,280)--(100,160);
\draw  [dash pattern = on 2pt off 3 pt, line width =1pt] (100,130)--(100,0);
\draw  [dash pattern = on 2pt off 3 pt, line width =1pt] (300,0)--(300,150);

\draw  [fill={rgb, 255:red, 0; green, 0; blue, 0 }  ,fill opacity=1 ]  (100,0) circle (20);
\draw  [fill={rgb, 255:red, 0; green, 0; blue, 0 }  ,fill opacity=1 ] (300,0) circle (20);
\draw  (60,300) circle (20);
\draw  (100,150) circle (20);
\end{tikzpicture}), 
\]
again by assumption (A). The rest of the proof is the same as Proposition \ref{keyprop3}.
\end{proof}


\begin{thebibliography}{9}
\bibitem [AT 1] {AT 1} G. Arone, V. Turchin, \textit{On the rational homology of high-dimensional analogues of spaces of long knots}. Geom. Topol.  18  (2014),  no. 3, 1261--1322.
\bibitem [AT 2] {AT 2} G. Arone, V. Turchin, \textit{Graph-complexes computing the rational homotopy of high dimensional analogues of spaces of long knots}. Ann. Inst. Fourier (Grenoble)  65  (2015),  no. 1, 1--62.
\bibitem [BG] {BG} R. Budney, D.Gabai, \textit{Knotted 3-balls in $S^4$.} arXiv.1912.09029.
\bibitem [Bot] {Bot} R. Bott, \textit{Configuration spaces and imbedding invariants}. Turkish J. Math.  20  (1996),  no. 1, 1--17.
\bibitem [BT] {BT} R. Bott, C. Taubes, \textit{On the self-linking of knots}. Topology and physics. J. Math. Phys. 35 (1994), no. 10, 5247--5287. 
\bibitem [BW] {BW} P. Boavidia de Brito, M. Weiss, \textit{Manifold calculus and homotopy sheaves}. Homology Homotopy Appl. 15 (2013), no. 2, 361--383.
\bibitem[CCTW]{CCTW} J. Conant, J. Costello, V. Turchin, P. Weed, \textit{Two-loop part of the rational homotopy of spaces of long embeddings}. J. Knot Theory Ramifications  23, (2014),  no. 4, 1450018, 23 pp.
\bibitem [CR] {CR}  A. S. Cattaneo, C. A. Rossi, \textit{Wilson surfaces and higher dimensional knot invariants}. Comm. Math. Phys.  256  (2005),  no. 3, 513--537.
\bibitem [FTW 1] {FTW 1} B. Fresse, V. Turchin, T.Willwacher, \textit{The rational homotopy of mapping space of $E_n$ operads}. arXiv.1703.06123.
\bibitem [FTW 2] {FTW 2} B. Fresse, V. Turchin, T. Willwacher, \textit{On the rational homotopy type of embedding spaces of manifolds in Rn}. arXiv.2008.08146.
\bibitem [GGP] {GGP} S. Garoufalidis, M. Goussarov, M. Polyak, \textit{Calculus of clovers and finite type invariants of 3-manifolds.} Geom. Topol.  5  (2001), 75--108.
\bibitem [GW] {GW} T. Goodwillie, M. Weiss, \textit{Embeddings from the point of view of immersion theory. II.} Geom. Topol.  3  (1999), 103--118.
\bibitem [GKW] {GKW} T. Goodwillie, J. Klein, M. Weiss,  \textit{Spaces of smooth embeddings, disjunction and surgery.} Surveys on surgery theory, Vol. 2, 221--284, Ann. of Math. Stud., 149, Princeton Univ. Press, Princeton, NJ,  2001.
\bibitem [Hab] {Hab} K. Habiro, \textit{Claspers and finite type invariants of links.} Geom. Topol.  4  (2000), 1--83.
\bibitem [HKS] {HKS} K. Habiro, T. Kanenobu, A. Shima,  \textit{Finite type invariants of ribbon 2-knots.} Low-dimensional topology (Funchal, 1998), 187--196, Contemp. Math., 233, Amer. Math. Soc., Providence, RI,  1999.
\bibitem [HS] {HS} K. Habiro, A. Shima, \textit{Finite type invariants of ribbon 2-knots. II.} Topology Appl. 111 (2001), no.3, 265--287.
\bibitem [MO] {MO} D. Moskovich and T. Ohtsuki, \textit{Vanishing of 3-loop Jacobi diagrams of odd degree}. J. Combin. Theory Ser. A 114 (2007), no.~5, 919--930.
\bibitem [Nak] {Nak} T. Nakatsuru, \textit{Vassiliev fuhenryou no jigen no keisan}. Master’s thesis, Tokyo Institute of Technology, January 1998.
\bibitem [Sak] {Sak} K. Sakai, \textit{Configuration space integrals for embedding spaces and the Haefliger invariant}. J. Knot Theory Ramifications 19  (2010),  no. 12, 1597--1644.
\bibitem [SW] {SW} K. Sakai, T. Watanabe, \textit{$1-$loop graphs and configuration space integral for embedding spaces.} Math. Proc. Cambridge Philos. Soc.  152  (2012),  no. 3, 497--533.
\bibitem [Wat 1] {Wat 1} T. Watanabe, \textit{Clasper-moves among ribbon 2-knots characterizing their finite type invariants.} J. Knot Theory Ramifications 15 (2006), no. 9, 1163--1199.
\bibitem [Wat 2] {Wat 2} T. Watanabe, \textit{Configuration space integral for long n-knots and the Alexander polynomial.} Algebr. Geom. Topol.  7  (2007), 47--92.
\bibitem [Wat 3] {Wat 3} T. Watanabe, \textit{On Kontsevich's characteristic classes for higher-dimensional sphere bundles. II. Higher classes.}  J. Topol.  2  (2009),  no. 3, 624--660.
\bibitem [Wat 4] {Wat 4} T. Watanabe, \textit{Some exotic nontrivial elements of rational homotopy groups of $\text{Diff}(S^4)$.} arxiv.1812.02448.
\bibitem [Wat 5] {Wat 5} T. Watanabe, \textit{Theta-graph and diffeomorphisms of some 4-manifolds.} arXiv.2005.09545
\bibitem [Wei] {Wei} M. Weiss, \textit{Embeddings from the point of view of immersion theory I.} Geom. Topol.  3  (1999), 67--101.
\bibitem [Yos 1] {Yos 1} L. Yoshioka, \textit{Cocycles of the space of long embeddings and BCR graphs with more than one loop}. To appear in Algebr. Geom. Topol., arXiv.2212.01573.
\bibitem [Yos 2] {Yos 2} L. Yoshioka, \textit{Two graph homologies and the space of long embeddings}. arXiv.2310.10896.
\bibitem [Yos 3] {Yos 3} L. Yoshioka, \textit{On hidden face contributions of configuration space integrals for long embeddings}. arXiv. 2410. 13168. 
\end{thebibliography}
\end{document}